\documentclass{amsart}
\usepackage[a4paper,centering,textwidth=155mm,textheight=195mm]{geometry}
\usepackage{amssymb}
\usepackage{xypic}
\usepackage{BOONDOX-cal}

\DeclareMathOperator{\ind}{ind}
\DeclareMathOperator{\End}{End}
\DeclareMathOperator{\Hom}{Hom}
\DeclareMathOperator{\ad}{ad}
\DeclareMathOperator{\charac}{char}
\DeclareMathOperator{\Sym}{Sym}
\DeclareMathOperator{\Skew}{Skew}
\DeclareMathOperator{\Alt}{Alt}
\DeclareMathOperator{\Trd}{Trd}
\DeclareMathOperator{\Tr}{Tr}
\DeclareMathOperator{\Trp}{Trp}

\DeclareMathOperator{\Int}{Int}

\DeclareMathOperator{\GO}{GO}
\DeclareMathOperator{\PGO}{PGO}

\DeclareMathOperator{\Orth}{O}

\DeclareMathOperator{\Str}{Str}
\DeclareMathOperator{\BHomot}{\mathbf{H}}

\DeclareMathOperator{\strf}{\str{f}} 
\DeclareMathOperator{\BAut}{\mathbf{Aut}}
\DeclareMathOperator{\BGL}{\mathbf{GL}}

\DeclareMathOperator{\BSim}{\mathbf{Sim}}
\DeclareMathOperator{\BSpin}{\mathbf{Spin}}
\DeclareMathOperator{\BOrth}{\mathbf{O}}
\DeclareMathOperator{\BGO}{\mathbf{GO}}
\DeclareMathOperator{\BPGO}{\mathbf{PGO}}
\DeclareMathOperator{\BGamma}{\mathbf{\Gamma}}
\DeclareMathOperator{\BOmega}{\mathbf{\Omega}}
\DeclareMathOperator{\BStr}{\mathbf{Str}}
\DeclareMathOperator{\homot}{\mathfrak{h}}

\DeclareMathOperator{\BZO}{\mathbf{Z}}

\newcommand{\Id}{\operatorname{Id}}
\newcommand{\qf}[1]{{\langle{#1}\rangle}} 
\newcommand{\invo}{\overline{\rule{2.5mm}{0mm}\rule{0mm}{4pt}}}
\newcommand{\Comp}{\mathcal{C}}
\newcommand{\Calg}{\mathcal{A}}

\newcommand{\orth}{\mathfrak{o}}
\renewcommand{\go}{\mathfrak{go}}
\newcommand{\Lie}{\mathfrak{L}}
\newcommand{\pgo}{\mathfrak{pgo}}

\newcommand{\spin}{\mathfrak{spin}}
\newcommand{\so}{\mathfrak{so}}
\newcommand{\clie}{\boldsymbol{\gamma}}
\newcommand{\xclie}{\boldsymbol{\omega}}
\newcommand{\str}[1]{\mathfrak{#1}} 
\newcommand{\TT}{\mathcal{T}}
\renewcommand{\tilde}{\widetilde}
\newcommand{\aqp}{\mathfrak{A}}
\newcommand{\Clqp}{\mathfrak{C}}
\newcommand{\Bmu}{\boldsymbol{\mu}}

\newcommand{\BGm}{\mathbf{G}_{\mathbf{m}}}
\newcommand{\Falg}{F_\text{\rm alg}}

\newtheorem{prop}{Proposition}[section]
\newtheorem{corol}[prop]{Corollary}
\newtheorem{thm}[prop]{Theorem}
\newtheorem{lemma}[prop]{Lemma}
\theoremstyle{definition}
\newtheorem{definition}[prop]{Definition}
\newtheorem{defs}[prop]{Definitions}

\newtheorem{remark}[prop]{Remark}

\newtheorem{examples}[prop]{Examples}

\numberwithin{equation}{section}
\begin{document}

\title{Trialitarian Triples}

\author{Demba Barry}
\address{Facult\'e des Sciences et Techniques, Universit\'e des
  Sciences, des Techniques et des Technologies de Bamako, Bamako,
  Mali}
\email{barry.demba@gmail.com}

\author{Jean-Pierre Tignol}
\address{ICTEAM Institute, UCLouvain, avenue Georges Lema\^itre
  4--6, Box L4.05.01, B-1348 Louvain-la-Neuve, Belgium}
\email{Jean-Pierre.Tignol@uclouvain.be}

\thanks{The first author would like to thank the second author and
  UCLouvain for their hospitality during several visits while the work
  for this paper was carried out. He gratefully acknowledges support
  from the Association pour la Promotion Scientifique de l'Afrique
  through a grant ``APSA Awards 2020.'' This work was partially
  supported by the Fonds de la Recherche Scientifique--FNRS under
  grant J.0159.19.} 


\maketitle
\begin{abstract}
Trialitarian triples are triples of central simple algebras of
degree~$8$ with orthogonal involution that provide a convenient
structure for the representation of trialitarian algebraic groups as
automorphism groups. This paper explicitly describes the canonical
``trialitarian'' isomorphisms between the spin groups of the algebras
with involution involved in a trialitarian triple, using a rationally
defined shift operator that cyclically permutes the algebras. The
construction relies on compositions of quadratic spaces of
dimension~$8$, which yield all the trialitarian triples of split
algebras. No restriction on the characteristic of the base field is
needed. 
\end{abstract}

\section*{Introduction}

Trialitarian triples were introduced by Knus \emph{et al.}
\cite[\S42.A]{BoI} to provide the groundwork for the study of
algebraic groups of trialitarian type $\mathsf{D}_4$ as automorphism
groups of trialitarian algebras. They consist in three central simple
algebras of degree~$8$ with orthogonal involution $(A,\sigma_A)$,
$(B,\sigma_B)$, $(C,\sigma_C)$ over a field $F$ of characteristic
different from~$2$ related by the property that the Clifford algebra
of $(A,\sigma_A)$ is isomorphic to the direct product of
$(B,\sigma_B)$ and $(C,\sigma_C)$. Trialitarian algebras over $F$ are
defined in~\cite[\S43]{BoI} as algebras with orthogonal involution of
degree~$8$ over a cubic \'etale $F$-algebra that are isomorphic after
scalar extension of $F$ to the direct product of the algebras in a
trialitarian triple.

The main goal of this work is
to elucidate the trialitarian isomorphisms that arise canonically
between the spin groups of algebras with quadratic pair involved in a
trialitarian triple and also between the groups of projective
similitudes of these algebras, see~\cite[(42.5)]{BoI}.
The basic tool is a shift operator $\partial$ of period~$3$ on
trialitarian triples, which accounts for all the trialitarian features
of the theory. The cohomological approach in \cite[(42.3)]{BoI} reveals
the existence of $\partial$ but does not provide any explicit
description. By contrast, the definition of $\partial$ in
\S\ref{subsec:dertriple} below is entirely explicit and
cohomology-free, and the definition of the trialitarian isomorphisms
follows easily in \S\ref{subsec:triso}. A noteworthy feature of our
discussion is that the restriction on the characteristic of the base
field is made obsolete thanks to the ground-breaking paper \cite{DQM}
of Dolphin--Qu\'eguiner-Mathieu, in which a canonical quadratic
pair is defined on Clifford algebras (and where the cohomological
approach to the definition of $\partial$ in arbitrary characteristic
is given~\cite[Th.~4.11]{DQM}).\footnote{Prior to~\cite{DQM}, examples
  of trialitarian triples in characteristic~$2$ were given by
Knus--Villa~\cite{KV}.} Thus, all the structures
considered in this paper are (unless explicitly mentioned) over fields
of arbitrary characteristic. 

To prepare for the discussion of trialitarian triples in
\S\ref{chap:tritri}, we found it necessary to consider first
trialitarian triples of split algebras. These triples arise from
compositions of quadratic spaces, which are studied in
\S\ref{chap:compospaces}. Compositions of quadratic spaces
provide a new perspective on the classical theory of composition
algebras by triplicating their underlying vector space. They also
demonstrate more diversity, because---in contrast with compositions
arising from composition algebras---the three quadratic spaces
involved in a composition need not be isometric; this accounts for the
interpretation in \S\ref{subsec:comp8} of the $\bmod\:2$ cohomological
invariants of $\BSpin_8$, since compositions of quadratic spaces of
dimension~$8$ are torsors under $\BSpin_8$.

Compositions of three different quadratic spaces of equal dimension
have been considered earlier, for instance in Knus' monograph
\cite[V(7.2)]{Knus}, in \cite[(35.18)]{BoI} and in the papers
\cite[5.3]{AlsGille} and \cite[\S3]{Als} by Alsaody--Gille and Alsaody
respectively. However, the shift operator $\partial$ on compositions
of quadratic spaces, briefly mentioned in~\cite[(35.18)]{BoI}, seems
to have been mostly ignored so far. By attaching to every composition
on quadratic spaces $(V_1,q_1)$, $(V_2,q_2)$, $(V_3,q_3)$ two cyclic
derivatives, which are compositions on $(V_2,q_2)$, $(V_3,q_3)$,
$(V_1,q_1)$ and on $(V_3,q_3)$, $(V_1,q_1)$, $(V_2,q_2)$ respectively,
the shift operator provides the model for the operator $\partial$ on
trialitarian triples.

Compositions of quadratic spaces of dimension~$8$ also afford a
broader view of the classical Principle of Triality for similitudes of
the underlying vector space of an octonion algebra, as discussed by
Springer--Veldkamp~\cite[\S3.2]{SpV}, and also of the local version of
this principle in characteristic~$2$ described by Elduque~\cite[\S3,
\S5]{Eld}, see Corollary~\ref{corol:PoT} and
Corollary~\ref{corol:localPoT}. Automorphisms of the compositions of 
quadratic spaces arising from composition algebras are by definition
the \emph{related triples} of isometries defined in~\cite[\S3.6]{SpV},
\cite[\S1]{Eld} 
and \cite[\S3]{AlsGille} (see Remark~\ref{rem:RT}); they are closely
related to 
\emph{autotopies} of the algebra, which form the structure group
defined for alternative algebras by Petersson~\cite{P}, see
\S\ref{subsec:compalg} and \S\ref{subsec:strgrp8}.

The first section reviews background information on Clifford groups
and their Lie algebras, notably on extended Clifford groups, which
play a central r\^ole in subsequent sections.

More detail on the contents of this work can be found in the
introduction of each section.

\tableofcontents

\section{Clifford groups and Lie algebras}
\label{chap:Cligrps}

The purpose of the first subsections of this section is to recall
succinctly the Clifford groups of even-dimensional quadratic spaces
and their twisted analogues (in the sense of Galois cohomology), which
are defined in arbitrary characteristic through central simple
algebras with quadratic pair. Most of the material is taken
from~\cite{BoI}, but we incorporate a few complements that are made
possible by the definition of canonical quadratic pairs on Clifford
algebras by Dolphin--Qu\'eguiner-Mathieu~\cite{DQM}.

A detailed discussion of the corresponding Lie algebras is given
in~\S\ref{subsec:Lieorth} and \S\ref{subsec:xclie}. For a central
simple algebra with quadratic pair $\aqp$, we emphasize the difference
between the Lie algebra $\orth(\aqp)$ of the orthogonal group and the
Lie algebra $\pgo(\aqp)$ of the group of projective similitudes, which
are canonically isomorphic when the characteristic is different
from~$2$ but contain different information in characteristic~$2$.

The last subsection provides a major tool for the definition of
homomorphisms $\Clqp(\aqp)\to\aqp'$ from the Clifford algebra of a
central simple algebra with quadratic pair $\aqp$ to a central simple
algebra with quadratic pair $\aqp'$. These homomorphisms are shown to
be uniquely determined by Lie algebra homomorphisms
$\pgo(\aqp)\to\pgo(\aqp')$; see Theorem~\ref{thm:lift}.

\subsection{Quadratic forms and quadratic pairs}
\label{subsec:quadpair}

Let $(V,q)$ be a (finite-dimensional) quadratic space over $F$. The
polar form $b\colon V\times V\to F$ is defined by
\[
  b(x,y)=q(x+y)-q(x)-q(y)\qquad\text{for $x$, $y\in V$.}
\]
We only consider quadratic spaces whose polar form $b$ is
nonsingular. This restriction entails that $\dim V$ is even if
$\charac F=2$, for $b$ is 
then an alternating form. Nonsingularity of $b$ allows us to define
the adjoint involution $\sigma_b$ on $\End V$ by the condition
\[
  b\bigl(x,a(y)\bigr) = b(\sigma_b(a)(x),y) \qquad\text{for $a\in\End
    V$ and $x$, $y\in V$.}
\]
Moreover, we may identify $V\otimes V$ with $\End V$ by mapping
$x\otimes y\in V\otimes V$ to the operator $z\mapsto x\,b(y,z)$. Under
the identification $V\otimes V=\End V$, the involution $\sigma_b$ and
the (reduced) trace $\Trd$ are given by
\[
  \sigma_b(x\otimes y) = y\otimes x \qquad\text{and}\qquad
  \Trd(x\otimes y) = b(x,y) \qquad\text{for $x$, $y\in V$,}
\]
see \cite[\S5.A]{BoI}. Moreover, for $a\in\End V$ and $x$, $y\in V$ we
have
\[
  a\circ(x\otimes y)=a(x)\otimes y \qquad\text{and}\qquad (x\otimes
  y)\circ a = x\otimes\sigma_b(a)(y).
\]
The identification $V\otimes V = \End V$, which depends on the choice
of the nonsingular polar form $b$, will be used repeatedly in the
sequel. It will be referred to as a \emph{standard identification}.

Let
\[
  \Sym(\sigma_b)=\{a\in\End V\mid \sigma_b(a)=a\}.
\]
To the quadratic form $q$ on $V$ we further associate a linear form
$\strf_q$ on $\Sym(\sigma_b)$ defined by the condition
\[
  \strf_q(x\otimes x) = q(x) \qquad\text{for $x\in V$,}
\]
see \cite[(5.11)]{BoI}. Linearizing this condition yields
$\strf_q\bigl(x\otimes y+\sigma_b(x\otimes y)\bigr) = b(x,y)$ for $x$,
$y\in V$, hence
\[
  \strf_q\bigl(a+\sigma_b(a)\bigr)=\Trd(a)\qquad\text{for $a\in\End V$.}
\]
The pair $(\sigma_b,\strf_q)$ determines the quadratic form $q$ up to
a
scalar factor by \cite[(5.11)]{BoI}, which is sufficient to define the
orthogonal group $\Orth(q)$ of isometries of $(V,q)$, as well as the
group of similitudes $\GO(q)$ and the group of projective similitudes
$\PGO(q)$, as follows:
\begin{align*}
  \Orth(q) & =\{a\in\End V\mid q\bigl(a(x)\bigr)=q(x)\text{ for all
             $x\in V$}\},\\
  & = \{a\in\End V\mid \sigma_b(a)a=1\text{ and
    $\strf_q(asa^{-1})=\strf_q(s)$ for all $s\in\Sym(\sigma_b)$}\},\\
  \GO(q) & = \{a\in\End V\mid\text{ there exists $\mu\in F^\times$
           such that $q\bigl(a(x)\bigr) = \mu\, q(x)$ for all $x\in
           V$}\},\\
  & = \{a\in\End V\mid \sigma_b(a)a\in F^\times \text{ and
    $\strf_q(asa^{-1})=\strf_q(s)$ for all $s\in\Sym(\sigma_b)$}\},\\
  \PGO(q) & = \GO(q)/F^\times.
\end{align*}
In the equivalent definitions of $\GO(q)$, the scalar $\mu$ such that
$q\bigl(a(x)\bigr)= \mu\,q(x)$ for all $x\in V$ is $\sigma_b(a)a$. It
is called the \emph{multiplier} of the similitude~$a$.

Isometries and similitudes are also defined between different
quadratic spaces: if $(V,q)$ and $(\tilde V, \tilde q)$ are quadratic
spaces over a field $F$, a \emph{similitude} $u\colon(V,q)\to(\tilde
V, \tilde q)$ is a linear bijection $V\to \tilde V$ for which there
exists a scalar $\mu\in F^\times$ such that $\tilde q\bigl(u(x)\bigr)
= \mu\, q(x)$ for all $x\in V$. The scalar $\mu$ is called the
\emph{multiplier} of the similitude, and similitudes with
multiplier~$1$ are called \emph{isometries}. Abusing notation, for
every linear bijection $u\colon V\to\tilde V$ we write
\[
  \Int(u)\colon \End V\to\End\tilde V \quad\text{for the map $a\mapsto
    u\circ a\circ u^{-1}$}.
\]
It is readily verified that for every similitude $u$ the isomorphism
$\Int(u)$ restricts to group isomorphisms
\[
  \Orth(q)\xrightarrow{\sim}\Orth(\tilde q),\qquad
  \GO(q)\xrightarrow{\sim}\GO(\tilde q),\qquad
  \PGO(q)\xrightarrow{\sim}\PGO(\tilde q).
\]
\medbreak

The groups $\Orth(q)$, $\GO(q)$ and $\PGO(q)$ are groups of
rational points of algebraic groups (i.e., smooth affine algebraic
group schemes) which are denoted respectively by $\BOrth(q)$,
$\BGO(q)$ and $\BPGO(q)$, see~\cite[\S23]{BoI}. 
As pointed out in \cite{BoI}, twisted forms (in the sense of Galois
cohomology) of these groups can be defined through a notion of
quadratic pair on central simple algebras, which is recalled next.

Let $A$ be a central simple algebra over an arbitrary field
$F$. An $F$-linear involution $\sigma$ on $A$ is said to be
\emph{orthogonal} 
(resp.\ \emph{symplectic}) if after scalar extension to a splitting
field of $A$ it is adjoint to a symmetric nonalternating (resp.\ to an
alternating) bilinear form. For any involution $\sigma$ on $A$ we
write
\[
  \Sym(\sigma)=\{a\in A\mid \sigma(a)=a\}.
\]

\begin{definition}
  A \emph{quadratic pair} $(\sigma,\strf)$ on a central simple algebra
  $A$ consists of an 
  involution $\sigma$ on $A$ and a linear map
  $\strf\colon\Sym(\sigma)\to 
  F$ subject to the following conditions:
  \begin{enumerate}
  \item[(i)]
    $\sigma$ is orthogonal if $\charac F\neq2$ and symplectic if
    $\charac F=2$;
  \item[(ii)]
    $\strf\bigl(x+\sigma(x)\bigr)=\Trd_A(x)$ for $x\in A$, where
    $\Trd_A$ 
    is the reduced trace.
  \end{enumerate}
  The map $\strf$ is called the \emph{semitrace} of the quadratic pair
  $(\sigma,\strf)$. This terminology is motivated by the observation
  that 
  when $\charac F\neq2$ every $x\in\Sym(\sigma)$ can be written as
  $x=\frac12\bigl(x+\sigma(x)\bigr)$, hence
  $\strf(x)=\frac12\Trd_A(x)$. Thus, the semitrace of a quadratic pair
  $(\sigma,\strf)$ is uniquely determined by the orthogonal involution
  $\sigma$ if $\charac F\neq2$.
\end{definition}

To simplify notation, when possible without confusion we use a single
letter to denote a central simple algebra with quadratic pair, and
write
\[
  \aqp=(A,\sigma,\strf).
\]

The twisted forms of orthogonal groups are defined as follows: for
$\aqp$ as above,
\begin{align*}
  \Orth(\aqp) & = \{a\in A\mid\sigma(a)a=1\text{ and
                    $\strf(asa^{-1})=\strf(s)$ for all
                        $s\in\Sym(\sigma)$}\},\\ 
  \GO(\aqp) & = \{a\in A\mid \sigma(a)a\in F^\times \text{ and
                  $\strf(asa^{-1})=\strf(s)$ for all
                      $s\in\Sym(\sigma)$}\},\\ 
  \PGO(\aqp) & = \GO(\aqp)/F^\times.
\end{align*}
The group of similitudes $\GO(\aqp)$ can be alternatively defined as
the group of elements $a\in A^\times$ such that $\Int(a)$ is an
automorphism of $\aqp$. Therefore, by the Skolem--Noether theorem the
group $\PGO(\aqp)$ can be identified with the group of automorphisms
of $\aqp$. The groups $\Orth(\aqp)$, $\GO(\aqp)$ and $\PGO(\aqp)$ are
groups of rational points of algebraic groups denoted respectively by
$\BOrth(\aqp)$, $\BGO(\aqp)$ and $\BPGO(\aqp)$, see~\cite[\S23]{BoI}.

For $a\in \GO(\aqp)$, the scalar $\sigma(a)a\in F^\times$ is
called the \emph{multiplier} of the similitude $a$. We write
$\mu(a)=\sigma(a)a$ and thus obtain a group homomorphism
\[
  \mu\colon\GO(\aqp)\to F^\times
\]
whose kernel is $\Orth(\aqp)$. Thus, for every quadratic space
$(V,q)$ we have by definition
\[
  \Orth(\End V,\sigma_b,\strf_q)=\Orth(q),\qquad
  \GO(\End V,\sigma_b,\strf_q) = \GO(q),\qquad
  \PGO(\End V,\sigma_b,\strf_q)=\PGO(q).
\]

The following statement is given without detailed proof
in~\cite[(12.36)]{BoI}.

\begin{prop}
  \label{prop:simiso}
  Let $(V,q)$ and $(\tilde V, \tilde q)$ be quadratic spaces over an
  arbitrary field $F$. If $u\colon(V,q)\to(\tilde V,\tilde q)$ is a
  similitude, then $\Int(u)$ is an isomorphism of algebras with
  quadratic pair
  \[
    \Int(u)\colon(\End V,\sigma_b,\strf_q) \xrightarrow{\sim} (\End
    \tilde V, \sigma_{\tilde b}, \strf_{\tilde q}).
  \]
  Conversely, every isomorphism $(\End V,\sigma_b,\strf_q)
  \xrightarrow{\sim} (\End\tilde V, \sigma_{\tilde b}, \strf_{\tilde
    q})$ has the form $\Int(u)$ for some similitude
  $u\colon(V,q)\to(\tilde V,\tilde q)$ uniquely determined up to a
  scalar factor.
\end{prop}

\begin{proof}
  Observe that for every linear bijection $u\colon V\to\tilde V$ there
  exists a map $\hat u\colon V\to \tilde V$ such that
  \[
    \tilde b(\hat u(x),\tilde y) = b\bigl(x,u^{-1}(\tilde y)\bigr)
    \qquad\text{for all $x\in V$ and $\tilde y\in\tilde V$,}
  \]
  because the polar forms $b$ and $\tilde b$ are nonsingular. Under
  the standard identifications $\End V=V\otimes V$ and $\End\tilde
  V=\tilde V\otimes\tilde V$ afforded by $b$ and $\tilde b$, we 
  have
  \[
  \Int(u)(x\otimes y)=u(x)\otimes \hat u(y) \qquad\text{for all $x$,
    $y\in V$.}
  \]
  If $u$ is a similitude with multiplier $\mu$, then $\hat
  u=\mu^{-1}u$, hence $\Int(u)\circ\sigma_b=\sigma_{\tilde
    b}\circ\Int(u)$ and 
  \[
    \strf_{\tilde q}\bigl(\Int(u)(x\otimes x)\bigr) = \mu^{-1}\,\tilde
    q\bigl(u(x)\bigr) = q(x)=\strf_q(x\otimes x)
    \qquad\text{for all $x\in V$}.
  \]
  Since $\Sym(\sigma_b)$ is spanned by elements of the form $x\otimes
  x$, 
  it follows that $\Int(u)$ is an isomorphism of algebras with
  quadratic pair.

  For the converse, note that the Skolem--Noether theorem shows that
  every $F$-algebra isomorphism $\End V\xrightarrow{\sim}\End\tilde V$
  has the form $\Int(u)$ for some linear bijection $u\colon V\to\tilde
  V$. If $\Int(u)$ is an isomorphism of algebras with quadratic pair,
  then $\Int(u)(x\otimes x)\in\Sym(\sigma_{\tilde b})$ for every $x\in
  V$, hence $\hat u = \mu^{-1}u$ for some $\mu\in F^\times$. Since
  $\strf_{\tilde q}\bigl(\Int(u)(x\otimes x)\bigr) = \strf_q(x\otimes
  x)$ for all $x\in V$, it follows that $\tilde q\bigl(u(x)\bigr) =
  \mu\,q(x)$ for all $x\in V$, hence $u$ is a similitude.

  To complete the proof, suppose that $u$, $u'\colon(V,q)\to(\tilde V,
  \tilde q)$ are similitudes such that $\Int(u)=\Int(u')$. Then
  $\Int(u^{-1}u') = \Id_V$, hence $u^{-1}u'$ lies in the center of
  $\End V$, which is $F$. Therefore, $u$ and $u'$ differ by a scalar
  factor. 
\end{proof}

\subsection{Clifford algebras}
\label{subsec:ClAlg}

For any quadratic space $(V,q)$ over $F$ we let $C(V,q)$ denote the
Clifford algebra of $(V,q)$ and $C_0(V,q)$ its even Clifford
algebra. We will only consider even-dimensional quadratic spaces; if
$\dim V=2m$, then the algebra $C(V,q)$ is central simple of
degree~$2^m$ and $C_0(V,q)$ is semisimple with center a quadratic
\'etale $F$-algebra $Z$ given by the discriminant or Arf invariant of
$q$, see~\cite[Ch.~9]{Sch}. In most cases considered through this
text, the algebra $Z$ is split, i.e., $Z\simeq F\times F$. We may then
define a polarization of $(V,q)$ as follows:

\begin{definition}
  \label{defn:orientationqs}
  If $(V,q)$ is an even-dimensional quadratic space with trivial
  discriminant or Arf invariant, a \emph{polarization} of $(V,q)$ is
  a 
  designation of the primitive central idempotents of $C_0(V,q)$ as
  $z_+$ and $z_-$. Given a polarization of $(V,q)$, we let
  $C_+(V,q)=C_0(V,q)z_+$ and $C_-(V,q)=C_0(V,q)z_-$, so
  \[
    C_0(V,q)=C_+(V,q)\times C_-(V,q).
  \]
  Each even-dimensional quadratic space of trivial discriminant or Arf
  invariant thus has two possible polarizations.
\end{definition}

The algebra $C(V,q)$ carries an involution $\tau$ such
that $\tau(x)=x$ for all $x\in V$. This involution
preserves $C_0(V,q)$ and restricts to an involution
$\tau_0$ on $C_0(V,q)$. The type of the involutions
$\tau$ and $\tau_0$ is determined in~\cite[(8.4)]{BoI} as follows:
\begin{itemize}
\item
  If $\dim V\equiv2\bmod4$ the involution $\tau_0$ does
  not leave $Z$ fixed; we will not need to consider this case.
\item
  If $\dim V\equiv4\bmod8$, then the involutions $\tau$
  and $\tau_0$ are symplectic. When $Z\simeq F\times F$, this
  means that $\tau_0$ restricts to symplectic involutions
  on each of the simple components of $C_0(V,q)$.
\item
  If $\dim V\equiv0\bmod8$ and $\charac F\neq2$, then the involutions
  $\tau$ and $\tau_0$ are orthogonal.
\item
  If $\dim V\equiv0\bmod8$ and $\charac F=2$, then the involutions
  $\tau$ and $\tau_0$ are symplectic.
\end{itemize}

Following Dolphin--Qu\'eguiner-Mathieu \cite[Prop.~6.2]{DQM}, 
a canonical quadratic pair $(\tau,\str{g})$ can be
defined on $C(V,q)$
when\footnote{Dolphin--Qu\'eguiner-Mathieu only assume $\dim V$ even,
  $\dim V\geq6$, but they restrict to $\charac F=2$.} $\dim
V\equiv0\bmod8$ by associating to $\tau$ the following semitrace:
\[
  \str{g}(s) = \Trd_{C(V,q)}(e e' s)\in F \qquad\text{for
    $s\in\Sym(\tau)$},
\]
where $e$, $e'\in V$ are arbitrary vectors such that $b(e,e')=1$.  If
$\charac F\neq2$, then for any such vectors $e$, $e'$ and for every
$s\in\Sym(\tau)$ we have
\[
  \Trd_{C(V,q)}(ee's)=\Trd_{C(V,q)}\bigl(\tau(ee's)\bigr)
  = \Trd_{C(V,q)}(se'e)=\Trd_{C(V,q)}(e'es).
\]
Therefore,
\[
  \Trd_{C(V,q)}(ee's)={\textstyle\frac12}\Trd_{C(V,q)}\bigl((ee'+e'e)s\bigr)
  = {\textstyle\frac12}\Trd_{C(V,q)}(s),
\]
as expected.

Likewise, Dolphin--Qu\'eguiner-Mathieu show in \cite[Prop.~3.6]{DQM}
that a canonical quadratic pair $(\tau_0,\str{g}_0)$
can be defined on $C_0(V,q)$ when $\dim V\equiv0\bmod8$ by associating
to $\tau_0$ the following semitrace:
\[
  \str{g}_0(s)=\Trd_{C_0(V,q)}(ee's)\in Z \qquad\text{for
    $s\in\Sym(\tau_0)$},
\]
where $e$, $e'\in V$ are arbitrary vectors such that $b(e,e')=1$. If
$Z\simeq F\times F$, then $C_0(V,q)\simeq C_+(V,q)\times C_-(V,q)$ for
some central simple $F$-algebras $C_+(V,q)$, $C_-(V,q)$, and the
quadratic pair $(\tau_0,\str{g}_0)$ defined above is
a pair of quadratic pairs $(\tau_+,\str{g}_+)$ on $C_+(V,q)$ and
$(\tau_-,\str{g}_-)$ on $C_-(V,q)$.

Every similitude of quadratic spaces $u\colon (V,q)\to(\tilde V,\tilde
q)$ with multiplier $\mu$ defines an $F$-isomorphism $C_0(u)\colon
C_0(V,q)\xrightarrow{\sim} C_0(\tilde V,\tilde q)$ such that
\[
  C_0(u)(x\cdot y) = \mu^{-1} u(x)\cdot u(y) \qquad\text{for $x$,
    $y\in V$.}
\]
It is clear from the definition that $C_0(u)$ preserves the canonical
involutions $\tau_0$ and $\tilde\tau_0$ on $C_0(V,q)$ and $C_0(\tilde
V,\tilde q)$. If $\dim V\equiv0\bmod8$, then $C_0(u)$ also preserves
the semitraces $\str{g}_0$ and $\tilde{\str{g}}_0$. To see this,
observe that the images
$u(e)$, $u(e')$ of vectors $e$, $e'\in V$ such that $b(e,e')=1$
satisfy $\tilde b\bigl(u(e),u(e')\bigr)=\mu$. We may therefore use
$\mu^{-1}u(e)$ and $u(e')$ to compute the semitrace
$\tilde{\str{g}}_0$: for $s\in\Sym(\tau_0)$,
\[
  \tilde{\str{g}}_0\bigl(C_0(u)(s)\bigr) = \Trd_{C_0(\tilde V,\tilde
    q)}\bigl(\mu^{-1}u(e)u(e')C_0(u)(s)\bigr).
\]
Now, $\mu^{-1}u(e)u(e')=C_0(u)(ee')$, hence by substituting in the
preceding equation and using the property that algebra isomorphisms
preserve reduced traces, we obtain
\[
  \tilde{\str{g}}_0\bigl(C_0(u)(s)\bigr)
= \Trd_{C_0(\tilde V,\tilde
    q)}\bigl(C_0(u)(ee's)\bigr) =
  C_0(u)\bigl(\Trd_{C_0(V,q)}(ee's)\bigr) = 
  C_0(u)\bigl(\str{g}_0(s)\bigr).
\]
Thus, $C_0(u)$ is an isomorphism of algebras with involution
\[
  C_0(u)\colon(C_0(V,q),\tau_0) \xrightarrow{\sim} (C_0(\tilde V,
  \tilde q), \tilde\tau_0)
\]
and an isomorphism of algebras with quadratic pair if $\dim V=\dim
\tilde V\equiv0\bmod8$
\[
  C_0(u)\colon(C_0(V,q),\tau_0,\str{g}_0) \xrightarrow{\sim}
  (C_0(\tilde V, \tilde q), \tilde\tau_0, \tilde{\str{g}}_0).
\]

Among auto-similitudes $u\in\GO(q)$ we may distinguish proper
similitudes by considering the restriction of $C_0(u)$ to the center
$Z$ of $C_0(V,q)$: the similitude $u$ is said
to be \emph{proper} if $C_0(u)$ fixes $Z$ and \emph{improper} if
$C_0(u)$ restricts to the nontrivial $F$-automorphism of $Z$, see
\cite[\S13.A]{BoI}. The proper similitudes form a subgroup $\GO^+(q)$
of index~$2$ in $\GO(q)$, and we let
\[
  \Orth^+(q)=\Orth(q)\cap\GO^+(q),\qquad
  \PGO^+(q)=\GO^+(q)/F^\times.
\]

\subsubsection*{Twisted forms}
Following ideas of Jacobson and Tits, an analogue of the even Clifford
algebra for a central simple algebra with quadratic pair
$\aqp=(A,\sigma,\strf)$ of even degree is defined in~\cite[\S8.B]{BoI}. The
Clifford algebra $C(\aqp)$ is obtained by a functorial
construction such that for every quadratic space $(V,q)$ of even
dimension, the identification $\End V=V\otimes V$ set up in
\S\ref{subsec:quadpair} yields an identification
\[
  C(\End V,\sigma_b,\strf_q)=C_0(V,q).
\]
This property implies that $C(\aqp)$ is a semisimple algebra
with center a quadratic \'etale $F$-algebra given by the discriminant
of the quadratic pair $(\sigma,\strf)$.

\begin{definition}
  \label{defn:orientationaqp}
  If the discriminant of $(\sigma,\strf)$ is trivial, a
  \emph{polarization} of $\aqp$ is a designation of the primitive
  central idempotents of $C(\aqp)$ as $z_+$ and $z_-$. A polarization
  induces the labeling of the simple components of $C(\aqp)$ as
  $C_+(\aqp)=C(\aqp)z_+$ and $C_-(\aqp)=C(\aqp)z_-$, so
  \[
    C(\aqp)=C_+(\aqp)\times C_-(\aqp).
  \]
\end{definition}

The algebra $C(\aqp)$ comes equipped with a canonical linear map
\[
  c\colon A\to C(\aqp)
\]
whose image generates $C(\aqp)$ as an $F$-algebra. In the
split 
case $A=\End V$, the map $c$ is given by multiplication in
$C(V,q)$:
\[
  c\colon V\otimes V\to C_0(V,q),\qquad x\otimes y\mapsto x\cdot y.
\]

The algebra $C(\aqp)$ carries a canonical involution
$\underline\sigma$ characterized by the condition that
$\underline\sigma\bigl(c(a)\bigr)= c\bigl(\sigma(a)\bigr)$ for $a\in
A$. If $\deg A\equiv0\bmod8$, Dolphin--Qu\'eguiner-Mathieu show that a
canonical quadratic pair $(\underline\sigma,\underline \strf)$ is defined
on $C(\aqp)$ by associating to $\underline\sigma$ the following
semitrace: 
\[
  \underline \strf(s)=\Trd_{C(\aqp)}(c(a)s) \qquad\text{for
    $s\in\Sym(\underline\sigma)$,}
\]
where $a\in A$ is any element such that $\Trd_A(a)=1$,
see~\cite[Def.~3.3]{DQM}. These constructions are compatible with the
corresponding definitions in the split case, in the sense that for
every even-dimensional quadratic space $(V,q)$ the standard
identification 
$\End V=V\otimes V$ of \S\ref{subsec:quadpair} yields identifications
of algebras with involution or quadratic pair:
\begin{align*}
  (C(\End V, \sigma_b,\strf_q),\underline\sigma) & = (C_0(V,q),\tau_0)
     &\text{if $\dim V\equiv 0\bmod4$},\\
  (C(\End V,\sigma_b,\strf_q),\underline\sigma,\underline \strf)
                                             &
                                               =(C_0(V,q),\tau_0,\str{g}_0)
      &\text{if $\dim V\equiv 0\bmod8$}.
\end{align*}

By functoriality of the Clifford algebra construction, every
isomorphism of algebras with quadratic pair $\varphi\colon\aqp
\xrightarrow{\sim} \tilde\aqp$ induces an isomorphism of algebras with
involution or with quadratic pair
\[
  C(\varphi)\colon (C(\aqp),\underline\sigma) \xrightarrow{\sim}
  (C(\tilde\aqp), \underline{\tilde\sigma}) \quad\text{or}\quad
  (C(\aqp),\underline\sigma,\underline\strf) \xrightarrow{\sim}
  (C(\tilde\aqp), \underline{\tilde\sigma},\underline{\tilde\strf})
\]
such that
\[
  C(\varphi)\bigl(c(a)\bigr) = c\bigl(\varphi(a)\bigr)
  \qquad\text{for $a\in A$.}
\]

As in the split case, we may distinguish between proper and improper
similitudes: every similitude $u\in\GO(\aqp)$ induces an
$F$-automorphism $\Int(u)$ of $\aqp$, hence an
$F$-automor\-phism $C\bigl(\Int(u)\bigr)$ of $C(\aqp)$.
The similitude $u$ is said to be \emph{proper} if
$C\bigl(\Int(u)\bigr)$ leaves the 
center of $C(\aqp)$ elementwise fixed; otherwise it is said
to be \emph{improper}. This definition agrees with the previous
definition of proper similitude in the case where $\aqp=(\End V,
\sigma_b, \strf_q)$ for a quadratic space $(V,q)$, because
$C\bigl(\Int(u)\bigr)=C_0(u)$ for every similitude $u\in\GO(q)$,
see~\cite[(13.1)]{BoI}.

Proper similitudes form a subgroup
$\GO^+(\aqp)$ of index~$1$ or $2$ in $\GO(\aqp)$, and
we let 
\[
  \Orth^+(\aqp)=\Orth(\aqp)\cap\GO^+(\aqp),
  \qquad  
  \PGO^+(\aqp)=\GO^+(\aqp)/F^\times.
\]
These groups are groups of rational points of algebraic groups
$\BOrth^+(\aqp)$, $\BGO^+(\aqp)$ and $\BPGO^+(\aqp)$, which are the
connected components of the identity in $\BOrth(\aqp)$, $\BGO(\aqp)$
and $\BPGO(\aqp)$, see~\cite[\S23.B]{BoI}.

\subsection{Clifford groups}
\label{subsec:ClGrp}

Let $(V,q)$ be a quadratic space of even dimension. The multiplicative
group of $C_0(V,q)$ acts on $C(V,q)$ by conjugation. The \emph{special
  Clifford group} $\BGamma^+(q)$ is defined in~\cite[p.~349]{BoI} as
the normalizer of the subspace $V$. Thus, for every commutative
$F$-algebra $R$, letting $V_R=V\otimes_FR$,
\[
  \BGamma^+(q)(R)=\{\xi\in C_0(V,q)_R^\times\mid \xi\cdot
  V_R\cdot\xi^{-1} = V_R\}.
\]
For $\xi\in\BGamma^+(q)(R)$, the map $\Int(\xi)\rvert_{V_R}\colon
V_R\to V_R$ is a proper isometry. Mapping $\xi$ to
$\Int(\xi)\rvert_{V_R}$ defines a morphism of algebraic groups $\chi$
known as the \emph{vector representation}, which fits in an exact
sequence
\begin{equation}
  \label{eq:exseqBGamma}
  1\to\BGm \to\BGamma^+(q)\xrightarrow{\chi}\BOrth^+(q)\to 1,
\end{equation}
where $\BGm$ is the multiplicative group, see~\cite[p.~349]{BoI}.

Mapping $\xi\in\BGamma^+(q)(R)$ to $\tau_0(\xi)\xi$ defines a morphism
\[
  \underline\mu\colon\BGamma^+(q)\to \BGm.
\]
Its kernel is the Spin group $\BSpin(q)$. It is an algebraic group to
which we may restrict the vector representation to obtain the
following exact sequence:
\[
  1\to\Bmu_2\to \BSpin(q) \xrightarrow{\chi} \BOrth^+(q)\to 1,
\]
where $\Bmu_2$ is the algebraic group scheme defined by
\[
  \Bmu_2(R)=\{\rho\in R\mid \rho^2=1\}
  \quad\text{for every commutative $F$-algebra $R$}.
\]
Note that $\Bmu_2$ is not smooth if $\charac F=2$.

\subsubsection*{Extended Clifford groups}

Let $Z$ be the center of $C_0(V,q)$.
Henceforth, we assume $\dim V\equiv 0\bmod4$, so the canonical
involution $\tau_0$ acts trivially on $Z$.

Let $\BSim(\tau_0)$ be the group of similitudes of
$(C_0(V,q),\tau_0)$, whose rational points over any commutative
$F$-algebra $R$ is
\[
  \BSim(\tau_0)(R)=\{\xi\in C_0(V,q)_R^\times\mid \tau_0(\xi)\xi \in
  Z_R^\times\}.
\]
The multiplier map $\xi\mapsto \tau_0(\xi)\xi$ is a morphism
\[
  \underline\mu\colon \BSim(\tau_0)\to R_{Z/F}(\BGm),
\]
where $R_{Z/F}(\BGm)$ is the corestriction (or Weil's
\emph{restriction of scalars}) of the multiplicative group. Mapping
$x\in C(V,q)_R$ and $\xi\in\BSim(\tau_0)(R)$ to $\tau_0(\xi)x\xi$
defines an action of $\BSim(\tau_0)$ on $C(V,q)$ (on the right).
The \emph{extended Clifford group} $\BOmega(q)$ is
defined\footnote{For a more general definition covering the case where
  $\dim V\equiv2\bmod4$, see~\cite[\S13.B]{BoI}.} as the normalizer of
$V$. Thus, for every commutative $F$-algebra $R$,
\[
  \BOmega(q)(R)=\{\xi\in\BSim(\tau_0)(R) \mid \tau_0(\xi)\cdot
  V_R\cdot \xi = V_R\}.
\]
We proceed to show that $\BGamma^+(q)$ is a subgroup of $\BOmega(q)$
by reformulating the condition that $\tau_0(\xi)\cdot V_R\cdot\xi =
V_R$. 

Let $\iota\colon Z\to Z$ denote the nontrivial $F$-automorphism of
$Z$. Note that $xz=\iota(z)x$ for all $x\in V$ and $z\in Z$.

\begin{lemma}
  \label{lem:eqcondXClifgrp}
  Let $R$ be a commutative $F$-algebra and
  let $\xi\in\BSim(\tau_0)(R)$ and $u\in\BGL(V)(R)$. The following are
  equivalent:
  \begin{enumerate}
  \item[(a)]
    $\tau_0(\xi)x\xi=\sigma_b(u)(x)$ for all $x\in V_R$;
  \item[(b)]
    $u(y)=\iota\bigl(\underline\mu(\xi)\bigr) \xi y\xi^{-1}$ for all
    $y\in V_R$.
  \end{enumerate}
  When these conditions hold, then $u\in\BGO^+(q)(R)$,
  $C_0(u)=\Int(\xi)$ and 
  $\mu(u)=N_{Z/F}\bigl(\underline\mu(\xi)\bigr)$. 
\end{lemma}

\begin{proof}
  Suppose~(a) holds. Squaring each side of the equation yields
  \[
    \tau_0(\xi)x\underline\mu(\xi)x\xi = q\bigl(\sigma_b(u)(x)\bigr)
    \qquad\text{for all $x\in V_R$}
  \]
  hence, since $\tau_0(\xi)x\underline\mu(\xi)x\xi = \tau_0(\xi)x^2\xi
  \iota\bigl(\underline\mu(\xi)\bigr) =
  q(x)N_{Z/F}\bigl(\underline\mu(\xi)\bigr)$, 
  \[
    q\bigl(\sigma_b(u)(x)\bigr) =
    N_{Z/F}\bigl(\underline\mu(\xi)\bigr) q(x)\qquad\text{for all
      $x\in V_R$.}
  \]
  It follows that $\sigma_b(u)\in\BGO(q)(R)$ and
  $\mu\bigl(\sigma_b(u)\bigr) =
  N_{Z/F}\bigl(\underline\mu(\xi)\bigr)$, hence also $u\in\BGO(q)(R)$
  and $\mu(u)=N_{Z/F}\bigl(\underline\mu(\xi)\bigr)$. 

  On the other hand, multiplying each side of~(a) on the left by $\xi$
  and on the right by $\xi^{-1}$ yields
  \[
    \underline\mu(\xi)x=\xi\sigma_b(u)(x)\xi^{-1}.
  \]
  Letting $y=\sigma_b(u)(x)$, we have $u(y)=\mu(u)x$. By substituting
  in the last displayed equation we obtain
  \[
    \mu(u)^{-1}\underline\mu(\xi)u(y)=\xi y\xi^{-1}.
  \]
  As $\mu(u)=N_{Z/F}\bigl(\underline\mu(\xi)\bigr)$, condition~(b)
  follows.

  Now, suppose (b) holds. Squaring each side of the equation yields
  \[
    q\bigl(u(y)\bigr)=\iota\bigl(\underline\mu(\xi)\bigr)
    \underline\mu(\xi) \xi y^2\xi^{-1} =
    N_{Z/F}\bigl(\underline\mu(\xi)\bigr) q(y) \qquad\text{for all
      $y\in V_R$,}
  \]
  hence $u\in\BGO(q)(R)$ and
  $\mu(u)=N_{Z/F}\bigl(\underline\mu(\xi)\bigr)$. On the other hand,
  multiplying each side of~(b) by $\tau_0(\xi)$ on the left and by
  $\xi$ on the right yields
  \[
    \tau_0(\xi)u(y)\xi = N_{Z/F}\bigl(\underline\mu(\xi)\bigr) y =
    \mu(u)y \qquad\text{for all $y\in V_R$.}
  \]
  Letting $x=u(y)$, we have $\sigma_b(u)(x)=\mu(u) y$, hence by
  substituting in the last displayed equation we obtain~(a).

  To complete the proof, we compute $C_0(u)$ using~(b). For $x$, $y\in
  V_R$, taking into account that
  $\mu(u)=N_{Z/F}\bigl(\underline\mu(\xi)\bigr)$ we find
  \[
    C_0(u)(xy) = \mu(u)^{-1}u(x)u(y) =
    \mu(u)^{-1}\iota\bigl(\underline\mu(\xi)\bigr) \xi x\xi^{-1}
    \iota\bigl(\underline\mu(\xi)\bigr)\xi y\xi^{-1} = \xi xy
    \xi^{-1}.
  \]
  Since $\xi\in C_0(V,q)_R$, it follows that $C_0(u)$ restricts to the
  identity on $Z_R$, hence $u$ is a proper similitude.
\end{proof}

For $\xi\in\BOmega(q)(R)$, the map $x\mapsto \tau_0(\xi)x\xi$ is an
invertible 
linear operator on $V_R$. If $u\in\BGL(V)(R)$ is the image of this
operator under $\sigma_b$, then condition~(a) of
Lemma~\ref{lem:eqcondXClifgrp} holds 
for this $u$. We write $u=\chi_0(\xi)$, so
$\chi_0(\xi)\in\BGO^+(q)(R)$ is 
equivalently defined by any of the two equations
\begin{equation}
  \label{eq:defchi0}
  \tau_0(\xi)x\xi=\sigma_b\bigl(\chi_0(\xi)\bigr)(x)
  \quad\text{and}\quad
  \chi_0(\xi)(x)=\iota\bigl(\underline\mu(\xi)\bigr)\xi x \xi^{-1}
  \quad\text{for all $x\in V_R$.}
\end{equation}
The map $\chi_0$ is a morphism
\[
  \chi_0\colon\BOmega(q)\to\BGO^+(q).
\]
Lemma~\ref{lem:eqcondXClifgrp} yields
\begin{equation}
  \label{eq:propchi0}
  \Int\rvert_{\BOmega(q)} = C_0\circ\chi_0\in\BAut\bigl(C_0(V,q)\bigr)
  \qquad\text{and}\qquad 
  N_{Z/F}\circ\underline\mu = \mu\circ\chi_0\colon\BOmega(q)\to\BGm.
\end{equation}

\begin{prop}
  \label{prop:ClifinXClif}
  The special Clifford group $\BGamma^+(q)$ is a subgroup of
  $\BOmega(q)$. More precisely,
  \[
    \BGamma^+(q)=\underline\mu^{-1}(\BGm)\subset\BOmega(q).
  \]
  Moreover,
  $\chi_0\rvert_{\BGamma^+(q)}=
  \underline\mu\cdot\chi\colon\BGamma^+(q)\to\BOrth^+(q)$, 
  hence $\chi_0$ and $\chi$ coincide on $\BSpin(q)$. 
\end{prop}

\begin{proof}
  As pointed out in the definition of $\BSpin(q)$ above, for every
  commutative $F$-algebra $R$ the
  multiplier 
  $\underline\mu(\xi)$ of any $\xi\in\BGamma^+(q)(R)$ lies in
  $R^\times$. Therefore, Lemma~\ref{lem:eqcondXClifgrp} shows that
  $\xi V_R\xi^{-1}=V_R$ implies $\tau_0(\xi)V_R\xi=V_R$, hence
  $\BGamma^+(q)(R)\subset\BOmega(q)(R)$. Conversely, if 
  $\xi\in\BOmega(q)(R)$ and $\underline\mu(\xi)\in R^\times$,
  Lemma~\ref{lem:eqcondXClifgrp} shows that $\tau_0(\xi)V_R\xi=V_R$
  implies $\xi V_R\xi^{-1}=V_R$, hence
  $\xi\in\BGamma^+(q)(R)$. Therefore $\BGamma^+(q)(R)$ is the subgroup 
  of elements in $\BOmega(q)(R)$ whose multiplier lies in $R^\times$. 
  
  Moreover, for $\xi\in\BGamma^+(q)(R)$ we have
  $\chi(\xi)(x)=\xi x\xi^{-1}$ and
  $\chi_0(\xi)(x)=\underline\mu(\xi)\xi x 
  \xi^{-1}$ for all $x\in V_R$, hence
  $\chi_0(\xi)=\underline\mu(\xi)\chi(\xi)$. 
\end{proof}

\subsubsection*{Twisted forms}
Twisted forms of $\BGamma^+(q)$ and $\BOmega(q)$ are defined in
\cite[\S13.B and \S23.B]{BoI} by using a Clifford bimodule $B(\aqp)$
associated to any central simple algebra of even degree with quadratic
pair $\aqp=(A,\sigma,\strf)$. This bimodule is defined in such a way
that for 
every even-dimensional quadratic space $(V,q)$ the standard
identification $\End V=V\otimes V$ yields
\[
  B(\End V,\sigma_b,\strf_q) = V\otimes C_1(V,q),
\]
where $C_1(V,q)$ is the odd part of the Clifford algebra $C(V,q)$. The
left action $*$ and the right action $\cdot$ of $C(\aqp)$ on
$B(\aqp)$ are given in the split case by
\[
  \xi*(x\otimes\eta) = x\otimes(\xi\eta) \qquad\text{and}\qquad
  (x\otimes\eta)\cdot\xi=x\otimes(\eta\xi)
\]
for $\xi\in C_0(V,q)$, $\eta\in C_1(V,q)$ and $x\in V$.
The bimodule $B(\aqp)$ also carries a left $A$-module structure
and a canonical left $A$-module homomorphism $b\colon A\to
B(\aqp)$ (for which 
we use the exponential notation) given in the split case by
\[
  a(x\otimes\eta)=a(x)\otimes\eta\quad\text{and}\quad
  (x\otimes y)^b=x\otimes y\in V\otimes C_1(V,q)
\]
for $a\in\End V$, $x$, $y\in V$ and $\eta\in C_1(V,q)$.
\medbreak

The multiplicative group of $C(\aqp)$ acts on $B(\aqp)$ on the right
as follows: $\eta\mapsto \xi^{-1}*\eta\cdot\xi$ for $\xi\in
C(\aqp)^\times$ and $\eta\in B(\aqp)$. 
The \emph{Clifford group} $\BGamma(\aqp)$ is the normalizer of the
subspace $A^b\subset B(\aqp)$, hence for every commutative $F$-algebra
$R$
\[
  \BGamma(\aqp)(R)=\{\xi\in C(\aqp)_R^\times\mid
  \xi^{-1}*A_R^b\cdot\xi=A_R^b\}.
\]
On the same model, when $\deg A\equiv0\bmod4$, we define\footnote{An
  alternative definition, which also covers the case where $\deg
  A\equiv2\bmod4$, is given in~\cite[\S23.B]{BoI}.} the \emph{extended
  Clifford group} $\BOmega(\aqp)$ as the normalizer of $A^b$ under the
action on $B(\aqp)$ of the group of similitudes of the canonical
involution $\underline\sigma$ by
\[
\xi\mapsto(\eta\mapsto\underline\sigma(\xi)*\eta\cdot\xi).
\]
Thus, letting $Z$ denote the center of $C(\aqp)$, 
\[
\BOmega(\aqp)(R)=\{\xi\in C(\aqp)^\times_R \mid
\underline\sigma(\xi)\xi\in Z^\times_R \text{ and }
\underline\sigma(\xi)*A^b\cdot\xi = A^b\}
\]
for every commutative $F$-algebra $R$.
Let $\underline\mu$ denote the multiplier map
\[
  \underline\mu\colon\BOmega(\aqp)\to R_{Z/F}(\BGm),\qquad
  \xi\mapsto \underline\sigma(\xi)\xi
\]
and define morphisms
\[
  \chi\colon\BGamma(\aqp)\to\BOrth^+(\aqp)
  \qquad\text{and}\qquad
  \chi_0\colon\BOmega(\aqp)\to\BGO^+(\aqp)
\]
by
\[
  \xi^{-1}*1^b\cdot\xi=\chi(\xi)^b \qquad\text{and}\qquad
  \underline\sigma(\xi)*1^b\cdot\xi=\chi_0(\xi)^b,
\]
see \cite[(13.11) and (13.29)]{BoI}. In the split case where
$(\sigma,\strf)=(\sigma_b,\strf_q)$ for some quadratic space $(V,q)$,
the standard identification yields
\[
  \BGamma(\End V,\sigma_b,\strf_q)=\BGamma^+(q) \qquad\text{and}\qquad
  \BOmega(\End V,\sigma_b,\strf_q)=\BOmega(q),
\]
and the maps $\chi$ and $\chi_0$ are identical respectively to the
vector representation and to the map $\chi_0$ defined
in~\eqref{eq:defchi0}. We next show that they satisfy analogues of
\eqref{eq:propchi0} and Proposition~\ref{prop:ClifinXClif}.

\begin{prop}
  \label{prop:ClifinXClifbis}
  Let $\aqp=(A,\sigma,\strf)$ be an $F$-algebra with quadratic
  pair of degree divisible by~$4$. The Clifford group $\BGamma(\aqp)$
  is a subgroup of $\BOmega(\aqp)$. More precisely,
  \[
    \BGamma(\aqp)=\underline\mu^{-1}(\BGm)\subset\BOmega(\aqp).
  \]
  Moreover, $R_{Z/F}(\BGm)\subset\BOmega(\aqp)$ and
  $\chi_0\rvert_{R_{Z/F}(\BGm)}=N_{Z/F}\colon R_{Z/F}(\BGm)\to\BGm$,
  \[
    \Int\rvert_{\BOmega(\aqp)}=C\circ\Int\circ\chi_0\in\BAut\bigl(C(\aqp)\bigr)
    \qquad\text{and}\qquad 
    N_{Z/F}\circ\underline\mu=\mu\circ\chi_0\colon
    \BOmega(\aqp)\to\BGm, 
  \]
  and
  \[
    \chi_0\rvert_{\BGamma(\aqp)}=\underline\mu\cdot\chi\colon
    \BGamma(\aqp)\to\BOrth^+(\aqp). 
  \]
\end{prop}

\begin{proof}
  The first part is proved in \cite[(13.25)]{BoI}. (Alternatively, it
  follows from Proposition~\ref{prop:ClifinXClif} by Galois descent
  from a Galois splitting field of $A$.)

  Let $R$ be a commutative $F$-algebra.
  For $z\in Z_R^\times$ we have $\underline\sigma(z)=z$ and $z*1^b =
  1^b\cdot\iota(z)$, hence $z\in\BOmega(\aqp)(R)$ with
  $\chi_0(z)=N_{Z/F}(z)$. 
  The rest follows from
  Proposition~\ref{prop:ClifinXClif} by scalar 
  extension to a splitting field of $A$.
\end{proof}

Define $\chi'\colon\BOmega(\aqp)\to\BPGO^+(\aqp)$ by composing
$\chi_0$ with the canonical map $\BGO^+(\aqp)\to\BPGO^+(\aqp)$. Recall
from~\cite[p.~352]{BoI} the 
following commutative diagram with exact rows, whose vertical maps are
canonical:
\begin{equation}
  \label{eq:commdiagClif}
  \begin{aligned}
    \xymatrix{1\ar[r]&\BGm\ar[r]\ar[d]&
      \BGamma(\aqp)\ar[r]^{\chi}\ar[d]&\BOrth^+(\aqp)\ar[d]\ar[r]&1\\
      1\ar[r]&R_{Z/F}(\BGm)\ar[r]&\BOmega(\aqp)\ar[r]^-{\chi'}&
      \BPGO^+(\aqp)\ar[r]&1
      }
  \end{aligned}
\end{equation}
The exact rows of this diagram show that $\BGamma(\aqp)$ and
$\BOmega(\aqp)$ are connected, since $\BGm$, $\BOrth^+(\aqp)$,
$R_{Z/F}(\BGm)$ and $\BPGO^+(\aqp)$ are connected.
\medbreak

In the next proposition, we write $R^1_{Z/F}(\BGm)$ for the kernel of
the norm map 
\[
  N_{Z/F}\colon R_{Z/F}(\BGm)\to \BGm.
\]

\begin{prop}
  \label{prop:chi0ontobis}
  Let $\aqp=(A,\sigma,\strf)$ be an algebra with quadratic
  pair of degree divisible by~$4$. The following sequence is exact:
  \[
    1\to
    R^1_{Z/F}(\BGm)\to\BOmega(\aqp)\xrightarrow{\chi_0}\BGO^+(\aqp) 
    \to1.
  \]
\end{prop}

\begin{proof}
  Since $\ker\chi_0\subset\ker\chi'$, it follows from the exactness of
  the lower row in \eqref{eq:commdiagClif} that
  $\ker\chi_0\subset R_{Z/F}(\BGm)$. Moreover, the following diagram
  is commutative with exact rows:
  \[
    \xymatrix{1\ar[r]& R_{Z/F}(\BGm)\ar[r]\ar[d]_{N_{Z/F}} &
      \BOmega(\aqp)\ar[r]^-{\chi'}\ar[d]_{\chi_0}&
      \BPGO^+(\aqp)\ar[r] \ar@{=}[d]&1\\
      1\ar[r]& \BGm\ar[r]& \BGO^+(\aqp) \ar[r]&
      \BPGO^+(\aqp)\ar[r]&1}
  \]
  Since we already know that $\ker\chi_0\subset R_{Z/F}(\BGm)$, it
  follows that $\ker\chi_0=R^1_{Z/F}(\BGm)$.

  As $\BGO^+(\aqp)$ is
  smooth, to prove that $\chi_0$ is onto it suffices
  by~\cite[(22.3)]{BoI} to see that $\chi_0$ defines a surjective map
  on the group of rational points over an algebraic closure. This is
  clear from the last commutative diagram above, because the norm
  $N_{Z/F}$ is surjective when $F$ is algebraically closed.
\end{proof}

As in the split case, we define the Spin group
\[
  \BSpin(\aqp)=\ker(\underline\mu\colon\BGamma(\aqp)\to \BGm)
  =\ker\bigl(\underline\mu\colon\BOmega(\aqp)\to R_{Z/F}(\BGm)\bigr)
\]
and we have an exact sequence (see~\cite[p.~352]{BoI}):
\[
  1\to\Bmu_2\to\BSpin(\aqp)\xrightarrow{\chi}\BOrth^+(\aqp) \to1.
\]
We may also restrict the map $\chi'$ to $\BSpin(\aqp)$ to obtain a
morphism $\chi'\colon\BSpin(\aqp)\to\BPGO^+(\aqp)$. This morphism is
surjective since the vector representation $\chi$ is surjective and
the canonical map $\BOrth^+(\aqp)\to\BPGO^+(\aqp)$ is surjective. Its
kernel is $R_{Z/F}(\BGm)\cap\BSpin(\aqp)=R_{Z/F}(\Bmu_2)$, hence the
following sequence is exact:
\begin{equation}
  \label{eq:exseqBSpinBPGO}
  1\to R_{Z/F}(\Bmu_2)\to \BSpin(\aqp) \xrightarrow{\chi'}
  \BPGO^+(\aqp)\to 1.
\end{equation}

The last proposition refers to the canonical quadratic pair
$(\underline\sigma,\underline \strf)$ on $C(\aqp)$ defined by
Dolphin--Qu\'eguiner-Mathieu (see \S\ref{subsec:ClAlg}). Assuming
$\deg A\equiv0\bmod8$, we write
$\Clqp(\aqp)$ for the Clifford algebra of $\aqp$ with its canonical
quadratic pair:
\[
  \Clqp(\aqp) = (C(\aqp),\underline\sigma, \underline\strf).
\]

\begin{prop}
  \label{prop:OmegainGObis}
  Let $\aqp=(A,\sigma,\strf)$ be an algebra with quadratic
  pair. If $\deg A\equiv0\bmod8$, then
  $\BOmega(\aqp)\subset\BGO^+\bigl(\Clqp(\aqp)\bigr)$. 
\end{prop}

\begin{proof}
  Let $R$ be a commutative $F$-algebra and let
  $\xi\in\BOmega(\aqp)(R)$. Since $\chi_0(\xi)\in\BGO(\aqp)(R)$, it
  follows 
  that $\Int\bigl(\chi_0(\xi)\bigr)$ is an automorphism of $\aqp_R$,
  hence $C\bigl(\Int(\chi_0(\xi))\bigr)$ is an automorphism of
  $\Clqp(\aqp)_R$. But
  Proposition~\ref{prop:ClifinXClifbis} shows that
  $C\bigl(\Int(\chi_0(\xi))\bigr)=\Int(\xi)$, hence
  $\xi\in\BGO\bigl(\Clqp(\aqp)\bigr)(R)$.
  We thus see that
  $\BOmega(\aqp)\subset\BGO\bigl(\Clqp(\aqp)\bigr)$. Since
  $\BOmega(\aqp)$ is connected, it actually 
  lies in the connected component
  $\BGO^+\bigl(\Clqp(\aqp)\bigr)$.
\end{proof}

\subsection{Lie algebras of orthogonal groups}
\label{subsec:Lieorth}

Throughout this subsection, $A$ is a central simple algebra of even
degree~$n=2m$ over an arbitrary field $F$, and $(\sigma,\strf)$ is a
quadratic pair on $A$. We discuss several Lie algebras related to the
algebra with quadratic pair
$\aqp=(A,\sigma,\strf)$, and obtain different results depending on
whether the characteristic is~$2$ or not. The discrepancies derive
from the observation that the Lie algebra of the algebraic group
scheme $\boldsymbol{\mu}_2$ is $F$ when $\charac F=2$, whereas it
vanishes when $\charac F\neq2$.
\medbreak

The bracket $[a,b]=ab-ba$ turns $A$ into a Lie algebra denoted by
$\Lie(A)$. As usual, for $a\in A$ we let $\ad_a\colon A\to A$ denote
the linear operator defined by
\[
  \ad_a(x)=[a,x] \qquad\text{for $x\in A$.}
\]

The following are subalgebras of $\Lie(A)$ associated with
the quadratic pair $(\sigma,\strf)$; they are the Lie algebras of the
algebraic group schemes $\BOrth(\aqp)$ and
$\BGO(\aqp)$ respectively, see \cite[\S23.B]{BoI}:
\begin{align*}
  \orth(\aqp) & = \Alt(\sigma) = \{a-\sigma(a)\mid a\in A\}\\
  \go(\aqp) & = \{g\in A\mid \sigma(g)+g\in F \text{ and }
                  \strf([g,s])=0\text{ for all $s\in\Sym(\sigma)$}\}.
\end{align*}
Note that $\orth(\aqp)$ depends only on $\sigma$ and not on
$\strf$. Clearly, $F\subset\go(\aqp)$. We let
\[
  \pgo(\aqp) = \go(\aqp)/F
\]
and define
\[
  \dot\mu\colon \go(\aqp)\to F \qquad\text{by}\quad
  \dot\mu(g)=\sigma(g)+g.
\]
This map is the differential of the multiplier morphism
$\mu\colon\BGO(\aqp)\to\BGm$, hence it is a Lie algebra homomorphism.

\begin{prop}
  \label{prop:Lie1}
  Let $\ell\in A$ be such that $\strf(s)=\Trd_A(\ell s)$ for all
  $s\in\Sym(\sigma)$. Then
  \begin{equation}
    \label{eq:Lie1}
    \begin{array}{rcl}
      \go(\aqp)&=&\{g\in A\mid \ad_g\circ\sigma=\sigma\circ\ad_g
                       \text{ and 
          $(\strf\circ\ad_g)(s)=0$ for all $s\in\Sym(\sigma)$}\}
                     \\
      &=&\{g\in A\mid 
    \Trd_A(gs)=(\sigma(g)+g)\strf(s) \text{ for all
          $s\in\Sym(\sigma)$}\}\\
       &=& \orth(\aqp)+\ell F
    \end{array}
  \end{equation}
  and the following sequence is exact:
  \begin{equation}
    \label{eq:Lie2}
    0\to\orth(\aqp)\to\go(\aqp) \xrightarrow{\dot\mu}
    F \to 0. 
  \end{equation}
  Moreover,
  \[
    \dim\orth(\aqp)=\dim\pgo(\aqp)=m(2m-1)
    \qquad\text{and}\qquad 
    \dim\go(\aqp)=m(2m-1)+1.
  \]
  If $\charac F\neq2$, the inclusion
  $\orth(\aqp)\hookrightarrow\go(\aqp)$ is split by
  the map 
  $\frac12(\Id-\sigma)\colon\go(\aqp)\to\orth(\aqp)$,
  and it induces a canonical isomorphism
  \[
    \orth(\aqp)\stackrel\sim\to\pgo(\aqp).
  \]
  If $\charac F=2$, the map $\dot\mu$ induces a map
  $\pgo(\aqp)\to 
  F$ for which we also use the notation $\dot\mu$, and the map
  $\orth(\aqp)\to\pgo(\aqp)$ induced by the inclusion
  $\orth(\aqp)\hookrightarrow\go(\aqp)$ fits into an
  exact sequence
  \[
    0\to
    F\to\orth(\aqp)\to\pgo(\aqp)\xrightarrow{\dot\mu} 
    F\to 0.
  \]
\end{prop}

\begin{proof}
  For $g$, $x\in A$,
  \[
    (\ad_g\circ\sigma-\sigma\circ\ad_g)(x) = [g,\sigma(x)] -
    \sigma\bigl([g,x]\bigr) = [g+\sigma(g),\sigma(x)].
  \]
  Therefore, $\ad_g\circ\sigma = \sigma\circ\ad_g$ if and only if
  $g+\sigma(g)\in F$, and the definition of $\go(\aqp)$
  readily yields
  \[
    \go(\aqp) = \{g\in A\mid \ad_g\circ\sigma=\sigma\circ\ad_g
    \text{ and $(\strf\circ\ad_g)(s)=0$ for all $s\in\Sym(\sigma)$}\}.
   \]
                     
  Now, suppose $g\in A$ satisfies $\sigma(g)+g\in F$, and let
  $\mu=\sigma(g)+g$. For $s\in\Sym(\sigma)$ we have 
  \begin{equation}
    \label{eq:Lie2bis}
    \Trd_A(gs)=\strf\bigl(gs+\sigma(gs)\bigr) =
    \strf\bigl(gs+s\sigma(g)\bigr) 
    =\strf\bigl(gs+s(\mu-g)\bigr)
    = \strf([g,s]) + \mu \strf(s).
  \end{equation}
  Therefore, $\Trd_A(gs)=\mu\strf(s)$ for $g\in\go(\aqp)$ and
  $s\in\Sym(\sigma)$, 
  hence
  \[
    \go(\aqp)\subset \{g\in A\mid 
    \Trd_A(gs)=(\sigma(g)+g)\strf(s) \text{ for all
      $s\in\Sym(\sigma)$}\}. 
  \]
  
  To prove the reverse inclusion, suppose $g\in A$ satisfies
  $\Trd_A(gs)=(\sigma(g)+g)\strf(s)$ for all $s\in\Sym(\sigma)$. We
  first show that $\sigma(g)+g\in F$. If $x\in
  A$ is such that $\Trd_A(x)=1$, then $\strf(\sigma(x)+x)=1$, hence
  the hypothesis on $g$ yields
  $\Trd_A\bigl(g(\sigma(x)+x)\bigr)=\sigma(g)+g$, which shows that
  $\sigma(g)+g\in F$. Letting $\mu=\sigma(g)+g$, we have
  by~\eqref{eq:Lie2bis} above $\Trd_A(gs)=\strf([g,s])+\mu \strf(s)$
  for 
  all $s\in\Sym(\sigma)$. On the other hand, the hypothesis on $g$
  yields $\Trd_A(gs)=\mu\strf(s)$, hence $\strf([g,s])=0$, proving 
  $g\in\go(\aqp)$.
  
  The first two equations
  in~\eqref{eq:Lie1} are thus proved. The second one shows that
  $\ell\in\go(\aqp)$ since $\Trd_A(\ell s)=\strf(s)$ for all
  $s\in\Sym(\sigma)$ and $\sigma(\ell)+\ell=1$. This last equation
  also reads $\dot\mu(\ell)=1$, hence the map
  $\dot\mu\colon\go(\aqp)\to F$ is onto. The second
  characterization of $\go(\aqp)$ in~\eqref{eq:Lie1} also
  shows that 
  \[
    \ker(\dot\mu\colon\go(\aqp)\to F) = \{g\in A\mid
    \Trd_A(gs)=0\text{ for all $s\in\Sym(\sigma)$}\},
  \]
  which means that $\ker(\dot\mu)$ is the orthogonal complement of
  $\Sym(\sigma)$ for the bilinear form $\Trd_A(XY)$. This orthogonal
  complement is known to be $\Alt(\sigma)$ by~\cite[(2.3)]{BoI}. As
  $\orth(\aqp)=\Alt(\sigma)$, it follows that
  $\orth(\aqp)\subset\go(\aqp)$ and the
  sequence~\eqref{eq:Lie2} is exact.

  From the above observations it follows that
  \[
    \orth(\aqp)+\ell F\subset\go(\aqp).
  \]
  We use dimension count to show that this inclusion is an equality,
  completing the proof of~\eqref{eq:Lie1}. Note that
  $\ell\notin\orth(\aqp)$ since
  $\orth(\aqp)=\ker(\dot\mu)$ whereas
  $\dot\mu(\ell)=1$. Therefore, $\dim(\orth(\aqp)+\ell
  F)=1+\dim\orth(\aqp)$. On the other hand, the exact
  sequence~\eqref{eq:Lie2} yields
  $\dim\go(\aqp)=1+\dim\orth(\aqp)$, hence the proof
  of~\eqref{eq:Lie1} is complete. Since $\dim\Alt(\sigma) = m(2m-1)$
  by~\cite[(2.6)]{BoI}, we obtain
  \[
    \dim\orth(\aqp)=m(2m-1) \quad\text{and}\quad
    \dim\go(\aqp)=m(2m-1)+1.
  \]
  It follows that $\dim\pgo(\aqp)=m(2m-1)$ because
  $\pgo(\aqp)=\go(\aqp)/F$. 

  If $\charac F\neq2$, then we may take $\ell=\frac12$ in the
  discussion above, so
  $\go(\aqp)=\orth(\aqp)\oplus F$ and
  $\pgo(\aqp)\simeq\orth(\aqp)$ canonically.

  If $\charac F=2$, then $F\subset\Alt(\sigma)$ because the involution
  $\sigma$ is symplectic, and the map
  $\dot\mu\colon\go(\aqp)\to F$ vanishes on $F$. Therefore,
  $\dot\mu$ induces a map $\pgo(\aqp)\to F$ whose kernel is
  the image of $\orth(\aqp)$.
\end{proof}

When the algebra $A$ is split, we may represent it as $A=\End V$
for some $F$-vector space $V$ of dimension~$n$. The quadratic pair
$(\sigma,\strf)$ is then the quadratic pair $(\sigma_b,\strf_q)$
adjoint to a nonsingular quadratic form $q$ on $V$ (see
\S\ref{subsec:quadpair}), and we write simply $\go(q)$ for
$\go(\End V, \sigma_b,\strf_q)$. 

\begin{prop}
  \label{prop:Lie1bis}
  Let $g\in\End V$ and $\mu\in F$. We have $g\in\go(q)$ and
  $\dot\mu(g)=\mu$ if and only if
  \begin{equation}
    \label{eq:trial7}
    b(g(u),u) = \mu\,q(u) \qquad\text{for all $u\in V$.}
  \end{equation}
\end{prop}

\begin{proof}
  We use the standard identification $V\otimes V=\End V$ set up in
  \S\ref{subsec:quadpair}. 
  For $s=u\otimes u\in\Sym(\sigma_b)$ we have $gs=g(u)\otimes u$,
  hence $\Trd(gs)=b(g(u),u)$. On the other hand $\strf_q(s)=q(u)$,
  hence if $g\in\go(q)$ and $\dot\mu(g)=\mu$ then the second
  characterization of $\go(\End V,\sigma_b,\strf_q)$
  in~\eqref{eq:Lie1} shows 
  that~\eqref{eq:trial7} holds. 

  Conversely, if \eqref{eq:trial7} holds, then $\Trd(gs)=\mu
  \strf_q(s)$ for 
  all $s\in\Sym(\sigma_b)$ of the form $s=u\otimes u$ with $u\in
  V$. Applying this to $s=(u+v)\otimes(u+v)$ with $u$, $v\in V$ yields
  \[
    \Trd\bigl(g(u\otimes v+v\otimes u)\bigr) = \mu \strf_q(u\otimes
    v+v\otimes u) = \mu\Trd(u\otimes v),
  \]
  hence $b(g(u),v)+b\bigl(u,g(v)\bigr)=\mu\, b(u,v)$. Since $\sigma_b$  is 
  the adjoint involution of $b$, it follows that
  $\sigma_b(g)+g=\mu$. We 
  thus see that $\Trd(gs)=(\sigma_b(g)+g)\strf_q(s)$ for all
  $s\in\Sym(\sigma_b)$, which proves $g\in\go(q)$ by the second
  characterization of $\go(\End V,\sigma_b,\strf_q)$ in~\eqref{eq:Lie1}.
\end{proof}

Returning to the general case, where the algebra $A$ is not
necessarily split,
let $C(\aqp)$ denote the Clifford algebra of $\aqp=(A,\sigma,\strf)$, and
write $c\colon A\to C(\aqp)$ for the canonical map. Every
$g\in\go(\aqp)$ defines a derivation $\delta_g$ of $C(\aqp)$ such that
\begin{equation}
  \label{eq:defdeltag}
  \delta_g\bigl(c(a)\bigr) = c([g,a]) \qquad\text{for $a\in A$;}
\end{equation}
this can be checked directly from the definition of $C(\aqp)$ or by
viewing the map $g\mapsto\delta_g$ as the differential of the morphism
$\BGO(\aqp)\to\BAut\bigl(C(\aqp)\bigr)$ defined on rational points by
mapping $g\in\GO(\aqp)$ to $C\bigl(\Int(g)\bigr)$.
The derivation $\delta_g$ is uniquely determined
by~\eqref{eq:defdeltag}, because $c(A)$ generates $A$ as an
associative algebra. 

Recall
from~\cite[\S8.C]{BoI} that $c(A)$ is a Lie subalgebra of
$\Lie\bigl(C(\aqp)\bigr)$. By \cite[p.~351]{BoI}, $c(A)$ is the
Lie algebra of the algebraic group $\BGamma(\aqp)$, whose
group of rational points is the Clifford group $\Gamma(\aqp)$,
hence we call it the \emph{Clifford Lie algebra} of $\aqp$ and write 
\[
  \clie(\aqp)=c(A)\subset\Lie\bigl(C(\aqp)\bigr).
\]
The kernel of the map $c\colon A\to\clie(\aqp)$ is
$\ker(\strf)\subset\Sym(\sigma)$ by~\cite[(8.14)]{BoI}, hence
\begin{equation}
  \label{eq:dimclie}
  \dim\clie(\aqp)=m(2m-1)+1.
\end{equation}

Let $\underline{\sigma}$ be the canonical involution on
$C(\aqp)$, which is characterized by the condition that
$\underline{\sigma}\bigl(c(a)\bigr)=c\bigl(\sigma(a)\bigr)$ for $a\in
A$. We have
\[
  \underline{\sigma}\bigl(c(a)\bigr) + c(a) = c(\sigma(a)+a) =
  \strf(\sigma(a)+a)=\Trd_A(a),
\]
hence $\underline{\sigma}(\xi)+\xi\in F$ for $\xi\in\clie(\aqp)$,
and we may define a Lie algebra homomorphism
\begin{equation}
  \label{eq:Lie2ter}
  \underline{\dot\mu}\colon\clie(\aqp)\to F \qquad\text{by}
  \quad
    \underline{\dot\mu}(\xi)=\underline{\sigma}(\xi)+\xi,
\end{equation}
so $\underline{\dot\mu}\bigl(c(a)\bigr)=\Trd_A(a)$ for $a\in A$. We
let $\spin(\aqp)$ denote the kernel
\[
  \spin(\aqp)=\ker\underline{\dot\mu}=\{c(a)\mid \Trd_A(a)=0\}
  \subset \clie(\aqp),
\]
which is the Lie algebra of the algebraic group $\BSpin(\aqp)$ defined
in~\S\ref{subsec:ClGrp}. By definition of $\spin(\aqp)$, 
the following sequence is exact:
\[
  0\to\spin(\aqp) \to\clie(\aqp)
  \xrightarrow{\underline{\dot\mu}} F\to 0,
\]
and therefore
\begin{equation}
  \label{eq:dimspin}
  \dim\spin(\aqp)=m(2m-1).
\end{equation}

Recall from \cite[(8.15)]{BoI} the Lie homomorphism
\[
  \dot\chi\colon\clie(\aqp)\to\orth(\aqp),
  \qquad
  c(a)\mapsto a-\sigma(a) \text{ for $a\in A$,}
\]
which fits in the following exact sequence
\begin{equation}
  \label{eq:Lie3}
  0\to F\to \clie(\aqp) \xrightarrow{\dot\chi} \orth(\aqp) \to
  0. 
\end{equation}
That sequence is the Lie algebra version of the exact sequence of
algebraic groups from~\cite[p.~352]{BoI}:
\[
  1\to\BGm\to\BGamma(\aqp)\xrightarrow{\chi} \BOrth^+(\aqp)
  \to 1.
\]
We let
\[
  \so(\aqp)=\dot\chi\bigl(\spin(\aqp)\bigr) =
  \{a-\sigma(a)\mid \Trd_A(a)=0\}\subset\orth(\aqp).
\]
If $\charac F\neq 2$, then $\orth(\aqp)=\Skew(\sigma)$, hence
every $a\in\orth(\aqp)$ satisfies $\Trd_A(a)=0$ and $a=\frac12a -
\sigma(\frac12a)$, hence
\[
  \so(\aqp)=\orth(\aqp).
\]
Moreover, in $\clie(\aqp)$ we have $F\cap\spin(\aqp)=0$ because
$\dot\mu(\lambda)=2\lambda$ for $\lambda\in F$, hence the restriction
of $\dot\chi$ is an isomorphism
\begin{equation}
  \label{eq:isospino}
  \dot\chi\colon\spin(\aqp)\xrightarrow{\sim}\orth(\aqp).
\end{equation}

By contrast, if $\charac F=2$ we may define a map
\[
  \Trp\colon\orth(\aqp)\to F \qquad\text{by}\quad
    \Trp\bigl(a-\sigma(a)\bigr)=\Trd_A(a),
\]
because $\Trd_A\bigl(\Sym(\sigma)\bigr)=0$. (The map $\Trp$ is the
\emph{pfaffian trace}, see \cite[(2.13)]{BoI}.) For $a$, $b\in A$ we
have 
\[
  [a-\sigma(a),b-\sigma(b)] = [a-\sigma(a),b] -
  \sigma([a-\sigma(a),b]),
\]
hence
\[
  \Trp([a-\sigma(a),b-\sigma(b)])=\Trd_A([a-\sigma(a),b]) = 0.
\]
Therefore, $\Trp$ is a Lie algebra homomorphism. Note also that
$F\subset\spin(\aqp)$ because $\dot\mu(\lambda)=2\lambda=0$ for
$\lambda\in F$. Therefore, there is a commutative diagram with exact
rows and columns: 
\begin{equation}
  \label{eq:commdiagspinchar2}
  \begin{aligned}
  \xymatrix{
    &0\ar[d]&0\ar[d]&&\\
    &F\ar[d]\ar@{=}[r]&F\ar[d]&&\\
    0\ar[r]&\spin(\aqp)\ar[r]\ar[d]_{\dot\chi}&
    \clie(\aqp) \ar[r]^{\underline{\dot\mu}} \ar[d]_{\dot\chi} & F
    \ar[r] \ar@{=}[d]&0\\
    0\ar[r]&\so(\aqp)\ar[r]\ar[d]&\orth(\aqp)\ar[r]^{\Trp}\ar[d]&
    F\ar[r]&0\\
    &0&0&&}
  \end{aligned}
\end{equation}

\subsection{Extended Clifford Lie algebras}
\label{subsec:xclie}

Throughout this subsection $A$ is a central simple algebra of
degree~$n=2m$ over an arbitrary field $F$, and we assume $m$ is
even. Let $(\sigma,\strf)$ be a quadratic pair on $A$, and let
$\aqp=(A,\sigma,\strf)$. Recall from 
\S\ref{subsec:ClGrp} the Clifford bimodule $B(\aqp)$ with
its canonical left $A$-module homomorphism $b\colon A\to
B(\aqp)$. We write $Z$ for the center of $C(\aqp)$
and $\iota$ for the nontrivial $F$-automorphism of $Z$.

Since the left $A$-module action on $B(\aqp)$ commutes with the left
and right $C(\aqp)$-module actions, the condition
$\underline\sigma(\xi)*A_R^b\cdot\xi=A_R^b$ in the definition of the
extended Clifford group $\BOmega(\aqp)$ is equivalent to
$\underline\sigma(\xi)*1^b\cdot\xi\in A_R^b$. The Lie algebra of
$\BOmega(\aqp)$ is therefore as follows:

\begin{definition}
  \label{defn:xclie}
  The \emph{extended Clifford Lie algebra of $\aqp$} is
  \[
    \xclie(\aqp)= \{\xi\in C(\aqp)\mid
    \underline\sigma(\xi)+\xi\in Z \text{ and }
    \underline\sigma(\xi)*1^b+1^b\cdot\xi\in A^b\}.
  \]
  It is shown in~\cite[p.~352]{BoI} that the algebraic group scheme
  $\BOmega(\aqp)$ is smooth, because $R_{Z/F}(\BGm)$ and
  $\BPGO^+(\aqp)$ are smooth and the lower row of the
  diagram~\eqref{eq:commdiagClif} is exact.
  Since $\dim R_{Z/F}(\BGm)=2$ and $\dim\BPGO^+(\aqp)=m(2m-1)$, it
  follows that
  \[
    \dim\BOmega(\aqp) = \dim\xclie(\aqp)=m(2m-1)+2.
  \]
  
  For $\xi\in\xclie(\aqp)$ we write
  \[
    \underline{\dot\mu}(\xi)=\underline\sigma(\xi)+\xi\in Z.
  \]
  Since the map $b$ is injective, for each
  $\xi\in\xclie(\aqp)$ there is a uniquely determined element
  $\dot\chi_0(\xi)\in A$ such that
  \[
    \underline\sigma(\xi)*1^b+1^b\cdot\xi = \dot\chi_0(\xi)^b.
  \]
  Thus, letting $F[\varepsilon]$ denote the algebra of dual numbers,
  where $\varepsilon^2=0$, we have
  \[
    \underline\sigma(1+\varepsilon\xi)*1^b\cdot (1+\varepsilon\xi) =
    \bigl(1+\varepsilon\dot\chi_0(\xi)\bigr)^b
    \qquad\text{for $\xi\in\xclie(\aqp)$.}
  \]
  This shows that $\dot\chi_0$ is the differential of
  $\chi_0\colon\BOmega(\aqp)\to \BGO^+(\aqp)$.
\end{definition}

For the next statement, recall from~\eqref{eq:defdeltag} that every
$g\in\go(\aqp)$ defines a derivation $\delta_g$ of $C(\aqp)$ such that
$\delta_g\bigl(c(a)\bigr)=c([g,a])$ for all $a\in A$.

\begin{prop}
  \label{prop:xclie1new}
  The Lie algebra $\xclie(\aqp)$ is a subalgebra
  of 
  $\Lie\bigl(C(\aqp)\bigr)$ containing $Z$ and
  $\clie(\aqp)$, and $\dot\chi_0$, $\underline{\dot\mu}$ are
  Lie algebra homomorphisms
  \[
    \dot\chi_0\colon\xclie(\aqp)\to\go(\aqp)
    \qquad\text{and}\qquad
    \underline{\dot\mu}\colon \xclie(\aqp)\to Z.
  \]
  Moreover, $\dot\chi_0(z)=\Tr_{Z/F}(z)\in F$ for $z\in Z$,
  \[
    \ad_\xi = \delta_{\dot\chi_0(\xi)}
    \quad\text{and}\quad
  \dot\mu\bigl(\dot\chi_0(\xi)\bigr)=
  \Tr_{Z/F}\bigl(\underline{\dot\mu}(\xi)\bigr)
    \quad\text{for
      $\xi\in\xclie(\aqp)$,}
  \]
  and
  \[
    \dot\chi_0(\xi)=\underline{\dot\mu}(\xi)+\dot\chi(\xi)
    \qquad\text{for $\xi\in\clie(\aqp)$.}
  \]
\end{prop}

\begin{proof}
  That $\xclie(\aqp)$ is a Lie subalgebra of
  $\Lie\bigl(C(\aqp)\bigr)$ and $\dot\chi_0$, $\underline{\dot\mu}$
  are Lie algebra homomorphisms is clear because $\xclie(\aqp)$ is the
  Lie 
  algebra of $\BOmega(\aqp)$ and $\dot\chi_0$, $\underline{\dot\mu}$
  are the differentials of $\chi_0$ and $\underline\mu\colon
  \BOmega(\aqp)\to R_{Z/F}(\BGm)$ respectively.

  Over the algebra $F[\varepsilon]$ of dual numbers,
  Proposition~\ref{prop:ClifinXClifbis} yields
  \[
    \Int(1+\varepsilon\xi) =
    C\bigl(\Int(\chi_0(1+\varepsilon\xi))\bigr) \qquad\text{for
      $\xi\in\xclie(\aqp)$},
  \]
  Hence for $\xi\in\xclie(\aqp)$ and $a\in A$
  \[
    (1+\varepsilon\xi)\,c(a)\,(1-\varepsilon\xi) =
    c\bigl((1+\varepsilon\dot\chi_0(\xi))\,a
    \,(1-\varepsilon\dot\chi_0(\xi))\bigr).
  \]
  Comparing the coefficients of $\varepsilon$ yields $[\xi,c(a)] =
  c([\dot\chi_0(\xi),a])$. Therefore, the derivations $\ad_\xi$ and
  $\delta_{\dot\chi_0(\xi)}$ coincide on $c(A)$, hence
  $\ad_\xi=\delta_{\dot\chi_0(\xi)}$ because $c(A)$ generates
  $C(\aqp)$ as an associative algebra.
  
  The other equations similarly follow by
  taking the differentials of $\chi_0(z)=N_{Z/F}(z)$ for $z\in
  Z^\times$,  $\mu\bigl(\chi_0(\xi)\bigr) =
  N_{Z/F}\bigl(\underline\mu(\xi)\bigr)$ for $\xi\in\BOmega(\aqp)$ and
  $\chi_0(\xi)=\underline\mu(\xi)\chi(\xi)$ for $\xi\in\Gamma(\aqp)$
  (see Proposition~\ref{prop:ClifinXClifbis}).
\end{proof}

\begin{corol}
  \label{corol:xcliecharnot2}
  If $\charac F\neq2$, then $\xclie(\aqp)=\clie(\aqp)+Z$.
\end{corol}

\begin{proof}
  If $\charac F\neq2$, then $Z\cap\clie(\aqp)=F$, while
  Proposition~\ref{prop:xclie1new} shows that
  $\clie(\aqp)+Z\subset\xclie(\aqp)$. Dimension count then shows that
  $\xclie(\aqp)=\clie(\aqp)+Z$.
\end{proof}

Note that $Z\subset\clie(\aqp)$ if $\charac F=2$
(see~\cite[(8.27)]{BoI}), hence
$\clie(\aqp)+Z=\clie(\aqp)\subsetneq\xclie(\aqp)$ in that case.
\medbreak

The following Lie algebra versions of the commutative
diagram~\eqref{eq:commdiagClif} and of
Proposition~\ref{prop:chi0ontobis} can be derived from their algebraic
group scheme versions. We give a direct proof instead.

\begin{prop}
  \label{prop:xclie2new}
  Let $Z^0=\ker(\Tr\colon Z\to F)$ and let
  $\dot\chi'\colon\xclie(\aqp)\to\pgo(\aqp)$ be defined by
  $\dot\chi'(\xi)=\dot\chi_0(\xi)+F$ for $\xi\in\xclie(\aqp)$. The
  following sequence is exact: 
  \begin{equation}
    \label{eq:exseqxclienew}
    0\to Z^0\to\xclie(\aqp)\xrightarrow{\dot\chi_0}
    \go(\aqp)\to 0.
  \end{equation}
  The following diagram is commutative with exact rows and canonical
  vertical maps:
  \begin{equation}
    \label{eq:commdiagxclienew}
    \begin{aligned}
    \xymatrix{0\ar[r]&F\ar[r]
      \ar[d]&\clie(\aqp)
      \ar[d]\ar[r]^{\dot\chi}&
      \orth(\aqp)\ar[r]\ar[d]&0\\  
      0\ar[r]&Z\ar[r]&\xclie(\aqp)\ar[r]^{\dot\chi'}&
      \pgo(\aqp)\ar[r]&0}
    \end{aligned}
  \end{equation}
  Moreover,
  \[
    \clie(\aqp)=\{\xi\in\xclie(\aqp)\mid
  \underline{\dot\mu}(\xi)\in F\}\quad\text{and}\quad
  \spin(\aqp)=
  \ker(\underline{\dot\mu}\colon\xclie(\aqp)\to Z).
  \]
\end{prop}

\begin{proof}
  We first show $Z^0=\ker\dot\chi_0$. The inclusion
  $Z^0\subset\ker\dot\chi_0$ follows from
  Proposition~\ref{prop:xclie1new}. To prove the reverse inclusion,
  let 
  $\xi\in\ker\dot\chi_0$. Proposition~\ref{prop:xclie1new} yields
  $[\xi,c(a)]=0$ 
  for all $a\in A$. As $c(A)$ generates $C(\aqp)$, we conclude that
  $\xi\in Z$. But then Proposition~\ref{prop:xclie1new} shows that
  $\dot\chi_0(\xi)=\Tr_{Z/F}(\xi)$, hence $\xi\in Z^0$.

  Dimension count now shows that $\dot\chi_0$ is surjective, hence
  \eqref{eq:exseqxclienew} is an exact sequence.

  The upper sequence of diagram~\eqref{eq:commdiagxclienew}
  is~\eqref{eq:Lie3}. We have just seen that $\dot\chi_0$
  is surjective, hence $\dot\chi'$ also is surjective. By
  Proposition~\ref{prop:xclie1new}, its kernel contains $Z$. Dimension
  count then yields $\ker\dot\chi'=Z$, hence the lower sequence of the
  diagram is exact. Commutativity of the diagram follows from
  Proposition~\ref{prop:xclie1new}, since $\underline{\dot\mu}(\xi)\in
  F$ for $\xi\in\clie(\aqp)$.

  This last observation shows that $\clie(\aqp)$ lies in the
  kernel of the map
  \[
    \dot\varkappa\colon\xclie(\aqp)\to Z/F,\qquad
    \xi\mapsto\underline{\dot\mu}(\xi)+F.
  \]
  We have to prove that $\clie(\aqp)=\ker\dot\varkappa$. To
  see this, it suffices to show that $\dot\varkappa$ is
  onto, because
  $\dim\clie(\aqp) = (\dim\xclie(\aqp))-1$ and
  $\dim(Z/F)=1$.

  If $\charac F\neq2$, surjectivity is clear because
  $\underline{\dot\mu}(z)=2z$ for all $z\in Z$. If $\charac F=2$, we
  pick an element $\ell\in\go(\aqp)$ such that
  $\dot\mu(\ell)=1$. Since $\dot\chi_0$ is onto, we may find
  $\xi\in\xclie(\aqp)$ such that $\dot\chi_0(\xi)=\ell$. Then by
  Proposition~\ref{prop:xclie1new} we have
  $\Tr_{Z/F}\bigl(\underline{\dot\mu}(\xi)\bigr) = 1$, hence
  $\underline{\dot\mu}(\xi)\notin F$. This shows $\dot\varkappa$ is
  onto. 

  To complete the proof, it suffices
  to observe that 
  $\spin(\aqp)=\ker(\underline{\dot\mu}\colon
  \clie(\aqp)\to F)$ by definition.
\end{proof}

When $\charac F=2$ we have $\underline{\dot\mu}(Z)=0$, hence
$Z\subset\spin(\aqp)$ and we may define a Lie algebra
homomorphism $\dot S\colon \pgo(\aqp)\to Z$ by
\[
  \dot S(g+F) = \underline{\dot\mu}(\xi)
  \quad\text{for any $\xi\in\xclie(\aqp)$ such that $\dot\chi'(\xi) =
    g+F$.}
\]

\begin{corol}
  \label{corol:spinpgo}
  If $\charac F\neq2$, then $\dot\chi'$ yields an isomorphism
  $\spin(\aqp)\xrightarrow{\sim}\pgo(\aqp)$.

  If $\charac F=2$, the restriction of $\dot\chi'$ fits in the exact
  sequence
  \[
    0\to Z\to\spin(\aqp)\xrightarrow{\dot\chi'} \pgo(\aqp)
    \xrightarrow{\dot S} Z\to 0.
  \]
\end{corol}

\begin{proof}
  If $\charac F\neq2$ we saw in \eqref{eq:isospino} that $\dot\chi$
  yields an isomorphism $\spin(\aqp)\simeq\orth(\aqp)$, and in
  Proposition~\ref{prop:Lie1} we saw that the canonical map is an
  isomorphism $\orth(\aqp)\xrightarrow{\sim}\pgo(\aqp)$, hence
  $\dot\chi'$ is an isomorphism $\spin(\aqp)\simeq\pgo(\aqp)$.

  For the rest of the proof, assume $\charac F=2$. Since
  $\dot\chi'\colon\xclie(\aqp)\to\pgo(\aqp)$ is onto and
  $\spin(\aqp)=\ker \underline{\dot\mu}$ by
  Proposition~\ref{prop:xclie2new}, it is clear from the definition of
  $\dot S$ that $\ker\dot S=\dot\chi'\bigl(\spin(\aqp)\bigr)$. As
  $\dot S\bigl(\pgo(\aqp)\bigr) \subset Z$, it follows that
  \begin{equation}
    \label{eq:ineq1}
    \dim\pgo(\aqp)-\dim\dot\chi'\bigl(\spin(\aqp)\bigr)\leq 2.
  \end{equation}
  On the other hand we have $Z\subset\spin(\aqp)$ because
  $\underline{\dot\mu}(Z)=0$, and $Z\subset\ker\dot\chi'$ by
  Proposition~\ref{prop:xclie2new}, hence
  \begin{equation}
    \label{eq:ineq2}
    \dim\dot\chi'\bigl(\spin(\aqp)\bigr)\leq \dim \spin(\aqp) - 2.
  \end{equation}
  As $\dim\pgo(\aqp) = m(2m-1) = \dim\spin(\aqp)$ by
  Proposition~\ref{prop:Lie1} and \eqref{eq:dimspin}, the inequalities
  \eqref{eq:ineq1} and \eqref{eq:ineq2} cannot be strict. Therefore,
  $Z=\ker\dot\chi'=\dot S\bigl(\pgo(\aqp)\bigr)$ and the corollary is
  proved. 
\end{proof}

Finally, we consider the case where $m$ is divisible by~$4$; then
$C(\aqp)$ carries a canonical quadratic pair
$(\underline\sigma,\underline\strf)$ defined by
Dolphin--Qu\'eguiner-Mathieu: see the end of \S\ref{subsec:ClAlg}. As
in Proposition~\ref{prop:OmegainGObis}, we let
\[
  \Clqp(\aqp) = (C(\aqp),\underline\sigma, \underline\strf).
\]

\begin{prop}
  \label{prop:omega}
  If $\deg A\equiv0\bmod8$, then
  $\xclie(\aqp)\subset\go\bigl(\Clqp(\aqp)\bigr)$.
\end{prop}

\begin{proof}
  The definition of $\xclie(\aqp)$ entails that
  $\underline\sigma(\xi)+\xi\in Z$ for all
  $\xi\in\xclie(\aqp)$, hence it suffices to prove
  $\underline\strf([\xi,s])=0$ for $\xi\in\xclie(\aqp)$ and
  $s\in\Sym(\underline\sigma)$. By the definition of
  $\underline\strf$, this amounts to showing that if $a\in A$
  is such that $\Trd_A(a)=1$, then
  \[
    \Trd_{C(\aqp)}(c(a)\,[\xi,s])=0 \qquad\text{for
      $\xi\in\xclie(\aqp)$ and $s\in\Sym(\underline\sigma)$.}
  \]
  For this, observe that
  \[
    \Trd_{C(\aqp)}\bigl(c(a)(\xi s-s\xi)\bigr) =
    \Trd_{C(\aqp)}\bigl((\xi c(a)-c(a)\xi)s\bigr).
  \]
  Now, by Proposition~\ref{prop:xclie1new} we have
  $[\xi,c(a)]=c([\dot\chi_0(\xi),a])$. As
  $\Trd_A([\dot\chi_0(\xi),a])=0$, it 
  follows that 
  $c([\dot\chi_0(\xi),a])\in\spin(\aqp)$.
  Now,
  $\spin(\aqp)\subset\Alt(\underline\sigma)$
  by~\cite[Lemma~3.2]{DQM}, hence 
  \[
    \Trd_{C(\aqp)}\bigl(c([\dot\chi_0(\xi),a])s\bigr)=0
    \qquad\text{for all $s\in\Sym(\underline\sigma)$.}
    \qedhere
  \]
\end{proof}

\begin{remark}
When $\charac F=2$, the Lie algebra
$\Lie(A)$ has an additional structure given by the squaring map
$a\mapsto a^2$, which turns it into a \emph{restricted Lie
  algebra}. It can be verified that the Lie algebras $\orth(\aqp)$,
$\so(\aqp)$,  $\go(\aqp)$, $\pgo(\aqp)$, $\clie(\aqp)$, $\spin(\aqp)$, $\xclie(\aqp)$
are all restricted (i.e., preserved under the squaring map), and the
maps  
$\dot\mu$, $\underline{\dot\mu}$, $\dot\chi$, $\Trp$, $\dot\chi_0$,
$\dot S$ are
homomorphisms of restricted Lie algebras (i.e., commute with the
squaring map). The proof is omitted, as the restricted Lie algebra
structure will not be used in this work.
\end{remark}

\subsection{Homomorphisms from Clifford algebras}
\label{subsec:homos}

Throughout this subsection, $\aqp=(A,\sigma,\strf)$ is an algebra with
quadratic pair of
degree~$2m$ over an arbitrary field $F$. We assume $m\equiv0\bmod4$
and the discriminant 
of $(\sigma,\strf)$ is trivial, which implies that the Clifford
algebra $C(\aqp)$ decomposes as an algebra with quadratic
pair into a direct product of two central simple $F$-algebras with
quadratic pair of degree~$2^{m-1}$. We further choose a polarization
of $\aqp$ (see Definition~\ref{defn:orientationaqp}), which provides a
designation of the primitive central idempotents of $C(\aqp)$ as $z_+$
and $z_-$. The simple components of $C(\aqp)$ are then
\[
  C_+(\aqp)=C(\aqp)z_+ \qquad\text{and}\qquad
  C_-(\aqp)=C(\aqp)z_-.
\]
We write $\pi_+\colon C(\aqp)\to C_+(\aqp)$ and $\pi_-\colon
C(\aqp)\to C_-(\aqp)$ for the projections:
\[
  \pi_+(\xi)=\xi z_+,\qquad
  \pi_-(\xi)=\xi z_-\qquad\text{for $\xi\in C(\aqp)$,}
\]
and let
\[
  \Clqp(\aqp) = (C(\aqp),\underline\sigma, \underline\strf).
\]

Given another central simple $F$-algebra with quadratic pair
$\aqp'=(A',\sigma',\strf')$ of degree~$2^{m-1}$, we define a
\emph{homomorphism of algebras with quadratic pair}
\begin{equation}
  \label{eq:dephi}
  \varphi\colon \Clqp(\aqp)
  \to \aqp'
\end{equation}
to be an $F$-algebra homomorphism 
$\varphi\colon C(\aqp)\to A'$ such that
\[
  \varphi\circ\underline\sigma=\sigma'\circ\varphi
  \quad\text{and}\quad
  \varphi\bigl(\underline\strf(s)\bigr) = \strf'\bigl(\varphi(s)\bigr)
  \text{ for all $s\in\Sym(\underline\sigma)$.}
\]
Since we assume $\dim A'=\frac12\dim C(\aqp)$, such a
homomorphism factors through one of the projections $\pi_+$ or
$\pi_-$, and maps the 
center $Z$ of $C(\aqp)$ to $F$. It readily follows that $\varphi$
defines a morphism $\BGO\bigl(\Clqp(\aqp)\bigr)
\to\BGO(\aqp')$ and maps
$\go\bigl(\Clqp(\aqp)\bigr)$ to $\go(\aqp')$.

\begin{definition}
  \label{defn:signofhomo}
We say that \emph{$\varphi$ has the $+$ sign} if it factors through
$\pi_+$ (i.e., $\varphi(z_+)=1$ and $\varphi(z_-)=0$), and that
\emph{$\varphi$ has the $-$ sign} if it factors through $\pi_-$ (i.e.,
$\varphi(z_+)=0$ and $\varphi(z_-)=1$).
\end{definition}

Since $\BOmega(\aqp)\subset\BGO^+\bigl(\Clqp(\aqp)\bigr)$ by
Proposition~\ref{prop:OmegainGObis}, we may 
restrict $\varphi$ to $\BOmega(\aqp)$ to obtain the following
commutative diagram with exact rows, where $\pm$ is the sign
of $\varphi$:
\begin{equation}
  \label{eq:commdiaghomoClif2}
  \begin{aligned}
  \xymatrix{1\ar[r]& R_{Z/F}(\BGm)\ar[r]\ar[d]_{\pi_\pm} &
    \BOmega(\aqp) 
    \ar[d]_{\varphi} \ar[r]^-{\chi'} & \BPGO^+(\aqp) \ar[r]
    \ar[d]_{\overline\varphi} & 1\\
    1\ar[r]& \BGm\ar[r] & \BGO^+(\aqp') \ar[r] &
    \BPGO^+(\aqp')\ar[r] &
    1}
  \end{aligned}
\end{equation}
We also consider the corresponding diagram with exact rows involving
the differentials: 
\begin{equation}
    \label{eq:commdiagxcliehomo}
    \begin{aligned}
    \xymatrix{0\ar[r]&Z\ar[r]\ar[d]_{\pi_\pm}&
      \xclie(\aqp)\ar[r]^{\dot\chi'}
      \ar[d]_{\varphi}
        & \pgo(\aqp)\ar[r]\ar[d]_{\theta}&0
        \\
        0\ar[r]& F\ar[r]& \go(\aqp') \ar[r]&
        \pgo(\aqp')\ar[r]&0
      }
    \end{aligned}
  \end{equation}

Since $\varphi\circ\underline\sigma = \sigma'\circ\varphi$, it follows
that $\varphi\circ\underline\mu =
\mu\circ\varphi$ on $\BOmega(\aqp)$, hence $\varphi$
maps $\BSpin(\aqp)$ to $\BOrth^+(\aqp')$. Restricting the morphism
$\varphi$ to $\BSpin(\aqp)$, we obtain
from~\eqref{eq:commdiaghomoClif2} 
the following commutative diagram of algebraic group schemes with
exact rows: 
\begin{equation}
  \label{eq:commdiaghomoClif3}
  \begin{aligned}
  \xymatrix{1\ar[r]& R_{Z/F}(\Bmu_2) \ar[r] \ar[d]_{\pi_\pm} &
    \BSpin(\aqp)\ar[r]^-{\chi'} \ar[d]_{\varphi} & \BPGO^+(\aqp) \ar[r]
    \ar[d]_{\overline\varphi} &1\\
    1\ar[r]&\Bmu_2\ar[r]& \BOrth^+(\aqp') \ar[r] & \BPGO^+(\aqp')
    \ar[r] & 1}
  \end{aligned}
\end{equation}

Our goal in the rest of this subsection is to show that the map
$\theta$ in~\eqref{eq:commdiagxcliehomo} determines the homomorphism
$\varphi$ in~\eqref{eq:dephi} uniquely.

\begin{definition}
  \label{defn:lift}
  Given $\varphi$ as in~\eqref{eq:dephi}, of sign $\pm$, the
  Lie algebra homomorphism $\theta\colon\pgo(\aqp)\to\pgo(\aqp')$ in
  diagram~\eqref{eq:commdiagxcliehomo} is said to be \emph{induced} by
  $\varphi$. 
    Changing the perspective, a Lie algebra homomorphism
    $\theta\colon\pgo(\aqp)\to\pgo(\aqp')$ is said to
    be \emph{liftable} if it is induced by some homomorphism of
    algebras with quadratic pair
    $\varphi$, which is then called a \emph{lift} of
    $\theta$. If $\theta$ is induced by a homomorphism $\varphi$, the
    \emph{sign} of $\theta$ is defined to be the same as the sign of
    $\varphi$. 
\end{definition}

The following theorem shows that the latter definition is not
ambiguous:

\begin{thm}
  \label{thm:lift}
  If a Lie algebra homomorphism
  $\theta\colon\pgo(\aqp)\to\pgo(\aqp')$ is liftable,
  then its lift is unique.
\end{thm}

\begin{proof}
  It suffices to prove the theorem after scalar extension. We may
  therefore assume $\aqp=(\End V, \sigma_b, \strf_q)$ for some
  hyperbolic quadratic space $(V,q)$ of dimension~$2m$. We use the
  standard identification 
  $V\otimes V=\End V$ set up in \S\ref{subsec:quadpair}.
  
  Since $q$ is hyperbolic, by decomposing $V$ into an orthogonal sum
  of hyperbolic planes we may find a base $(e_i,e'_i)_{i=1}^m$ of $V$
  such that
  \[
    q(e_i)=q(e'_i)=b(e_i,e_j) = b(e'_i,e'_j)=0\qquad\text{for all $i$,
      $j=1$, \ldots, $m$}
\]
and
\[
b(e_i,e'_j)=
\begin{cases}
  1&\text{if $i=j$},\\
0&\text{if $i\neq j$.}
\end{cases}
\]
The products $e_ie_j$, $e_ie'_j$, $e'_ie_j$, $e'_ie'_j$ for $i$,
$j=1$, \ldots, $m$ span $V\cdot V\subset C_0(V,q)$, hence they
generate $C_0(V,q)$ as an $F$-algebra. Since $q(e_i)=q(e'_i)=0$ for
all $i$, we do not need to count $e_ie_j$ nor $e'_ie'_j$ among the
generators if $i=j$. Moreover, $e_je'_j+e'_je_j=b(e_j,e'_j)=1$ for all
$j$, hence if $i\neq j$
\[
  e_ie'_i=e_i(e_je'_j+e'_je_j)e'_i = (e_ie_j)(e'_je'_i) +
  (e_ie'_j)(e_je'_i)
\]
and similarly
\[
  e'_ie_i=e'_i(e_je'_j+e'_je_j)e_i = (e'_ie_j)(e'_je_i) +
  (e'_ie'_j)(e_je_i).
\]
These equations show that $e_ie'_i$ and $e'_ie_i$ lie in the
subalgebra of $C_0(V,q)$ generated by $e_ke_\ell$, $e_ke'_\ell$,
$e'_ke_\ell$, $e'_ke'_\ell$ for all $k\neq\ell$ in
$\{1,\ldots,m\}$. Therefore, these elements generate $C_0(V,q)$.

Consequently, if $\varphi_1$, $\varphi_2\colon C_0(V,q)\to A'$ are two
lifts of a given
$\theta\colon\pgo(\aqp)\to\pgo(\aqp')$, it suffices
to prove that $\varphi_1$ and $\varphi_2$ coincide on $e_ke_\ell$,
$e_ke'_\ell$, $e'_ke_\ell$, $e'_ke'_\ell$ for all $k\neq\ell$ in
$\{1,\ldots,m\}$ to conclude that $\varphi_1=\varphi_2$. This is what
we proceed to show.

The condition that $\varphi_1$ and $\varphi_2$ induce the same
$\theta$ means that $\varphi_1(\xi)-\varphi_2(\xi)\in F$ for all
$\xi\in\xclie(q)$, hence
\[
  \varphi_1([\xi_1,\xi_2]) = \varphi_2([\xi_1,\xi_2])
  \qquad\text{for all $\xi_1$, $\xi_2\in\xclie(q)$}.
\]
We apply this to $\xi_1=c(u_1\otimes v_1)=u_1v_1$ and
$\xi_2=c(u_2\otimes v_2)=u_2v_2\in\clie(q)\subset\xclie(q)$ for $u_1$,
$u_2$, $v_1$, $v_2\in V$. If $i\neq j$, we have
\[
  [e_ie_j,e'_je_j] = e_ie_je'_je_j - e'_je_je_ie_j.
\]
Since $e_i$ and $e_j$ anticommute and $e_j^2=0$, the second term on
the right side vanishes. In the first term, we may substitute
$1-e'_je_j$ for $e_je'_j$  and use $e_j^2=0$ to obtain
\[
  [e_ie_j,e'_je_j] = e_i(1-e'_je_j)e_j = e_ie_j.
\]
Similar computations yield for all $i\neq j$ in $\{1,\ldots,
m\}$ 
\[
  [e_ie'_j, e_j e'_j]= e_ie'_j,\qquad
  [e'_ie_j,e'_j e_j]=e'_ie_j, \qquad
  [e'_i e'_j,e_j e'_j]=e'_ie'_j.
\]
Since $\varphi_1$ and $\varphi_2$ take the same value on each
$[\xi_1,\xi_2]$ for $\xi_1$, $\xi_2\in\xclie(q)$, it follows that
$\varphi_1$ and $\varphi_2$ coincide on each $e_ie_j$, $e_ie'_j$,
$e'_ie_j$ and $e'_ie'_j$ for $i\neq j$, hence $\varphi_1=\varphi_2$.
\end{proof}

\begin{corol}
  \label{corol:trial3}
  Let $\theta\colon\pgo(\aqp)\to\pgo(\aqp')$ be a
  homomorphism of Lie algebras and let $K$ be a Galois field
  extension of $F$. If $\theta_K\colon \pgo(\aqp)_K\to
  \pgo(\aqp')_K$ is liftable, then $\theta$ is liftable.
\end{corol}

\begin{proof}
  Let $\varphi\colon C(\aqp)_K\to A'_K$ be the lift of
  $\theta_K$, and let $\rho$ be an element of the Galois group of
  $K/F$. Then
  $(\Id_{A'}\otimes\rho)\circ
  \varphi\circ(\Id_{C(\aqp)}\otimes\rho^{-1})\colon  
  C(\aqp)_K\to A'_K$ is a lift of
  $(\Id_{\pgo(\aqp')}\otimes\rho)\circ\theta_K\circ
  (\Id_{\pgo(\aqp)}\otimes\rho^{-1})=\theta_K$, 
  hence, by uniqueness of the lift,
  \[
    (\Id_{A'}\otimes\rho)\circ\varphi
    \circ(\Id_{C(\aqp)}\otimes\rho^{-1})=\varphi.
  \]
  Therefore, $\varphi\rvert_{C(\aqp)}$ maps
  $C(\aqp)$ to $A'$; it lifts $\theta$ since $\varphi$ lifts
  $\theta_K$.  
\end{proof}

\section{Compositions of quadratic spaces}
\label{chap:compospaces}

This section introduces the notion of a composition of quadratic
spaces.
We emphasize an important feature of compositions, which will be
central to the definition of trialitarian automorphisms in the next
section: each composition gives rise to two other compositions on the
quadratic spaces cyclically permuted.
Restricting to the case where the quadratic spaces have the
same finite dimension, we show that this dimension is~$1$, $2$, $4$ or
$8$, the comparatively trivial case of dimension~$1$ arising only when
the characteristic of the base field is different from~$2$. In order
to prove this fairly classical result we set up isomorphisms of
algebras with involution or with quadratic pair involving Clifford
algebras. In dimension~$8$, these isomorphisms will provide in the
next section examples of trialitarian triples of split algebras. In
\S\ref{subsec:simiso} we investigate similitudes and isometries of
compositions of quadratic spaces, which define algebraic groups that
are close analogues of those attached to quadratic spaces.

Even though the quadratic spaces in a composition are not necessarily
isometric, it is easy to see that every composition of quadratic
spaces is similar to a composition of \emph{isometric} quadratic
spaces (see Proposition~\ref{prop:simtoPfistercomp}).
The focus in the last two subsections is on this type of
compositions. Using a related notion of composition of \emph{pointed}
quadratic spaces, we
show in \S\ref{subsec:pointedcomp} 
that every composition of isometric quadratic spaces is isomorphic to
its derivatives and also to
a composition that is its own derivative, and in
\S\ref{subsec:compalg} we discuss compositions of quadratic spaces
arising from the classical notion of composition algebra. To each
composition algebra is associated a composition of isometric quadratic
spaces, and isotopies of composition algebras are shown in
Theorem~\ref{thm:isot} to be similitudes of the associated
compositions of quadratic spaces.
\medbreak

Throughout this section, $F$ is an arbitrary field. Unless explicitly
specified, there is no restriction on its characteristic $\charac F$.

\subsection{Composition maps and their cyclic derivatives}
\label{subsec:compmap}

Let $(V_1,q_1)$, $(V_2, q_2)$, $(V_3,q_3)$ be (finite-dimensional)
quadratic spaces over $F$. Write $b_1$, $b_2$,
$b_3$ for the associated polar bilinear forms
\[
  b_i\colon V_i\times V_i\to F, \qquad
  b_i(x_i,y_i) = q_i(x_i+y_i) - q_i(x_i) - q_i(y_i) \quad\text{for
    $i=1$, $2$, $3$}.
\]
We assume throughout that the forms $b_1$, $b_2$, $b_3$ are
nonsingular, hence each $\dim V_i$ is even if $\charac F=2$, and we
may use the polar forms to identify each $V_i$ with its dual
$V_i^*$. Bilinear maps $V_1\times V_2\to V_3$ are then identified with
tensors in $V_3\otimes V_2\otimes V_1$, so that for $v_i\in V_i$ the
tensor $v_3\otimes v_2\otimes v_1$ is regarded as the bilinear map
\[
  V_1\times V_2 \to V_3, \qquad (x_1,\,x_2)\mapsto v_3\,b_2(v_2,x_2)\,
  b_1(v_1,x_1).
\]
Let
\[
  \partial\colon V_3\otimes V_2\otimes V_1 \to V_1\otimes V_3 \otimes
  V_2
  \qquad\text{and}\qquad
  \partial^2\colon V_3\otimes V_2\otimes V_1 \to V_2\otimes V_1\otimes
  V_3
\]
be the isomorphisms that permute the tensor factors cyclically. These
maps allow us to derive bilinear maps $V_2\times V_3\to V_1$ and
$V_3\times V_1\to V_2$ from a given bilinear map $V_1\times V_2\to 
V_3$. In our notation, bilinear maps are adorned with the same index
as the target space.

\begin{prop}
  \label{prop:defnder}
  Let $*_3\colon V_1\times V_2\to V_3$ be a bilinear map, and let
  $*_1=\partial(*_3)$ and $*_2=\partial^2(*_3)$ be the derived maps
  \[
    *_1\colon V_2\times V_3\to V_1,\qquad
    *_2\colon V_3\times V_1\to V_2.
  \]
  The maps $*_1$ and $*_2$ are uniquely determined by the following
  property: for all $x_1\in V_1$, $x_2\in V_2$, $x_3\in V_3$,
  \begin{equation}
    \label{eq:comp45}
    b_1(x_1,\,x_2*_1x_3) = b_2(x_2,\,x_3*_2x_1) = b_3(x_3,\,x_1*_3x_2).
  \end{equation}
\end{prop}

\begin{proof}
  Uniqueness is clear because the forms $b_1$ and $b_2$ are
  nonsingular. By linearity, it suffices to prove~\eqref{eq:comp45} in
  the case where $*_3=v_3\otimes v_2\otimes v_1$ for some $v_1\in
  V_1$, $v_2\in V_2$, $v_3\in V_3$. Then $*_1=v_1\otimes v_3\otimes
  v_2$ and $*_2=v_2\otimes v_1\otimes v_3$, and each of the terms
  in~\eqref{eq:comp45} is equal to $b_1(v_1,x_1)b_2(v_2,x_2)b_3(v_3,x_3)$.
\end{proof}

The bilinear maps of interest in this work satisfy the following
multiplicativity condition:

\begin{definition}
  \label{def:compmap}
  A \emph{composition map} $*_3\colon V_1\times V_2\to V_3$ is a
  bilinear map subject to
  \begin{equation}
    \label{eq:comp1}
    q_3(x_1*_3x_2) = q_1(x_1)q_2(x_2) \qquad\text{for $x_1\in V_1$ and
      $x_2\in V_2$.}
  \end{equation}
\end{definition}

Even though this notion makes sense---and has been studied for instance in
\cite[Chap.~14]{Shapiro}---when the dimensions of $V_1$, $V_2$ and
$V_3$ are not the same, we will always assume in the sequel
that $\dim V_1=\dim V_2=\dim V_3$.

\begin{prop}
  \label{prop:comp1}
  Let $*_3\colon V_1\times V_2\to V_3$ be a composition map, with
  $\dim V_1=\dim V_2=\dim V_3$.
  The derived bilinear maps $*_1$ and $*_2$ are composition maps, i.e., for
  all $x_1\in V_1$, $x_2\in V_2$, $x_3\in V_3$, 
  \begin{equation}
    \label{eq:compp0}
    q_1(x_2*_1x_3)=q_2(x_2)q_3(x_3) \qquad\text{and}\qquad
    q_2(x_3*_2x_1)=q_3(x_3)q_1(x_1).
  \end{equation}
  Moreover, the following relations hold for all $x_1$, $y_1\in V_1$,
  $x_2$, $y_2\in V_2$, $x_3$, $y_3\in V_3$:
  \begin{align}
    \label{eq:comp2}
    b_3(x_1*_3x_2, x_1*_3y_2) & = q_1(x_1) b_2(x_2,y_2),\\
    \label{eq:comp3}
    b_3(x_1*_3x_2,y_1*_3x_2) & = b_1(x_1,y_1) q_2(x_2),\\
    \label{eq:compp1}
    b_1(x_2*_1x_3,x_2*_1y_3) & = q_2(x_2)b_3(x_3,y_3),\\
    \label{eq:compp2}
    b_1(x_2*_1x_3,y_2*_1x_3) & = b_2(x_2,y_2)q_3(x_3),\\
    \label{eq:compp3}
    b_2(x_3*_2x_1,x_3*_2y_1) & = q_3(x_3)b_1(x_1,y_1),\\
    \label{eq:compp4}
    b_2(x_3*_2x_1,y_3*_2x_1) & = b_3(x_3,y_3)q_1(x_1),
  \end{align}
  \begin{align}
    \label{eq:compp5}
    (x_1*_3x_2)*_2x_1 =x_2q_1(x_1) \qquad&\text{and}\qquad
                                          x_2*_1(x_1*_3x_2)
                                          =x_1q_2(x_2),\\
    \label{eq:compp6}
    (x_2*_1x_3)*_3x_2 = x_3q_2(x_2) \qquad&\text{and}\qquad
                                            x_3*_2(x_2*_1x_3) =
                                            x_2q_3(x_3),\\
    \label{eq:compp7}
    (x_3*_2x_1)*_1x_3 = x_1q_3(x_3) \qquad&\text{and}\qquad
                                            x_1*_3(x_3*_2x_1) =
                                            x_3q_1(x_1),
  \end{align}
  \begin{align}
    \label{eq:compp5lin}
    (x_1*_3x_2)*_2y_1+(y_1*_3x_2)*_2x_1 & = x_2b_1(x_1,y_1),\\
    \label{eq:compp5linbis}
    x_2*_1(x_1*_3y_2)+y_2*_1(x_1*_3x_2) &= x_1b_2(x_2,y_2),\\
    \label{eq:compp6lin}
    (x_2*_1x_3)*_3y_2+(y_2*_1x_3)*_3x_2 & = x_3b_2(x_2,y_2),\\
    \label{eq:compp6linbis}
           x_3*_2(x_2*_1y_3)+y_3*_2(x_2*_1x_3) & = x_2b_3(x_3,y_3),\\
    \label{eq:compp7lin}
    (x_3*_2x_1)*_1y_3+(y_3*_2x_1)*_1x_3 & = x_1b_3(x_3,y_3),\\
    \label{eq:comp7linbis}
    x_1*_3(x_3*_2y_1)+y_1*_3(x_3*_2x_1) & = x_3b_1(x_1,y_1).
  \end{align}
\end{prop}

\begin{proof}
  First, \eqref{eq:comp2} and \eqref{eq:comp3} are obtained by
  linearizing~\eqref{eq:comp1}. By~\eqref{eq:comp45} and
  \eqref{eq:comp2} we have for $x_1\in V_1$ and $x_2$, $y_2\in V_2$
  \[
    b_2\bigl((x_1*_3x_2)*_2x_1,y_2\bigr) = b_3(x_1*_3x_2, x_1*_3y_2) =
    q_1(x_1)b_2(x_2,y_2).
  \]
  Since $b_2$ is nonsingular, it follows that $(x_1*_3x_2)*_2x_1 =
  x_2q_1(x_1)$. Similarly, \eqref{eq:comp45} and \eqref{eq:comp3} yield
  \[
    b_1\bigl(y_1,x_2*_1(x_1*_3x_2)\bigr) = b_3(y_1*_3x_2,x_1*_3x_2) =
    b_1(y_1,x_1)q_2(x_2) \quad\text{for all $y_1\in V_1$,}
  \]
  hence $x_2*_1(x_1*_3x_2)=x_1q_2(x_2)$. We thus
  obtain~\eqref{eq:compp5}; then \eqref{eq:compp5lin},
  \eqref{eq:compp5linbis} follow by linearization.

  The main part of the proof consists in
  proving~\eqref{eq:compp0}. For this, fix an anisotropic vector
  $x_2\in V_2$. The map $r_{x_2}\colon
  V_1\to V_3$ defined by $r_{x_2}(x_1) = x_1*_3x_2$ is injective, for
  $x_1*_3x_2=0$ implies $x_1=0$ by~\eqref{eq:compp5}. Since
  $\dim V_1=\dim V_3$ the map $r_{x_2}$ is also surjective, hence
  every $x_3\in V_3$ can be written as $x_3=x_1*_3x_2$ for some
  $x_1\in V_1$. Then by~\eqref{eq:compp5}
  \[
    x_2*_1x_3=x_2*_1(x_1*_3x_2)=x_1q_2(x_2),
  \]
  hence
  \[
    q_1(x_2*_1x_3)=q_1(x_1)q_2(x_2)^2.
  \]
  But since $x_3=x_1*_3x_2$ it follows from~\eqref{eq:comp1} that
  $q_3(x_3)=q_1(x_1)q_2(x_2)$, hence the right side of the last
  displayed equation can be rewritten as
  $q_2(x_2)q_3(x_3)$. We have thus proven
  $q_1(x_2*_1x_3)=q_2(x_2)q_3(x_3)$ when 
  $x_2$ is anisotropic. Moreover, by~\eqref{eq:compp5} we have for all
  $z_2\in V_2$
  \[
    b_1(x_2*_1x_3, z_2*_1x_3) = b_1\bigl(x_2*_1(x_1*_3x_2),
    z_2*_1(x_1*_3x_2)\bigr)= q_2(x_2)b_1\bigl(x_1,
    z_2*_1(x_1*_3x_2)\bigr).
  \]
  By~\eqref{eq:comp45} and \eqref{eq:comp2},
  \[
    b_1\bigl(x_1,z_2*_1(x_1*_3x_2)\bigr) = b_3(x_1*_3z_2, x_1*_3x_2)
    =q_1(x_1)b_2(z_2,x_2),
  \]
  hence, as $q_1(x_1)q_2(x_2)=q_3(x_3)$,
  \begin{equation}
    \label{eq:comp6}
    b_1(x_2*_1x_3,z_2*_1x_3)=b_2(x_2,z_2)q_3(x_3).
  \end{equation}

  Now, assume $x_2$ is isotropic. Pick anisotropic vectors $y_2$,
  $z_2\in V_2$ such that $x_2=y_2+z_2$. (If $\dim V_2>2$, we may pick
  any anisotropic $y_2$ orthogonal to $x_2$ and let $z_2=x_2-y_2$.) By
  the first part of the proof we have
  \[
    q_1(y_2*_1x_3)=q_2(y_2)q_3(x_3) \qquad\text{and}\qquad
    q_1(z_2*_1x_3)=q_2(z_2)q_3(x_3).
  \]
  Moreover, \eqref{eq:comp6} yields
  \[
    b_1(y_2*_1x_3,z_2*_1x_3)= b_2(y_2,z_2)q_3(x_3).
  \]
  Therefore,
  \begin{multline*}
    q_1(x_2*_1x_3) = q_1(y_2*_1x_3) + b_1(y_2*_1x_3, z_2*_1x_3)
    +q_1(z_2*_1x_3) \\
    = q_2(y_2)q_3(x_3) + b_2(y_2,z_2)q_3(x_3) + q_2(z_2)q_3(x_3) =
    q_2(x_2)q_3(x_3).
  \end{multline*}
  Thus, the equation $q_1(x_2*_1x_3)=q_2(x_2)q_3(x_3)$ is proved for
  all $x_2\in V_2$ and  $x_3\in V_3$. The proof of $q_2(x_3*_2x_1) =
  q_3(x_3)q_1(x_1)$ for all $x_3\in V_3$, $x_1\in V_1$ is similar,
  using bijectivity of the map $\ell_{x_1}\colon V_2\to V_3$ carrying
  $x_2$ to $x_1*_3x_2$ for $x_1$ anisotropic. This completes the proof
  of~\eqref{eq:compp0}, and \eqref{eq:compp1}, \eqref{eq:compp2},
  \eqref{eq:compp3}, \eqref{eq:compp4} follow by linearization.

  The same arguments that gave~\eqref{eq:compp5} from~\eqref{eq:comp2}
  and \eqref{eq:comp3} yield~\eqref{eq:compp6} from~\eqref{eq:compp1}
  and \eqref{eq:compp2}, and also \eqref{eq:compp7}
  from~\eqref{eq:compp3} and \eqref{eq:compp4}. The
  relations~\eqref{eq:compp6lin} and \eqref{eq:compp6linbis} (resp.\
  \eqref{eq:compp7lin} and \eqref{eq:comp7linbis}) are derived by
  linearizing \eqref{eq:compp6} (resp.\ \eqref{eq:compp7}).
\end{proof}

Our main object of study in this section is defined next.

\begin{definition}
  \label{def:compofqs}
  A \emph{composition of quadratic spaces} over $F$ is a $4$-tuple
  \[
    \Comp=\bigl((V_1,q_1),\,(V_2,q_2),\,(V_3,q_3),\,*_3\bigr)
  \]
  where $(V_1,q_1)$, $(V_2,q_2)$, $(V_3,q_3)$ are nonsingular
  quadratic spaces of the same dimension over $F$ and $*_3\colon
  V_1\times V_2\to V_3$ is a composition map. We write $\dim\Comp=n$
  if $\dim V_1=\dim V_2=\dim V_3=n$.

In view of Proposition~\ref{prop:comp1}, each composition of quadratic
spaces $\Comp$ yields
  \emph{derived compositions of quadratic spaces} $\partial\Comp$ and
  $\partial^2\Comp$ defined by
  \[
    \partial\Comp=\bigl((V_2,q_2),\,(V_3,q_3),\,(V_1,q_1),\,*_1\bigr)
  \]
  and
  \[
    \partial^2\Comp=\bigl((V_3,q_3),\,(V_1,q_1),\,(V_2,q_2),\,*_2\bigr).
  \]
  The composition maps $*_1$ and $*_2$ are called the \emph{derived
    composition maps} of $*_3$. Since $\partial$ is a cyclic operation
  of period~$3$, we have
\[
  \partial(\partial\Comp)=\partial^2\Comp,\qquad
  \partial^2(\partial\Comp)=\Comp=\partial(\partial^2\Comp),\qquad
  \partial^2(\partial^2\Comp)=\partial\Comp.
\]
\end{definition}

\begin{examples}
  \label{ex:compos}
  \begin{enumerate}
  \item[(1)]
    Let $A$ be either $F$, a quadratic \'etale $F$-algebra, a
    quaternion $F$-algebra or an octonion $F$-algebra, and let
    $n_A\colon A\to F$ be (respectively) the squaring map, the norm,
    the quaternion (reduced) norm or the octonion norm. Assuming
    $\charac F\neq2$ if $A=F$, we know from the properties of these
    algebras that
    multiplication in $A$ defines a composition of quadratic spaces
    \[
      \Comp=\bigl((A,n_A),\, (A,n_A),\, (A,n_A),\,*_3\bigr).
    \]
    This particular type of composition is discussed in
    \S\ref{subsec:compalg} in 
    relation with composition algebras. Note that if $A\neq F$ the
    derived composition maps $*_1$ and $*_2$ are \emph{not} simply
    given by the multiplication in $A$; see
    Proposition~\ref{prop:dercompunit}. 
  \end{enumerate}
  The following examples are obtained in relation with a Galois
  $F$-algebra $L$ with elementary abelian Galois group
  $\{1,\,\sigma_1,\,\sigma_2,\,\sigma_3\}$, i.e., an \'etale
  biquadratic $F$-algebra.
  \begin{enumerate}
  \item[(2)]
    Assume $\charac F\neq2$, and for $i=1$, $2$, $3$ let $V_i$ denote
    the following $1$-dimensional subspace of $L$:
    \[
      V_i=\{x_i\in L\mid \sigma_j(x_i)=-x_i\text{ for $j\neq i$}\}.
    \]
    Define $q_i\colon V_i\to F$ by $q_i(x_i)=x_i^2$. For $x_1\in V_1$
    and $x_2\in V_2$ we have $x_1x_2\in V_3$ and
    $(x_1x_2)^2=x_1^2x_2^2$, hence multiplication in $L$ defines a
    composition map $*_3\colon V_1\times V_2\to V_3$. The derived
    composition maps $*_1$ and $*_2$ are also given by the
    multiplication in $L$.
  \item[(3)]
    Let $A$ be a central simple $F$-algebra of degree~$4$ containing
    $L$. Assume $\charac F\neq2$ and $F$ contains an element $\zeta$
    such that $\zeta^2=-1$. For $i=1$, $2$, $3$, define
    \[
      V_i=\{x_i\in A\mid x_i\ell=\sigma_i(\ell)x_i\text{ for all
        $\ell\in L$}\}.
    \]
    The $F$-vector space $V_i$ has dimension~$4$ and carries a
    quadratic form $q_i$ given by $q_i(x_i)=\Trd_A(x_i^2)$, where
    $\Trd_A$ is the reduced trace. It is shown in~\cite{RST} that the
    following formula defines a composition map $*_3\colon V_1\times
    V_2\to V_3$:
    \[
      x_1*_3x_2 = (1+\zeta)x_1x_2+(1-\zeta)x_2x_1.
    \]
    The derived maps are given by similar formulas:
    \begin{align*}
      x_2*_1x_3 &= (1+\zeta)x_2x_3+(1-\zeta)x_3x_2\qquad\text{for
                  $x_2\in V_2$ and $x_3\in V_3$},\\
      x_3*_2x_1 &= (1+\zeta)x_3x_1+(1-\zeta)x_1x_3\qquad\text{for
                  $x_3\in V_3$ and $x_1\in V_1$}.
    \end{align*}
    A characteristic~$2$ version of these composition maps is given
    in~\cite{T}.
  \item[(4)]
    Compositions of dimension~$8$ from central simple algebras with
    symplectic involution of degree~$8$ are given in a similar way
    in~\cite{BGBT}. 
  \end{enumerate}
\end{examples}

\subsection{Canonical Clifford maps}
\label{subsec:Climaps}

Our goal in this subsection is to obtain structural information on the
quadratic spaces for which a composition exists. This information will
be derived from algebra homomorphisms defined on Clifford and even
Clifford algebras.

Throughout this subsection, we fix a composition of quadratic spaces
\[
  \Comp=\bigl((V_1,q_1),\,(V_2,q_2),\,(V_3,q_3),\,*_3\bigr)
\]
and we let $*_1$ and $*_2$ denote the derived composition maps of
$*_3$, as per Definition~\ref{def:compofqs}.
For each $x_1\in V_1$ we may consider two linear maps
\[
  \ell_{x_1}\colon V_2\to V_3, \quad x_2\mapsto x_1*_3x_2
  \qquad\text{and}\qquad
  r_{x_1}\colon V_3\to V_2,\quad x_3\mapsto x_3*_2x_1.
\]
By \eqref{eq:compp5} and \eqref{eq:compp7} we have
\[
  \ell_{x_1}\circ r_{x_1} = q_1(x_1)\Id_{V_3} \qquad\text{and}\qquad
  r_{x_1}\circ\ell_{x_1} = q_1(x_1)\Id_{V_2}.
\]
Therefore, the linear map
\[
  \alpha\colon V_1 \to \End(V_2\oplus V_3), \qquad
  x_1\mapsto
  \begin{pmatrix}
    0&r_{x_1}\\ \ell_{x_1}&0
  \end{pmatrix}
\]
extends to an $F$-algebra homomorphism defined on the Clifford algebra
$C(V_1,q_1)$: 
\[
  C(\alpha)\colon C(V_1,q_1) \to \End(V_2\oplus V_3).
\]
The image of the even Clifford algebra $C_0(V_1,q_1)$ lies in the
diagonal subalgebra, hence $C(\alpha)$ restricts to an $F$-algebra
homomorphism
\[
  C_0(\alpha)\colon C_0(V_1,q_1) \to (\End V_2)\times (\End V_3).
\]
We write $\tau_1$ for the involution on $C(V_1,q_1)$ that leaves every
vector in $V_1$ fixed, and $\tau_{01}$ for the restriction of $\tau_1$
to $C_0(V_1,q_1)$. We let $\sigma_{b_2\perp b_3}$
(resp.\ $\sigma_{b_2}$, resp.\ $\sigma_{b_3}$) denote the involution
on $\End(V_2\oplus V_3)$ (resp.\ $\End V_2$, resp.\ $\End V_3$)
adjoint to $b_2\perp b_3$ (resp.\ $b_2$, resp.\ $b_3$).

\begin{thm}
  \label{thm:dimcomp}
  The maps $C(\alpha)$ and $C_0(\alpha)$ are homomorphisms of algebras
  with involution
  \[
  \begin{array}{rl}
    C(\alpha)\colon& (C(V_1,q_1),\tau_1) \to
                     (\End(V_2\oplus V_3),\sigma_{b_2\perp b_3}),\\
    C_0(\alpha)\colon& (C_0(V_1,q_1),\tau_{01}) \to
    (\End V_2,\sigma_{b_2}) \times (\End V_3,\sigma_{b_3}).
  \end{array}
  \]
  Moreover, $\dim\Comp=1$, $2$, $4$ or $8$.
\end{thm}

\begin{proof}
  For the first part,
  it suffices to show that for $x_1\in V_1$,
  \[
    \sigma_{b_2\perp b_3}
    \begin{pmatrix}
      0&r_{x_1}\\ \ell_{x_1}&0
    \end{pmatrix}
    =
    \begin{pmatrix}
      0&r_{x_1}\\ \ell_{x_1}&0
    \end{pmatrix}.
  \]
  This amounts to proving that for $x_1\in V_1$, $x_2$, $y_2\in V_2$
  and $x_3$, $y_3\in V_3$
  \[
    b_2(x_3*_2x_1,y_2) + b_3(x_1*_3x_2,y_3) = b_2(x_2,y_3*_2x_1) +
    b_3(x_3,x_1*_3y_2),
  \]
  which follows from~\eqref{eq:comp45}.

  To determine the various options for $\dim\Comp$, observe that
  the map $C(\alpha)$ endows $V_2\oplus V_3$ with a structure of left 
  $C(V_1,q_1)$-module; similarly, $V_2$ and $V_3$ are left modules over
  $C_0(V_1,q_1)$ through $C_0(\alpha)$. This observation yields
  restrictions on the dimensions of $V_2$ and $V_3$, because the
  dimension of a left module over a central simple algebra $A$ is
  a multiple of $(\deg A)(\ind A)$, where $\deg A$ is the degree of
  $A$ and $\ind A$ is its (Schur) index.

  Let $n=\dim\Comp$. If $n$ is even, then $C(V_1,q_1)$ is a central
  simple $F$-algebra, 
  and $V_2\oplus V_3$ is a left module over $C(V_1,q_1)$ through
  $C(\alpha)$, hence $(\deg C(V_1,q_1))(\ind C(V_1,q_1))$ divides
  $2n$. Since $\deg C(V_1,q_1) =
  2^{n/2}$, it follows that $2^{n/2}$ divides $2n$, hence $n=2$, $4$ or
  $8$.

  If $n$ is odd, the even Clifford algebra
  $C_0(V_1,q_1)$ is central simple over $F$, and $V_2$ is a left
  module over $C_0(V_1,q_1)$ through $C_0(\alpha)$, hence $\dim V_2$
  is a multiple of $(\deg C_0(V_1,q_1))(\ind C_0(V_1,q_1))$. As $\deg
  C_0(V_1,q_1)=2^{(n-1)/2}$, this means that $2^{(n-1)/2} \ind
  C_0(V_1,q_1)$ divides~$n$. As $n$ is assumed to be odd, we must have
  $n=1$. 
\end{proof}

Mimicking the construction above, we attach to the derived
compositions $\partial\Comp$ and $\partial^2\Comp$ linear maps
\[
  \alpha'\colon V_2 \to \End(V_3\oplus V_1), \qquad
  x_2\mapsto
  \begin{pmatrix}
    0&r_{x_2}\\ \ell_{x_2}&0
  \end{pmatrix}
\]
and
\[
  \alpha''\colon V_3 \to \End(V_1\oplus V_2), \qquad
  x_3\mapsto
  \begin{pmatrix}
    0&r_{x_3}\\ \ell_{x_3}&0
  \end{pmatrix}.
\]
These maps yield homomorphisms
\begin{equation}
  \label{eq:alpha'}
  \begin{aligned}
  C(\alpha')& \colon (C(V_2,q_2),\tau_2)\to (\End(V_3\oplus V_1),
  \sigma_{b_3\perp b_1}), \\
  C_0(\alpha')& \colon (C_0(V_2,q_2), \tau_{02}) \to (\End V_3,
  \sigma_{b_3}) \times (\End V_1, \sigma_{b_1})
  \end{aligned}
\end{equation}
and
\begin{equation}
  \label{eq:alpha''}
  \begin{aligned}
  C(\alpha'')& \colon (C(V_3,q_3),\tau_3)\to (\End(V_1\oplus V_2),
  \sigma_{b_1\perp b_2}), \\
  C_0(\alpha'')& \colon (C_0(V_3,q_3), \tau_{03}) \to (\End V_1,
  \sigma_{b_1}) \times (\End V_2, \sigma_{b_2}).
  \end{aligned}
\end{equation}
\medbreak

We next take a closer look at compositions of the various degrees. If
$\dim \Comp=1$, then $\charac F\neq2$ since odd-dimensional quadratic
forms are singular in characteristic~$2$. If $q_1$ represents
$\lambda_1\in F^\times$ and $q_2$ represents $\lambda_2\in F^\times$,
then by multiplicativity $q_3$ represents $\lambda_1\lambda_2\in
F^\times$, hence also $(\lambda_1\lambda_2)^{-1}$. Thus in this case
there exist $\lambda_1$, $\lambda_2$, $\lambda_3\in F^\times$ such
that $\lambda_1\lambda_2\lambda_3=1$ and
\[
  q_1\simeq\qf{\lambda_1},\qquad q_2\simeq\qf{\lambda_2}, \qquad
  q_3\simeq\qf{\lambda_3},
\]
and $\qf{1}\perp q_1\perp q_2\perp q_3$ is a $2$-fold Pfister form. We
will mostly ignore this easy case. (See however
Example~\ref{ex:compos}(2).) 

\begin{prop}
  \label{prop:compdim2}
  Let $\dim\Comp=2$. There exists a $1$-fold Pfister form $n_\Comp$,
  uniquely determined up to isometry, and scalars $\lambda_1$,
  $\lambda_2$, $\lambda_3\in F^\times$ such that
  $\lambda_1\lambda_2\lambda_3=1$ and
  \[
    q_1\simeq\qf{\lambda_1}n_\Comp, \qquad
    q_2\simeq\qf{\lambda_2}n_\Comp, \qquad
    q_3\simeq\qf{\lambda_3}n_\Comp.
  \]
  The form $n_\Comp\perp q_1\perp q_2\perp q_3$ is a $3$-fold Pfister
  form canonically associated to $\Comp$ up to isometry.
\end{prop}

\begin{proof}
  Since $\dim V_1=2$, we have $q_1\simeq \qf{\lambda_1}n_\Comp$ for some
  $\lambda_1\in F^\times$ and 
  some uniquely determined $1$-fold Pfister form $n_\Comp$. For any
  anisotropic $y_2\in V_2$, 
  the map $r_{y_2}\colon V_1\to V_3$ carrying $x_1$ to $x_1*_3y_2$ is
  a similitude with multiplier $q_2(y_2)$ by \eqref{eq:comp1}, hence
  $q_3\simeq\qf{q_2(y_2)}q_1$. Similarly, for any anisotropic $y_1\in
  V_1$ the map $\ell_{y_1}\colon V_2\to V_3$ is a similitude with
  multiplier $q_1(y_1)$, hence $q_3\simeq\qf{q_1(y_1)}q_2$. Therefore,
  \[
    q_3\simeq\qf{\lambda_1q_2(y_2)}n_\Comp
    \qquad\text{and}\qquad
    q_2\simeq\qf{q_1(y_1)}q_3\simeq \qf{\lambda_1
      q_1(y_1)q_2(y_2)}n_\Comp .
  \]
  Now, $\lambda_1 q_1(y_1)$ is represented by $n_\Comp$ since
  $q_1\simeq\qf{\lambda_1}n_\Comp$, hence
  $\qf{\lambda_1q_1(y_1)}n_\Comp\simeq 
  n_\Comp$. Letting $\lambda_2=q_2(y_2)$ and
  $\lambda_3=(\lambda_1q_2(y_2))^{-1}$, we then have
  $\lambda_1\lambda_2\lambda_3=1$ and
  \[
    q_1\simeq\qf{\lambda_1}n_\Comp, \qquad q_2\simeq\qf{\lambda_2}n_\Comp,
    \qquad q_3\simeq\qf{\lambda_3}n_\Comp.
    \qedhere
  \]
\end{proof}

\begin{prop}
  \label{prop:compdim4}
  Let $\dim\Comp=4$. There exists a $2$-fold quadratic Pfister form
  $n_\Comp$, 
  uniquely determined up to isometry, and scalars $\lambda_1$,
  $\lambda_2$, $\lambda_3\in F^\times$ such that
  $\lambda_1\lambda_2\lambda_3=1$ and
  \[
  q_1\simeq\qf{\lambda_1}n_\Comp, \qquad q_2\simeq\qf{\lambda_2}n_\Comp,
  \qquad q_3\simeq\qf{\lambda_3}n_\Comp.
  \]
  The form $n_\Comp\perp q_1\perp q_2\perp q_3$ is a $4$-fold Pfister
  form canonically associated to $\Comp$ up to isometry.
\end{prop}

\begin{proof}
  Consider the homomorphisms of algebras with involution induced by
  $C_0(\alpha)$:
  \[
    \varphi_2\colon (C_0(V_1,q_1),\tau_{01}) \to
    (\End V_2,\sigma_{b_2}) \qquad\text{and}\qquad
    \varphi_3\colon (C_0(V_1,q_1),\tau_{01})\to
    (\End V_3,\sigma_{b_3}).
  \]
  If $Z$ is a
  field, then $C_0(V_1,q_1)$ is simple and its image under $\varphi_2$
  is the centralizer in $\End V_2$ of a separable quadratic subfield 
  fixed under $\sigma_{b_2}$. But the
  restriction of $\sigma_{b_2}$ to such a centralizer is an orthogonal
  involution (see \cite[(4.12)]{BoI}), whereas $\tau_{01}$ is
  symplectic, so this case is impossible. Therefore, $Z$ is not a
  field, which
  means that the discriminant (or Arf invariant) of $q_1$ is
  trivial. It follows that $q_1$ is a 
  multiple of some uniquely determined $2$-fold Pfister form $n_\Comp$.
  The same arguments as in the proof of
  Proposition~\ref{prop:compdim2} show 
  that there exist $\lambda_1$, $\lambda_2$, $\lambda_3\in F^\times$
  such that $q_i\simeq\qf{\lambda_i}n_\Comp$ for $i=1$, $2$, $3$.
\end{proof}

Finally, we consider the case where $\dim\Comp=8$. Recall
from~\S\ref{subsec:ClAlg} that in this case the Clifford algebra
$C(V_1,q_1)$ and the even Clifford algebra $C_0(V_1,q_1)$ carry
canonical quadratic pairs. We use for these quadratic pairs the
notation $(\tau_1,\str{g}_1)$ and $(\tau_{01}, \str{g}_{01})$
respectively. 

\begin{prop}
  \label{prop:compdim8}
  Let $\dim\Comp=8$. The canonical maps $C(\alpha)$ and $C_0(\alpha)$
  are isomorphisms of algebras with quadratic pair
  \[
  \begin{array}{rl}
    C(\alpha)\colon& (C(V_1,q_1),\tau_1,\str g_1) \xrightarrow{\sim}
                     (\End(V_2\oplus V_3),\sigma_{b_2\perp
                     b_3},\strf_{q_2\perp q_3}),\\
    C_0(\alpha)\colon& (C_0(V_1,q_1),\tau_{01},\str g_{01})
                       \xrightarrow{\sim} 
    (\End V_2,\sigma_{b_2},\strf_{q_2}) \times
                       (\End V_3,\sigma_{b_3},\strf_{q_3}). 
  \end{array}
  \]
  Moreover, there exists a $3$-fold quadratic Pfister form $n_\Comp$,
  uniquely determined up to isometry, and scalars $\lambda_1$,
  $\lambda_2$, $\lambda_3\in F^\times$ such that
  $\lambda_1\lambda_2\lambda_3=1$ and
  \[
  q_1\simeq\qf{\lambda_1}n_\Comp, \qquad q_2\simeq\qf{\lambda_2}n_\Comp,
  \qquad q_3\simeq\qf{\lambda_3}n_\Comp.
  \]
  The form $n_\Comp\perp q_1\perp q_2\perp q_3$ is a $5$-fold Pfister
  form canonically associated to $\Comp$ up to isometry.  
\end{prop}

\begin{proof}
  In this case we have $\dim C(V_1,q_1) = \dim \End(V_2\oplus V_3)$. Since
  the algebra $C(V_1,q_1)$ is simple, it follows that $C(\alpha)$ is
  an isomorphism, hence $C(V_1,q_1)$ is split. Moreover,
  $C_0(\alpha)$ also is an isomorphism, hence the center of
  $C_0(V_1,q_1)$ is isomorphic to $F\times F$, and therefore the
  discriminant (or Arf invariant) of $q_1$ is trivial. It follows that
  $q_1$ is a multiple of some uniquely determined $3$-fold Pfister
  form $n_\Comp$, and the existence of $\lambda_1$, $\lambda_2$,
  $\lambda_3\in F^\times$ such that $q_i\simeq\qf{\lambda_i}n_\Comp$ for
  $i=1$, $2$, $3$ is proved as in the case where $\dim\Comp=2$ (see
  Proposition~\ref{prop:compdim2}).

  Since we already know from Theorem~\ref{thm:dimcomp} that
  $C(\alpha)$ and $C_0(\alpha)$ are homomorphisms of algebras with
  involution, it only remains to see that these maps also preserve the
  semitraces.

  The arguments in each case are similar. For $C(\alpha)$ we have to
  show that
  \[
    \strf_{q_2\perp q_3}\bigl(C(\alpha)(s)\bigr) = \str g_1(s)
    \qquad\text{for all $s\in\Sym(\tau_1)$.}
  \]
  Fix $e_1$, $e'_1\in V_1$ such that $b_1(e_1,e'_1)=1$. By definition,
  $\str g_1(s)=\Trd_{C(V_1,q_1)}(e_1e'_1s)$. Since isomorphisms of
  central 
  simple algebras preserve reduced traces, we have for all
  $s\in\Sym(\tau_1)$
  \[
    \str g_1(s)=\Trd_{\End(V_2\oplus
      V_3)}\bigl(C(\alpha)(e_1e'_1s)\bigr) = 
    \Trd_{\End(V_2\oplus
      V_3)}\bigl(C(\alpha)(e_1e'_1)\circ C(\alpha)(s)\bigr).
  \]
  Now, $C(\alpha)\bigl(\Sym(\tau_1)\bigr) =
  \Sym(\sigma_{b_2\perp b_3})$ because $C(\alpha)$ is an isomorphism
  of algebras with involution. Therefore, we may rewrite the equation we
  have to prove as
  \[
    \strf_{q_2\perp q_3}(s') =\Trd_{\End(V_2\oplus
      V_3)}(C(\alpha)(e_1e'_1)\circ s') \qquad\text{for all
      $s'\in\Sym(\sigma_{b_2\perp b_3})$}.
  \]
  Using the standard identification $\End(V_2\oplus V_3)=(V_2\oplus
  V_3)\otimes (V_2\oplus V_3)$ set up in \S\ref{subsec:quadpair}, we
  see that $\Sym(\sigma_{b_2\perp b_3})$ is spanned by elements of the form
  $(x_2+x_3)\otimes(x_2+x_3)$ with $x_2\in V_2$ and $x_3\in
  V_3$, and that for $s'=(x_2+x_3)\otimes(x_2+x_3)$
  \[
    C(\alpha)(e_1e'_1)\circ s' =
    \bigl(C(\alpha)(e_1e'_1)(x_2+x_3)\bigr)\otimes (x_2+x_3) =
    \bigl(r_{e_1}\ell_{e'_1}(x_2)+\ell_{e_1}r_{e'_1}(x_3)\bigr)
    \otimes (x_2+x_3).
  \]
  Therefore, it suffices to show that for all $x_2\in V_2$ and $x_3\in
  V_3$
  \begin{equation}
    \label{eq:semitracedim8}
    \strf_{q_2\perp q_3}\bigl((x_2+x_3)\otimes(x_2+x_3)\bigr) =
    \Trd_{\End(V_2\oplus
      V_3)}\bigl((r_{e_1}\ell_{e'_1}(x_2)+\ell_{e_1}r_{e'_1}(x_3))\otimes
    (x_2+x_3)\bigr).
  \end{equation}
  The right side is
  \begin{multline*}
    (b_2\perp b_3)\bigl((e'_1*_3x_2)*_2e_1 + e_1*_3(x_3*_2e'_1),
    x_2+x_3) =
    \\
    b_2\bigl((e'_1*_3x_2)*_2e_1,x_2\bigr) +
    b_3(e_1*_3(x_3*_2e'_1),x_3).
  \end{multline*}
  Now, by~\eqref{eq:comp45} and \eqref{eq:comp2} we have
  \[
    b_2\bigl((e'_1*_3x_2)*_2e_1,x_2\bigr) = b_3(e'_1*_3x_2,e_1*_3x_2) =
    b_1(e'_1,e_1)q_2(x_2)
  \]
  and, similarly,
  \[
    b_3(e_1*_3(x_3*_2e'_1),x_3) = b_2(x_3*_2e'_1,x_3*_2e_1)
    = q_3(x_3)b_1(e'_1,e_1).
  \]
  As $b_1(e_1,e'_1)=1$, it follows that
  \[
     \Trd_{\End(V_2\oplus
      V_3)}\bigl((r_{e_1}\ell_{e'_1}(x_2)+\ell_{e_1}r_{e'_1}(x_3))\otimes
    (x_2+x_3)\bigr) =q_2(x_2)+q_3(x_3).
  \]
  On the other hand, by definition of $\strf_{q_2\perp q_3}$ we have
  \[
    \strf_{q_2\perp q_3}\bigl((x_2+x_3)\otimes(x_2+x_3)\bigr) = (q_2\perp
    q_3)(x_2+x_3) = q_2(x_2)+q_3(x_3).
  \]
  We have thus checked~\eqref{eq:semitracedim8}.

  The proof that $C_0(\alpha)$ also preserves the semitraces is
  obtained by a slight variation of the preceding arguments. We have
  to show that
  \[
    (\strf_{q_2},\strf_{q_3})\bigl(C_0(\alpha)(s)\bigr) =
    C_0(\alpha)\bigl(\str g_{01}(s)\bigr) \qquad
    \text{for all $s\in\Sym(\tau_{01})$.}
  \]
  Since $C_0(\alpha)$ is an isomorphism of algebras with involution,
  this amounts to showing
  \begin{equation}
    \label{eq:semitracedim8bis}
    \bigl(\strf_{q_2}(s'_2),\strf_{q_3}(s'_3)\bigr) =
    \bigl(\Trd_{\End V_2}(r_{e_1}\ell_{e'_1}s'_2),
    \Trd_{\End V_3}(\ell_{e_1}r_{e'_1}s'_3)\bigr)
  \end{equation}
  for all $s'_2\in\Sym(\sigma_{b_2})$, $s'_3\in\Sym(\sigma_{b_3})$.
  It suffices to consider $s'_2$, $s'_3$ of the form $x_2\otimes x_2$,
  $x_3\otimes x_3$ for $x_2\in V_2$, $x_3\in V_3$ under the standard
  identifications $\End V_2=V_2\otimes V_2$, $\End V_3=V_3\otimes
  V_3$. For $s'_2=x_2\otimes x_2$ we have
  \[
    r_{e_1}\ell_{e'_1}s'_2 = \bigl((e'_1*_3x_2)*_2e_1\bigr)\otimes
    x_2,
  \]
  hence
  \[
    \Trd_{\End V_2}(r_{e_1}\ell_{e'_1}s'_2)=
    b_2\bigl((e'_1*_3x_2)*_2e_1, x_2) =
    b_3(e'_1*_3x_2,e_1*_3x_2)=q_2(x_2).
  \]
  On the other hand $\strf_2(x_2\otimes x_2)=q_2(x_2)$ by
  definition. Likewise, for $s'_3=x_3\otimes x_3$
  \[
    \Trd_{\End V_3}(\ell_{e_1}r_{e'_1}s'_3) = q_3(x_3) =
    \strf_{q_3}(s'_3),
  \]
  hence \eqref{eq:semitracedim8bis} is proved.
\end{proof}

\begin{remark}
  \label{rem:polarcomp}
  The map $C_0(\alpha)$ in Proposition~\ref{prop:compdim8} yields an
  isomorphism between the center of $C_0(V_1,q_1)$ and $F\times F$,
  hence also a polarization of $(V_1,q_1)$ (see
  Definition~\ref{defn:orientationqs}): the primitive central
  idempotents $z_+$ and $z_-$ of $C_0(V_1,q_1)$ are such that
  $C_0(\alpha)(z_+)=(1,0)$ and $C_0(\alpha)(z_-)=(0,1)$, so that
  $C_0(\alpha)$ induces homomorphisms
  \[
    C_+(\alpha)\colon C_0(V_1,q_1) \to \End V_2
    \quad\text{and}\quad
    C_-(\alpha)\colon C_0(V_1,q_1) \to \End V_3.
  \]
  Similarly, the maps $C_0(\alpha')$ and $C_0(\alpha'')$
  of~\eqref{eq:alpha'} and \eqref{eq:alpha''} attached to
  $\partial\Comp$ and $\partial^2\Comp$ yield polarizations of
  $(V_2,q_2)$ and $(V_3,q_3)$, so that
  \[
    C_+(\alpha')\colon C_0(V_2,q_2) \to \End V_3
    \quad\text{and}\quad
    C_-(\alpha')\colon C_0(V_2,q_2) \to \End V_1,
  \]
  and
  \[
    C_+(\alpha'')\colon C_0(V_3,q_3) \to \End V_1
    \quad\text{and}\quad
    C_-(\alpha'')\colon C_0(V_3,q_3) \to \End V_2.
  \]
\end{remark}

\begin{corol}
  \label{corol:comp1ter}
  For any composition of quadratic spaces $\Comp$, the following are
  equivalent: 
  \begin{enumerate}
  \item[(i)]
    $q_1\simeq q_2\simeq q_3$;
  \item[(ii)]
    $q_1$, $q_2$ and $q_3$ are Pfister forms;
  \item[(iii)]
    $q_1$, $q_2$ and $q_3$ represent~$1$.
  \end{enumerate}
\end{corol}

\begin{proof}
  According to Propositions~\ref{prop:compdim2}, \ref{prop:compdim4},
  \ref{prop:compdim8}, there exist a quadratic 
  Pfister form $n_\Comp$ and scalars $\lambda_1$, $\lambda_2$, $\lambda_3$
  such that $\lambda_1\lambda_2\lambda_3=1$ and
  $q_1\simeq\qf{\lambda_1}n_\Comp$, $q_2\simeq\qf{\lambda_2}n_\Comp$ 
  and $q_3\simeq\qf{\lambda_3}n_\Comp$. This also holds when
  $\dim\Comp=1$, with $n_\Comp=\qf1$.

  (i)~$\Rightarrow$~(ii) If $q_1\simeq q_2$, then
  $\qf{\lambda_1\lambda_2}n_\Comp\simeq n_\Comp$, hence
  $\qf{\lambda_3}n_\Comp\simeq n_\Comp$. Therefore, $q_3\simeq n_\Comp$.

  (ii)~$\Rightarrow$~(iii) This is clear since Pfister quadratic forms
  represent~$1$.

  (iii)~$\Rightarrow$~(i) For $i=1$, $2$, $3$, if $q_i$
  represents~$1$, then $n_\Comp$ represents $\lambda_i$, hence
  $\qf{\lambda_i}n_\Comp\simeq n_\Comp$.
\end{proof}

\subsection{Similitudes and isomorphisms}
\label{subsec:simiso}

Consider two compositions
$\Comp=\bigl((V_1,q_1),\,(V_2,q_2),\,(V_3,q_3),\,*_3\bigr)$ and 
$\tilde\Comp=\bigl((\tilde V_1,\tilde q_1),\,(\tilde V_2,\tilde
q_2),\,(\tilde V_3,\tilde q_3),\,\tilde *_3\bigr)$
over an  
arbitrary field $F$. As in Definition~\ref{def:compofqs}, we
write $*_1$ and $*_2$ (resp.\ $\tilde *_1$, $\tilde*_2$) for the
derived composition maps of $*_3$ (resp.\ $\tilde *_3$).

\begin{definition}
  \label{defn:simcomp2}
  For $i=1$, $2$, $3$, let $g_i\colon(V_i,q_i)\to(\tilde V_i,\tilde
  q_i)$ be a similitude with multiplier $\mu(g_i)\in F^\times$. 
  The triple $(g_1,g_2,g_3)$ is a \emph{similitude of compositions}
  $\Comp\to\tilde\Comp$ if there exists $\lambda_3\in F^\times$ such
  that 
  \begin{equation}
    \label{eq:simcompdef}
    \lambda_3\,g_3(x_1*_3x_2) = g_1(x_1)\,\tilde*_3\,g_2(x_2)
    \qquad\text{for all $x_1\in V_1$, $x_2\in V_2$.}
  \end{equation}
\end{definition}

\begin{prop}
  \label{prop:simdef2}
  If $g=(g_1,g_2,g_3)$ is a similitude $\Comp\to\tilde\Comp$, then
  $\partial g:=(g_2,g_3,g_1)$ is a similitude
  $\partial\Comp\to\partial\tilde\Comp$ and $\partial^2g:=(g_3,g_1,g_2)$ is
  a similitude $\partial^2\Comp\to\partial^2\tilde\Comp$. Moreover, the
  scalars $\lambda_1$, $\lambda_2$, $\lambda_3\in F^\times$ such that
  for all $x_1\in V_1$, $x_2\in V_2$, $x_3\in V_3$
  \begin{align*}
    \lambda_1\,g_1(x_2*_1x_3) & = g_2(x_2)\,\tilde*_1\,g_3(x_3),\\
    \lambda_2\,g_2(x_3*_2x_1) & = g_3(x_3)\,\tilde*_2\,g_1(x_1),\\
    \lambda_3\,g_3(x_1*_3x_2) & = g_1(x_1)\,\tilde*_3\,g_2(x_2)
  \end{align*}
  are related to the multipliers of $g_1$, $g_2$, $g_3$ by
  \begin{equation}
    \label{eq:lambdamu}
    \mu(g_1)=\lambda_2\lambda_3,\qquad
    \mu(g_2)=\lambda_3\lambda_1,\qquad
    \mu(g_3)=\lambda_1\lambda_2.
  \end{equation}
\end{prop}

\begin{proof}
  For the first part, we have to prove the existence of $\lambda_1\in
  F^\times$ such 
  that
  \[
    \lambda_1\,g_1(x_2*_1x_3)  = g_2(x_2)\,\tilde*_1\,g_3(x_3) \qquad\text{for
      all $x_2\in V_2$ and $x_3\in V_3$.}
  \]
  Multiplying each side of~\eqref{eq:simcompdef} on the left by
  $g_2(x_2)$, we
  obtain for all $x_1\in V_1$ and $x_2\in V_2$
  \[
    \lambda_3\,g_2(x_2)\,\tilde*_1\,g_3(x_1*_3x_2) = \tilde
    q_2\bigl(g_2(x_2)\bigr) 
    g_1(x_1) = \mu(g_2)q_2(x_2)g_1(x_1).
  \]
  If $x_2$ is anisotropic, then $r_{x_2}\colon V_1\to V_3$ is
  bijective with inverse $q_2(x_2)^{-1}\ell_{x_2}$, hence every
  $x_3\in V_3$ can be written as $x_3=x_1*_3x_2$ with
  $x_1=q_2(x_2)^{-1}x_2*_1x_3$. Substituting in the last displayed
  equation, we obtain for 
  $x_2\in V_2$ anisotropic and $x_3\in V_3$
  \[
    g_2(x_2)\,\tilde*_1\,g_3(x_3) = \mu(g_2)\lambda_3^{-1}g_1(x_2*_1x_3).
  \]
  Since anisotropic vectors span $V_2$, this relation holds for all
  $x_2\in V_2$ and $x_3\in V_3$. Therefore, $(g_2,g_3,g_1)$ is a
  similitude of compositions, with scalar
  $\lambda_1=\mu(g_2)\lambda_3^{-1}$.

  Applying the same arguments to $\partial g$ instead of $g$, we see
  that $\partial(\partial g)=\partial^2g$ is a similitude
  $\partial^2\Comp\to\partial^2\tilde\Comp$, with scalar
  $\lambda_2=\mu(g_3)\lambda_1^{-1}$. Applying the arguments one more
  time, we obtain that $g$ is a similitude $\Comp\to\tilde\Comp$ with
  scalar $\mu(g_1)\lambda_2^{-1}$, hence
  $\lambda_3=\mu(g_1)\lambda_2^{-1}$ and the proof is complete.
\end{proof}

\begin{definition}
  \label{defn:compmult}
  In the situation of Proposition~\ref{prop:simdef2}, the triple
  $(\lambda_1, \lambda_2,\lambda_3)\in F^\times\times F^\times\times
  F^\times$ is said to be the 
  \emph{composition multiplier} of the similitude of compositions
  $g\colon\Comp\to\tilde\Comp$, and we write
  \[
    \lambda(g) = (\lambda_1,\lambda_2,\lambda_3),
  \]
  hence $\lambda(\partial g) = (\lambda_2,\lambda_3,\lambda_1)$ and
  $\lambda(\partial^2g) = (\lambda_3,\lambda_1,\lambda_2)$.
  Writing $\rho(g)=\lambda_1\lambda_2\lambda_3$, we thus have
  by~\eqref{eq:lambdamu}
  \[
    \lambda(g) = (\rho(g)\mu(g_1)^{-1},\rho(g)\mu(g_2)^{-1},
    \rho(g)\mu(g_3)^{-1}) \quad\text{and}\quad
    \mu(g_1)\mu(g_2)\mu(g_3)=\rho(g)^2.
  \]

  Similitudes with
  composition multiplier $(1,1,1)$ are called \emph{isomorphisms} of
  compositions. 
\end{definition}

\begin{prop}
  \label{prop:simtoPfistercomp}
  Every composition of quadratic spaces is similar to a composition of
  isometric quadratic spaces.
\end{prop}

\begin{proof}
  Let $\Comp=\bigl((V_1,q_1),\,(V_2,q_2),\,(V_3,q_3),\,*_3\bigr)$ be
  an arbitrary composition of quadratic spaces. Let $\lambda_1\in
  F^\times$ (resp.\ $\lambda_2\in F^\times$) be a value represented by
  $q_1$ (resp.\ $q_2$) and let
  $\lambda_3=\lambda_1^{-1}\lambda_2^{-1}\in F^\times$. Then
  $\lambda_3$ is represented by $q_3$; define quadratic forms $\tilde
  q_1$, $\tilde q_2$, $\tilde q_3$ on $V_1$, $V_2$, $V_3$ by
  \[
    \tilde q_1(x_1)=\lambda_1^{-1}q_1(x_1),\qquad
    \tilde q_2(x_2)=\lambda_2^{-1}q_2(x_2),\qquad
    \tilde q_3(x_3) = \lambda_3^{-1}q_3(x_3)
  \]
  for $x_1\in V_1$, $x_2\in V_2$ and $x_3\in V_3$. Depending on the
  dimension of $\Comp$, Proposition~\ref{prop:compdim2},
  \ref{prop:compdim4} or \ref{prop:compdim8} shows that the forms
  $\tilde q_1$, $\tilde q_2$ and $\tilde q_3$ are isometric Pfister
  forms. Define a map $\tilde*_3\colon V_1\times V_2\to V_3$ by
  \[
    x_1\,\tilde*_3\, x_2=\lambda_3\,x_1*_3x_2 \qquad\text{for $x_1\in
      V_1$ and $x_2\in V_2$.}
  \]
  Straightforward computations show that $\tilde\Comp=\bigl((\tilde
  V_1,\tilde q_1),\,(\tilde V_2,\tilde q_2),\,(\tilde V_3,\tilde
  q_3),\,\tilde *_3\bigr)$ is a composition, and that
  $(\Id_{V_1},\,\Id_{V_2},\,\Id_{V_3})\colon \Comp\to\tilde\Comp$ is a
  similitude of compositions, with composition multiplier
  $(\lambda_1,\,\lambda_2,\,\lambda_3)$. 
\end{proof}

Auto-similitudes of compositions of quadratic
spaces define algebraic groups which we discuss next.

For every composition $
\Comp=\bigl((V_1,q_1),\,(V_2,q_2),\,(V_3,q_3),\,*_3\bigr)
$, we associate to each similitude
\[
  (g_1,g_2,g_3)\colon\Comp\to\Comp
\]
with multiplier $(\lambda_1,\lambda_2,\lambda_3)$ the $4$-tuple
$(g_1,g_2,g_3,\lambda_3)$, from which $\lambda_1$ and $\lambda_2$ can
be determined by the relations~\eqref{eq:lambdamu}. We may thus
consider the group of similitudes of $\Comp$ as the subgroup of
$\GO(q_1)\times\GO(q_2)\times\GO(q_3)\times F^\times$ defined by the equations
\[
    \lambda_3\,g_3(x_1*_3x_2)=g_1(x_1)*_3g_2(x_2) \qquad\text{for all
      $x_1\in V_1$, $x_2\in V_2$.}
\]
These equations define a closed subgroup of $\BGO(q_1)\times
\BGO(q_2)\times\BGO(q_3)\times\BGm$, hence an algebraic group scheme, 
for which we use the notation $\BGO(\Comp)$. From
Proposition~\ref{prop:simdef2} it follows that $\partial$ and
$\partial^2$ yield isomorphisms
\[
  \partial\colon \BGO(\Comp)\to\BGO(\partial\Comp)
  \qquad\text{and}\qquad
  \partial^2\colon\BGO(\Comp)\to \BGO(\partial^2\Comp)
\]
defined as follows: for every commutative $F$-algebra $R$ and
$(g_1,g_2,g_3,\lambda_3)\in\BGO(\Comp)(R)$,
\[
  \partial(g_1,g_2,g_3,\lambda_3) = (g_2,g_3,g_1,\lambda_1)
  \qquad\text{and}\qquad
  \partial(g_1,g_2,g_3,\lambda_3) = (g_3,g_1,g_2,\lambda_2),
\]
with
\[
  \lambda_1=\mu(g_2)\lambda_3^{-1} \qquad\text{and}\qquad
  \lambda_2=\mu(g_1)\lambda_3^{-1}.
\]

The Lie algebra $\go(\Comp)$ of $\BGO(\Comp)$ consists of $4$-tuples
$(g_1,g_2,g_3,\lambda_3)\in\go(q_1)\times\go(q_2)\times \go(q_3)\times
F$ satisfying the following condition:
\begin{equation}
    \label{eq:defnlocsim}
    g_3(x_1*_3x_2) = g_1(x_1)*_3x_2+x_1*_3g_2(x_2) - \lambda_3\,
    x_1*_3x_2 \qquad\text{for all $x_1\in V_1$, $x_2\in V_2$.}
\end{equation}

The following is the Lie algebra version of
Proposition~\ref{prop:simdef2}: 

\begin{prop}
  \label{prop:Liesimdef2}
  For $g=(g_1,g_2,g_3,\lambda_3)\in\go(\Comp)$, there are scalars
  $\lambda_1$, $\lambda_2\in F$ such that
  \[
    \dot\mu(g_1) = \lambda_2+\lambda_3,\qquad
    \dot\mu(g_2) = \lambda_3+\lambda_1,\qquad
    \dot\mu(g_3) = \lambda_1+\lambda_2
  \]
  and for all
  $x_1\in V_1$, $x_2\in V_2$, $x_3\in V_3$
  \begin{align*}
    g_1(x_2*_1x_3) & = g_2(x_2)*_1x_3+x_2*_1g_3(x_3) - \lambda_1\,
                     x_2*_1x_3,\\
    g_2(x_3*_2x_1) & = g_3(x_3)*_2x_1 + x_3*_2g_1(x_1) - \lambda_2\,
                     x_3*_2x_1,\\
    g_3(x_1*_3x_2) & = g_1(x_1)*_3x_2 + x_1*_3g_2(x_2) - \lambda_3\,
                     x_1*_3x_2.
  \end{align*}
  Thus, $\partial g :=(g_2,g_3,g_1,\lambda_1)$ lies in $\go(\partial\Comp)$
  and $\partial^2g:=(g_3,g_1,g_2,\lambda_2)$ in $\go(\partial^2\Comp)$. 
\end{prop}

The composition multiplier map $\lambda_\Comp$ yields a morphism of
algebraic group schemes
\[
    \lambda_\Comp\colon\BGO(\Comp)\to \BGm^3
\]
defined as follows: for every commutative $F$-algebra $R$ and
$(g_1,g_2,g_3,\lambda_3)\in\BGO(\Comp)(R)$,
\begin{equation}
  \label{eq:lambdaCompdef}
    \lambda_\Comp(g_1,g_2,g_3,\lambda_3) = (\mu(g_2)\lambda_3^{-1},
    \mu(g_1)\lambda_3^{-1}, \lambda_3) \in R^\times\times
    R^\times\times R^\times.
\end{equation}
Its differential $\dot\lambda_\Comp\colon\go(\Comp)\to F\times F\times F$ is
given by
\[
  \dot\lambda_\Comp(g_1,g_2,g_3,\lambda_3) = (\dot\mu(g_2)-\lambda_3,\,
  \dot\mu(g_1)-\lambda_3,\, \lambda_3).
\]
We let $\BOrth(\Comp)=\ker\lambda_\Comp$ and
$\orth(\Comp)=\ker\dot\lambda_\Comp$, 
so $\BOrth(\Comp)$ is the algebraic group scheme of automorphisms of
$\Comp$ and $\orth(\Comp)$ is its Lie algebra.

\begin{remark}
  \label{rem:RT}
  For every commutative $F$-algebra $R$ and
  $(g_1,g_2,g_3,\lambda_3)\in\BOrth(\Comp)(R)$ we have $\lambda_3=1$
  and $\mu(g_1)=\mu(g_2)=1$, hence also $\mu(g_1)=1$
  by~\eqref{eq:lambdamu}. Thus, $(g_1,g_2,g_3)$ is a \emph{related
    triple} of isometries according to the definition given by
  Springer--Veldkamp \cite[\S3.6]{SpV}, Elduque \cite[\S1]{Eld} or
  Alsaody--Gille 
  \cite[\S3.1]{AlsGille} for some specific compositions of quadratic
  spaces arising from composition algebras. In~\S\ref{subsec:triso}
  below we establish isomorphisms $\BOrth(\Comp)\simeq \BSpin(q_1)
  \simeq \BSpin(q_2)\simeq \BSpin(q_3)$, which are the analogues of
  the isomorphisms given in~\cite[Prop.~3.6.3]{SpV},
  \cite[Th.~1.1]{Eld} and 
  \cite[Th.~3.12]{AlsGille} in terms of related triples.
\end{remark}
\begin{prop}
  \label{prop:exseqOGOogo}
  The algebraic group schemes $\BOrth(\Comp)$ and $\BGO(\Comp)$ are
  smooth, and the following sequences are exact:
  \begin{equation}
    \label{eq:exseqOGO}
    1\to\BOrth(\Comp)\to\BGO(\Comp)\xrightarrow{\lambda_\Comp}\BGm^3\to 1
  \end{equation}
  and
  \begin{equation}
    \label{eq:exseqogo}
    0\to\orth(\Comp)\to\go(\Comp)\xrightarrow{\dot\lambda_\Comp} F^3\to 0.
  \end{equation}
\end{prop}

\begin{proof}
  \emph{Step~1:} We show that $\lambda_\Comp$ is surjective. Since $\BGm^3$
  is smooth, it suffices by~\cite[(22.3)]{BoI} to show that $\lambda_\Comp$
  is surjective on points over an algebraic closure $\Falg$ of
  $F$. For this, we consider the homotheties: if $\nu_1$, $\nu_2$,
  $\nu_3\in\Falg^\times$, then $\nu_i\Id_{(V_i)_{\Falg}}\colon
  (V_i,q_i)_{{\Falg}} \to (V_i,q_i)_{{\Falg}}$ is a similitude with
  multiplier $\nu_i^2$, and
  \[
    \bigl(\nu_1\Id_{(V_1)_{\Falg}},\,\nu_2\Id_{(V_2)_{\Falg}},\,
    \nu_3\Id_{(V_3)_{\Falg}}\bigr)\colon \Comp_{\Falg} \to \Comp_{\Falg}
  \]
  is a similitude with multiplier $(\nu_2\nu_3\nu_1^{-1},\,
  \nu_3\nu_1\nu_2^{-1},\, \nu_1\nu_2\nu_3^{-1})$. Therefore, the image
  of the map $\lambda_\Comp$ in $(\Falg^\times)^3$ contains $(\nu_2\nu_3\nu_1^{-1},\,
  \nu_3\nu_1\nu_2^{-1},\, \nu_1\nu_2\nu_3^{-1})$ for all $\nu_1$,
  $\nu_2$, $\nu_3\in \Falg^\times$. Given $\lambda_1$, $\lambda_2$,
  $\lambda_3\in\Falg^\times$, we may find $\nu_1$, $\nu_2$,
  $\nu_3\in\Falg^\times$ such that $\nu_2^2=\lambda_1\lambda_3$,
  $\nu_3^2=\lambda_1\lambda_2$ and
  $\nu_1=\lambda_1^{-1}\nu_2\nu_3$. Then
  \[
    (\nu_2\nu_3\nu_1^{-1},\,
  \nu_3\nu_1\nu_2^{-1},\, \nu_1\nu_2\nu_3^{-1}) =
  (\lambda_1,\,\lambda_2,\,\lambda_3),
  \]
  proving surjectivity of $\lambda_\Comp$.

  \emph{Step~2:} We show that $\dot\lambda_\Comp$ is surjective. For $u_1$,
  $v_1\in V_1$, consider the maps
  \begin{align*}
    g_1\colon V_1\to V_1,&\qquad x_1\mapsto
                            u_1b_1(v_1,x_1)-v_1b_1(u_1,x_1),\\
    g_2\colon V_2\to V_2,&\qquad x_2\mapsto(v_1*_3x_2)*_2u_1,\\
    g_3\colon V_3\to V_3,&\qquad x_3\mapsto u_1*_3(x_3*_2v_1).
  \end{align*}
  For $x_1\in V_1$,
  \[
    b_1(g_1(x_1),\,x_1) = b_1(u_1,\,x_1)b_1(v_1,\,x_1) -
    b_1(v_1,\,x_1)b_1(u_1,\,x_1)=0,
  \]
  hence $g_1\in\go(q_1)$ with $\dot\mu(g_1)=0$ by
  Proposition~\ref{prop:Lie1bis}. Moreover, \eqref{eq:comp45},
  \eqref{eq:comp3} and \eqref{eq:compp3} yield for $x_2\in V_2$ and
  $x_3\in V_3$
  \[
    b_2(g_2(x_2),\,x_2)=b_3(v_1*x_2,\,u_1*_2x_2) =
    b_1(v_1,\,u_1)q_2(x_2),
  \]
  \[
    b_3(g_3(x_3),\,x_3)=b_2(x_3*_2v_1,\,x_3*_2u_1) =
    q_3(x_3)b_1(v_1,u_1).
  \]
  Therefore, $g_2\in\go(q_2)$ and $g_3\in\go(q_3)$ with
  $\dot\mu(g_2)=\dot\mu(g_3)=b_1(v_1,\,u_1)$.

  Now, for $x_1\in V_1$ and $x_2\in V_2$ we compute
  $g_3(x_1*_3x_2)=u_1*_3\bigl((x_1*_3x_2)*_2v_1\bigr)$ by
  using~\eqref{eq:comp7linbis} twice in succession to interchange
  first $x_1$ and $v_1$, and then $x_1$ and $u_1$:
  \begin{align*}
    g_3(x_1*_3x_2) &= (u_1*_3x_2)b_1(v_1,\,x_1)-
                     u_1*_3\bigl((v_1*_3x_2)*_2x_1\bigr)\\
    &= (u_1*_3x_2)b_1(v_1,\,x_1)-(v_1*_3x_2)b_1(u_1,\,x_1) +
      x_1*_3\bigl((v_1*_3x_2)*_2u_1\bigr)\\
    &=g_1(x_1)*_3x_2+x_1*_3g_2(x_2).
  \end{align*}
  It follows that $(g_1,g_2,g_3,0)$ lies in $\go(\Comp)$, and the
  computation of $\dot\mu(g_2)$ and $\dot\mu(g_1)$ above yields
  \[
    \dot\lambda_\Comp(g_1,g_2,g_3,0) = (b_1(v_1,\,u_1),\,0,\,0).
  \]
  Thus, taking $u_1$, $v_1$ such that $b_1(v_1,\,u_1)=1$, we see that
  $(1,0,0)$ lies in the image of $\dot\lambda_\Comp$. Similarly, we may find
  $g'\in\go(\partial\Comp)$ and $g''\in\go(\partial^2\Comp)$ such that
  $\dot\lambda_{\partial\Comp}(g')=\dot\lambda_{\partial^2\Comp}(g'')=(1,0,0)$.
  Then
  $\partial^2(g')$, 
  $\partial(g'')\in\go(\Comp)$ satisfy
  $\dot\lambda_{\partial^2\Comp}\bigl(\partial^2(g')\bigr) = (0,0,1)$ and
  $\dot\lambda_{\partial\Comp}\bigl(\partial(g'')\bigr)=(0,1,0)$,
  hence $\dot\lambda_\Comp$ is surjective.
  \medbreak

  Steps~1 and 2 establish the exactness of the
  sequences~\eqref{eq:exseqOGO} and \eqref{eq:exseqogo}. Step~2 shows
  that the surjective map $\lambda_\Comp$ is separable, hence
  $\BOrth(\Comp)$ is smooth by~\cite[(22.13)]{BoI}. Since $\BGm^3$ is
  also smooth, it follows that $\BGO(\Comp)$ is smooth by~\cite[(22.12)]{BoI}.
\end{proof}

Step~1 of the proof of Proposition~\ref{prop:exseqOGOogo} introduces
the subgroup of homotheties of $\BGO(\Comp)$: this subgroup
$\BHomot(\Comp)$ is the image of the closed embedding
$\BGm^3\to\BGO(\Comp)$ given by
\[
  (\nu_1,\nu_2,\nu_3) \mapsto (\nu_1\Id_{V_1}, \nu_2\Id_{V_2},
  \nu_3\Id_{V_3}, \nu_1\nu_2\nu_3^{-1}).
\]
The algebraic group $\BHomot(\Comp)$ lies in the center of
$\BGO(\Comp)$, hence we may consider the quotient algebraic group
\[
  \BPGO(\Comp)=\BGO(\Comp)/\BHomot(\Comp).
\]
This is a smooth algebraic group since $\BGO(\Comp)$ is smooth. Let
also
\[
  \BZO(\Comp)=\BHomot(\Comp)\cap \BOrth(\Comp).
\]
This group is the kernel of the canonical map
$\BOrth(\Comp)\to\BPGO(\Comp)$. For every commutative $F$-algebra $R$,
\[
    \BZO(\Comp)(R)=\{(\nu_1,\nu_2,\nu_3,1)\mid
    \nu_1^2=\nu_2^2=\nu_3^2=\nu_1\nu_2\nu_3=1\} 
    \subset R^\times\times R^\times\times R^\times\times R^\times,
\]
hence $\BZO(\Comp)$ is isomorphic to the kernel of the multiplication
map $m\colon \Bmu_2\times\Bmu_2\times\Bmu_2\to\Bmu_2$ carrying
$(\nu_1,\nu_2,\nu_3)$ to $\nu_1\nu_2\nu_3$. It is thus also isomorphic to
$\Bmu_2\times\Bmu_2$, hence it is a smooth algebraic group if and only
if $\charac F\neq2$.

\begin{prop}
  \label{prop:exdiagGOComp}
  The following diagram is commutative with exact rows and columns:
  \[
    \xymatrix{&1\ar[d]&1\ar[d]&&\\
      1\ar[r]&\BZO(\Comp)\ar[r]\ar[d]&\BOrth(\Comp)\ar[r]\ar[d]&
      \BPGO(\Comp)\ar[r]\ar@{=}[d]&1\\
      1\ar[r]&\BHomot(\Comp)\ar[r]\ar[d]_{\lambda_\Comp}&
      \BGO(\Comp)\ar[r]\ar[d]^{\lambda_\Comp}&\BPGO(\Comp)\ar[r]&1\\
      &\BGm^3\ar@{=}[r]\ar[d]&\BGm^3\ar[d]&&\\
      &1&1&&}
  \]
\end{prop}

\begin{proof}
  Commutativity of the diagram is clear, and the lower row is exact by
  definition of $\BPGO(\Comp)$. Step~1 of the proof of
  Proposition~\ref{prop:exseqOGOogo} shows that $\lambda_\Comp\colon
  \BHomot(\Comp)\to\BGm^3$ is surjective, hence the left column is
  exact. Moreover, the right column is exact by
  Proposition~\ref{prop:exseqOGOogo}; therefore it only remains to
  prove that the canonical map $\BOrth(\Comp)\to\BPGO(\Comp)$ is
  surjective. Since $\BPGO(\Comp)$ is smooth, it suffices to consider
  the group of rational points over an algebraic closure $\Falg$ of
  $F$. We know $\lambda_\Comp\colon \BHomot(\Comp)\to\BGm^3$ is surjective,
  hence for every $g\in \BGO(\Comp)(\Falg)$ there exists
  $h\in\BHomot(\Comp)(\Falg)$ such that
  $\lambda_\Comp(g)=\lambda_\Comp(h)$. Then 
  $gh^{-1}$ lies in $\BOrth(\Comp)(\Falg)$ and has the same image in
  $\BPGO(\Comp)(\Falg)$ as $g$, hence the canonical map
  $\BOrth(\Comp)(\Falg)\to \BPGO(\Comp)(\Falg)$ is surjective.
\end{proof}

Let $\homot(\Comp)$ and $\pgo(\Comp)$ be the Lie algebras of
$\BHomot(\Comp)$ and $\BPGO(\Comp)$ respectively. By definition,
\[
\homot(\Comp)=\{(\nu_1\Id_{V_1},\nu_2\Id_{V_2},
\nu_3\Id_{V_3}, \nu_1+\nu_2-\nu_3)\mid \nu_1, \nu_2, \nu_3\in F\}\simeq
F\times F\times F.
\]
On the other hand, since $\BHomot(\Comp)$ is smooth, the canonical map
$\BGO(\Comp)\to\BPGO(\Comp)$ is separable by~\cite[(22.13)]{BoI},
hence its differential is surjective. Therefore,
\[
  \pgo(\Comp)=\go(\Comp)/\homot(\Comp).
\]

The following result yields an explicit description of $\pgo(\Comp)$
for use in \S\ref{subsec:dertriple}:

\begin{prop}
  \label{prop:pgoComp}
  Mapping $(g_1,g_2,g_3,\lambda_3)+\homot(\Comp)\in\pgo(\Comp)$ to
  $(g_1+F,\,g_2+F,\,g_3+F)\in \pgo(q_1)\times\pgo(q_2)\times
  \pgo(q_3)$ identifies $\pgo(\Comp)$ with the subgroup of
  $\pgo(q_1)\times\pgo(q_2)\times\pgo(q_3)$ consisting of triples
  $(g_1+F,\,g_2+F,\,g_3+F)$ where $g_1\in\go(q_1)$, $g_2\in\go(q_2)$ and
  $g_3\in\go(q_3)$ satisfy~\eqref{eq:defnlocsim} for some
  $\lambda_3\in F$.
\end{prop}

\begin{proof}
  It suffices to show that $\homot(\Comp)$ is the kernel of the map
  $\go(\Comp)\to \pgo(q_1)\times\pgo(q_2)\times\pgo(q_3)$ carrying
  $(g_1,g_2,g_3,\lambda_3)$ to $(g_1+F,\,g_2+F,\,g_3+F)$. Clearly,
  $\homot(\Comp)$ lies in the kernel of this map. Conversely, if
  $(g_1,g_2,g_3,\lambda_3)$ lies in the kernel, then there are scalars
  $\nu_1$, $\nu_2$, $\nu_3\in F$ such that $g_i=\nu_i\Id_{V_i}$ for
  $i=1$, $2$, $3$. Then~\eqref{eq:defnlocsim} yields
  $\lambda_3=\nu_1+\nu_2-\nu_3$, hence $(g_1,g_2,g_3,\lambda_3)$ lies
  in $\homot(\Comp)$.
\end{proof}

\begin{remark}
If $\charac F\neq2$, the upper row of the diagram in
Proposition~\ref{prop:exdiagGOComp} shows that the canonical map
$\orth(\Comp)\to\pgo(\Comp)$ is an isomorphism, for then $\BZO(\Comp)$
is smooth and its Lie algebra is~$0$. This canonical map is \emph{not}
bijective if $\charac F=2$, even though $\orth(\Comp)$ and
$\pgo(\Comp)$ have the same dimension.
\end{remark}

\subsection{Compositions of pointed quadratic spaces}
\label{subsec:pointedcomp}

Fixing a representation of $1$ in a quadratic space yields a new
structure: 

\begin{defs}
  \label{defs:pcomp}
  A \emph{pointed quadratic space} over an arbitrary field $F$ is a
  triple $(V,q,e)$ where $(V,q)$ is a quadratic space with nonsingular
  polar form over $F$ and $e\in V$ is a vector such that
  $q(e)=1$. Each pointed quadratic space is endowed with a canonical
  isometry $\invo$ of order~$2$, defined by
  \[
    \overline x = e\,b(e,x)-x \qquad\text{for $x\in V$,}
  \]
  where $b$ is the polar form of $q$. Isometries of
  pointed quadratic 
  spaces are required to preserve the distinguished vector
  representing~$1$. 

  A \emph{composition of pointed quadratic spaces} over $F$ is a
  $4$-tuple 
  \begin{equation}
    \label{eq:pointedcomp}
    \Comp^\bullet = \bigl((V_1,q_1,e_1),\,(V_2,q_2,e_2),\, (V_3,q_3,e_3),\,
    *_3\bigr)
  \end{equation}
  where $(V_1,q_1,e_1)$, $(V_2,q_2,e_2)$, $(V_3,q_3,e_3)$ are pointed
  quadratic spaces over $F$ and the $4$-tuple
  $\Comp:=\bigl((V_1,q_1)\,(V_2,q_2),\,(V_3,q_3),\,*_3\bigr)$ obtained
  by forgetting the distinguished vectors is a
  composition of quadratic spaces such that
  \[
    e_1*_3e_2=e_3.
  \]
\end{defs}

It follows from the definition that $q_1$, $q_2$ and $q_3$
represent~$1$, hence these forms are isometric Pfister forms, by
Proposition~\ref{prop:compdim2}, \ref{prop:compdim4} or
\ref{prop:compdim8} (depending on the dimension). Note that
\eqref{eq:compp5} readily yields
\[
  e_2*_1e_3=e_1 \qquad\text{and}\qquad e_3*_2e_1=e_2.
\]
Therefore, the following are compositions
of pointed quadratic spaces:
\begin{align*}
    \partial\Comp^\bullet & = \bigl((V_2,q_2,e_2), (V_3,q_3,e_3),
    (V_1,q_1,e_1), *_1),\\
    \partial^2\Comp^\bullet & = \bigl((V_3,q_3,e_3), (V_1,q_1,e_1),
    (V_2,q_2,e_2), *_2).
\end{align*}

Let $\Comp^\bullet$ and $\tilde\Comp^\bullet$ be compositions of
pointed quadratic spaces, and let $\Comp$ and $\tilde\Comp$ be the
compositions of quadratic spaces obtained from $\Comp^\bullet$ and
$\tilde\Comp^\bullet$ by forgetting the distinguished vectors. Every
similitude $f\colon\Comp\to\tilde\Comp$ that preserves the
distinguished vectors must be an isometry, because the equations
\begin{align*}
  \lambda_1\,f_1(x_2*_1x_3) & = f_2(x_2)\,\tilde*_1\,f_3(x_3),\\
  \lambda_2\,f_2(x_3*_2x_1) & = f_3(x_3)\,\tilde*_2\,f_1(x_1),\\
  \lambda_3\,f_3(x_1*_3x_2) & =f_1(x_1)\,\tilde*_3\,f_2(x_2)
\end{align*}
for $x_1\in V_1$, $x_2\in V_2$, $x_3\in V_3$  imply
$\lambda_1=\lambda_2=\lambda_3=1$ if $f(e_i)=\tilde e_i$ for $i=1$,
$2$, $3$. Therefore, between compositions of pointed quadratic spaces
the only type of maps we consider are isomorphisms.

\begin{definition}
  \label{defn:isocompointed}
  An \emph{isomorphism} $f\colon\Comp^\bullet\to \tilde\Comp^\bullet$ of
  compositions of pointed quadratic spaces is an isomorphism
  $f\colon\Comp\to\tilde\Comp$ of compositions of quadratic spaces
  that maps the distinguished vectors of $\Comp^\bullet$ to the
  distinguished vectors of $\tilde\Comp^\bullet$. The automorphisms of
  $\Comp^\bullet$ define an algebraic group scheme
  $\BOrth(\Comp^\bullet)$, which is a closed subgroup of
  $\BOrth(\Comp)$. 
\end{definition}

Our goal in this subsection is to show that every composition of pointed
quadratic spaces $\Comp^\bullet$ carries a canonical isomorphism
$\Delta\colon\Comp^\bullet\to\partial\Comp^\bullet$ and is isomorphic
to a composition $S(\Comp^\bullet)$ such that $\partial
S(\Comp^\bullet)=S(\Comp^\bullet)$. For this, we will use the following
identities relating the canonical isometry $\invo$ and multiplication
by the distinguished vectors:

\begin{lemma}
  \label{lem:pcomp}
  Let $\Comp^\bullet$ be a composition of pointed quadratic spaces as
  in~\eqref{eq:pointedcomp}.
  \begin{enumerate}
  \item[(a)]
    For every $x_1\in V_1$, $x_2\in V_2$, $x_3\in V_3$,
    \begin{align*}
      \overline{e_1*_3x_2}
      & = e_1*_3\overline{x_2},
      & \overline{x_1*_3e_2}
      & = \overline{x_1}*_3e_2,\\
      \overline{e_2*_1x_3}
      & = e_2*_1\overline{x_3},
      & \overline{x_2*_1e_3}
      & = \overline{x_2}*_1e_3,\\
      \overline{e_3*_2x_1}
      & = e_3*_2\overline{x_1},
      & \overline{x_3*_2e_1}
      & = \overline{x_3}*_2e_1.
    \end{align*}
  \item[(b)]
    For every $x_1\in V_1$, $x_2\in V_2$, $x_3\in V_3$,
    \begin{align*}
      e_1*_3(x_3*_2x_1)
      & = \overline{x_1}*_3(x_3*_2e_1),
      & (x_1*_3x_2)*_2e_1
      & = (e_1*_3x_2)*_2\overline{x_1},\\
      e_2*_1(x_1*_3x_2)
      & = \overline{x_2}*_1(x_1*_3e_2),
      & (x_2*_1x_3)*_3e_2
      & = (e_2*_1x_3)*_3\overline{x_2},\\
      e_3*_2(x_2*_1x_3)
      & = \overline{x_3}*_2(x_2*_1e_3),
      & (x_3*_2x_1)*_1e_3
      & = (e_3*_2x_1)*_1\overline{x_3}.
    \end{align*}
  \item[(c)]
    For every $x_1\in V_1$, $x_2\in V_2$, $x_3\in V_3$,
    \begin{align*}
      e_2*_1\bigl(e_1*_3(e_3*_2x_1)\bigr)
      & = \overline{x_1} = \bigl((x_1*_3e_2)*_2e_1\bigr)*_1e_3,\\
      e_3*_2\bigl(e_2*_1(e_1*_3x_2)\bigr)
      & = \overline{x_2} = \bigl((x_2*_1e_3)*_3e_2)*_2e_1,\\
      e_1*_3\bigl(e_3*_2(e_2*_1x_3)\bigr)
      & = \overline{x_3} = \bigl((x_3*_2e_1)*_1e_3\bigr)*_3e_2.
    \end{align*}
  \item[(d)]
    For every $x_1\in V_1$, $x_2\in V_2$, $x_3\in V_3$,
    \begin{align*}
      \overline{x_1*_3x_2} = (x_2*_1e_3)*_3(e_3*_2x_1)
      & = \bigl((e_3*_2\overline{x_1})*_1(e_1*_3\overline{x_2}\bigr)
        *_3e_2,\\
      \overline{x_2*_1x_3} = (x_3*_2e_1)*_1(e_1*_3x_2)
      & =
        \bigl((e_1*_3\overline{x_2})*_2(e_2*_1\overline{x_3})\bigr)
        *_1e_3,\\ 
      \overline{x_3*_2x_1} = (x_1*_3e_2)*_2(e_2*_1x_3)
      & = \bigl((e_2*_1\overline{x_3})*_3(e_3*_2\overline{x_1})\bigr)
        *_2e_1.
    \end{align*}
  \end{enumerate}
\end{lemma}

\begin{proof}
  To avoid repetitions, we just prove the first formulas in each case.
  
  (a)
  By definition,
  \[
    \overline{e_1*_3x_2} = e_3\,b_3(e_3,e_1*_3x_2) - e_1*_3x_2.
  \]
  Substituting $e_1*_3e_2$ for $e_3$ and using $b_3(e_3,e_1*_3x_2)=
  b_2(e_3*_2e_1,x_2)=b_2(e_2,x_2)$ yields
  \[
    \overline{e_1*_3x_2} = e_1*_3(e_2\,b_2(e_2,x_2) - x_2) =
    e_1*_3\overline{x_2}.
  \]

  (b)
  By \eqref{eq:comp7linbis} and \eqref{eq:compp7},
  \[
    e_1*_3(x_3*_2x_1) = x_3\,b_1(e_1,x_1) - x_1*_3(x_3*_2e_1)
    \quad\text{and}\quad
    x_3=e_1*_3(x_3*_2e_1),
  \]
  hence
  \[
    e_1*_3(x_3*_2x_1)= (e_1\,b_1(e_1,x_1)-x_1)*_3(x_3*_2e_1) =
    \overline{x_1}*_3(x_3*_2e_1).
  \]

  (c)
  Using (b) and \eqref{eq:compp5}, we have
  \[
    e_2*_1\bigl(e_1*_3(e_3*_2x_1)\bigr) =
    e_2*_1\bigl(\overline{x_1}*_3(e_3*_2e_1)\bigr) =
    e_2*_1(\overline{x_1}*_3e_2)= \overline{x_1}.
  \]

  (d)
  We compute $(x_2*_1e_3)*_3(e_3*_2x_1)$ by using
  \eqref{eq:comp7linbis} to exchange the factors $x_2*_1e_3$ and
  $x_1$:
  \[
    (x_2*_1e_3)*_3(e_3*_2x_1) = e_3\,b_1(x_2*_1e_3,x_1) -
    x_1*_3\bigl(e_3*_2(x_2*_1e_3)\bigr).
  \]
  Since $e_3*_2(x_2*_1e_3)=x_2$ by~\eqref{eq:compp6} and
  $b_1(x_2*_1e_3,x_1)=b_3(e_3,x_1*_3x_2)$ by~\eqref{eq:comp45}, it
  follows that
  \[
    (x_2*_1e_3)*_3(e_3*_2x_1)= e_3\,b_3(e_3,x_1*_3x_2) - x_1*_3x_2 =
    \overline{x_1*_3x_2}.
  \]
  On the other hand, (b) yields
  \[
    x_2*_1e_3=x_2*_1(e_1*_3e_2)=e_2*_1(e_1*_3\overline{x_2}),
  \]
  hence, using~(b) again together with~(a),
  \begin{multline*}
    (x_2*_1e_3)*_3(e_3*_2x_1) =
    \bigl(e_2*_1(e_1*_3\overline{x_2})\bigr)*_3(e_3*_2x_1) \\
    = \bigl(\overline{e_3*_2x_1}*_1(e_1*_3\overline{x_2})\bigr)*_3e_2
    = \bigl((e_3*_2\overline{x_1})*_1(e_1*_3\overline{x_2})\bigr)
    *_3e_2.
    \qedhere
  \end{multline*}
\end{proof}

For a composition of pointed quadratic spaces $\Comp^\bullet$ as
in~\eqref{eq:pointedcomp}, we define a composition of pointed 
quadratic spaces $S(\Comp^\bullet)$ as follows:
\begin{equation}
  \label{eq:functS}
  S(\Comp^\bullet) = \bigl((V_3,q_3,e_3), (V_3,q_3,e_3),
  (V_3,q_3,e_3), \circledast_3\bigr)
\end{equation}
where
\begin{equation}
  \label{eq:circledast}
  x\circledast_3y = (e_2*_1\overline x)*_3(\overline y *_2e_1)
  \qquad\text{for $x$, $y\in V_3$.}
\end{equation}
We also define linear maps $\Delta_1\colon V_1\to V_2$,
$\Delta_2\colon V_2\to V_3$, $\Delta_3\colon V_3\to V_1$ as follows:
for $x_1\in V_1$, $x_2\in V_2$ and $x_3\in V_3$,
\[
  \Delta_1(x_1)= e_3*_2\overline{x_1},\qquad
  \Delta_2(x_2) = e_1*_3\overline{x_2},\qquad
  \Delta_3(x_3) = e_2*_1\overline{x_3}.
\]

\begin{thm}
  \label{thm:pcomp}
  With the notation above,
  \begin{enumerate}
  \item[(a)]
    the triple $\Delta=(\Delta_1,\Delta_2,\Delta_3)$ is an isomorphism
    $\Delta\colon \Comp^\bullet\to\partial\Comp^\bullet$;
  \item[(b)]
    the triple $(\Delta_3,\Delta_2^{-1},\Id_{V_3})$ is an isomorphism
    $S(\Comp^\bullet)\to\Comp^\bullet$;
  \item[(c)]
    $\partial S(\Comp^\bullet) = S(\Comp^\bullet)$.
  \end{enumerate}
\end{thm}

\begin{proof}
  (a)
  It is clear that each $\Delta_i$ is an isometry of pointed quadratic
  spaces, so it suffices to prove $\Delta_3(x_1*_3x_2) =
  \Delta_1(x_1)*_1\Delta_2(x_2)$ for $x_1\in V_1$ and $x_2\in V_2$,
  which amounts to
  \[
    e_2*_1\overline{x_1*_3x_2} = (e_3*_2\overline{x_1})*_1
    (e_1*_3\overline{x_2}).
  \]
  This readily follows from~(d) of Lemma~\ref{lem:pcomp}.

  (b)
  For $y\in V_3$ we have $\Delta_2(\overline y*_2e_1) = e_1*_3
  \overline{\overline y*_2e_1} = y$, hence by definition
  $x\circledast_3y=\Delta_3(x)*_3\Delta_2^{-1}(y)$, which proves~(b).

  (c)
  It suffices to prove
  \[
    b_3(x\circledast_3y,z) = b_3(x, y\circledast_3z) \qquad\text{for
      $x$, $y$, $z\in V_3$.}
  \]
  For this, we first compute using Lemma~\ref{lem:pcomp}
  \begin{align*}
    \overline{y\circledast_3z}
    & = \overline{(e_2*_1\overline y)*_3(\overline z *_2e_1)}\\
    & = \bigl((e_3*_2\overline{e_2*_1\overline y})*_1
      (e_1*_3\overline{\overline z*_2e_1})\bigr)*_3e_2\\
      & = \bigl(\bigl(e_3*_2(e_2*_1y)\bigr) *_1
        \bigl(e_1*_3(z*_2e_1)\bigr)\bigr)*_3e_2\\
    & = \bigl((\overline y*_2e_1)*_1z\bigr)*_3e_2.
  \end{align*}
  Since $\invo$ and multiplication on the left by $e_2$ are
  isometries, it follows that
  \[
    b_3(x,y\circledast_3z) = b_1(e_2*_1\overline x, (\overline
    y*_2e_1)*_1z).
  \]
  By definition of $*_1$, the right side is
  \[
    b_3\bigl((e_2*_1\overline x)*_3(\overline y*_2e_1), z\bigr) =
    b_3(x\circledast_3y,z).
    \qedhere
  \]
\end{proof}

For compositions of (unpointed) isometric spaces, a result similar to
Theorem~\ref{thm:pcomp} easily follows:

\begin{corol}
  \label{corol:pcomp}
   For every composition of isometric quadratic spaces $\Comp$, there
   is an isomorphism $\Comp\simeq\partial\Comp$. Moreover, $\Comp$ is
   isomorphic to a composition $S(\Comp)$ such that
   $\partial S(\Comp)=S(\Comp)$. 
\end{corol}

\begin{proof}
  Let $\Comp=\bigl((V_1,q_1),(V_2,q_2),(V_3,q_3),*_3\bigr)$ be a
  composition of isometric quadratic spaces. By
  Corollary~\ref{corol:comp1ter}, we know that $q_1$, $q_2$ and $q_3$
  represent~$1$. We may therefore use the same constructions as in
  Theorem~\ref{thm:pcomp}, after choosing $e_1\in V_1$ and $e_2\in
  V_2$ such that 
  $q_1(e_1)=q_2(e_2)=1$ and letting $e_3=e_1*_3e_2$, for then
  $\bigl((V_1,q_1,e_1), (V_2,q_2,e_2), (V_3,q_3,e_3),*_3\bigr)$ is a
  composition of pointed quadratic spaces. Define the maps
  $\Delta_1\colon V_1\to V_2$, $\Delta_2\colon V_2\to V_3$,
  $\Delta_3\colon V_3\to V_1$ as in Theorem~\ref{thm:pcomp}. The proof
  of that theorem shows that $(\Delta_1,\Delta_2,\Delta_3)$ is an
  isomorphism $\Comp\to\partial\Comp$. Moreover, letting
  \[
    S(\Comp)=\bigl((V_3,q_3), (V_3,q_3), (V_3,q_3),
    \circledast_3\bigr)
  \]
  with $\circledast_3$ as in~\eqref{eq:circledast}, we see from the
  proof of Theorem~\ref{thm:pcomp} that $\partial S(\Comp) = S(\Comp)$
  and that $(\Delta_3,\Delta_2^{-1},\Id_{V_3})$ is an isomorphism
  $S(\Comp)\xrightarrow{\sim} \Comp$.
\end{proof}

Note that, in contrast with Theorem~\ref{thm:pcomp}, the constructions
in Corollary~\ref{corol:pcomp} are not canonical, since they depend on 
the choice of distinguished vectors.

\subsection{Composition algebras}
\label{subsec:compalg}

The purpose of this subsection is to briefly review the classical notion
of composition algebra, in order to underline its connections with compositions
of quadratic spaces.

\begin{defs}
  A \emph{composition algebra} over $F$ is a triple
  $\Calg=(A,q,\diamond)$ 
  where 
  $(A,q)$ is a (finite-dimensional) quadratic space over $F$ with
  nonsingular polar bilinear form and $\diamond\colon A\times A\to A$
  is a bilinear map such that
  \[
    q(x\diamond y)=q(x)q(y)\qquad\text{for all $x$, $y\in A$.}
  \]
  The definition can be rephrased as follows: the $4$-tuple
  \begin{equation}
    \label{eq:CompCalg}
    \Comp(\Calg) = \bigl((A,q),\,(A,q),\,(A,q),\,\diamond\bigr)
  \end{equation}
  is a composition of quadratic spaces.
  Theorem~\ref{thm:dimcomp} shows that the dimension of a composition
  algebra is $1$, $2$, $4$ or $8$, with dimension~$1$ occurring only
  when $\charac F\neq2$, and Corollary~\ref{corol:comp1ter} shows that
  $q$ is a Pfister form.

  A \emph{unital composition algebra}\footnote{Unital composition
    algebras are called 
    \emph{Hurwitz algebras} in \cite{BoI}, see \cite[(33.17)]{BoI}.}
  is a $4$-tuple  
  $\Calg^\bullet=(A,q,e,\diamond)$ where 
  $(A,q,e)$ is a pointed quadratic space and $\diamond\colon A\times
  A\to A$ is a bilinear map such that
  \[
    q(x\diamond y)=q(x)q(y) \qquad\text{and}\qquad
    e\diamond x=x\diamond e=x \qquad\text{for all $x$, $y\in A$.}
  \]
  In any unital composition algebra we have $e\diamond e=e$, hence
  \begin{equation}
    \label{eq:pCompCalg}
    \Comp^\bullet(\Calg^\bullet) =
    \bigl((A,q,e),\,(A,q,e),\,(A,q,e),\,\diamond\bigr)
  \end{equation}
  is a composition of pointed quadratic spaces. 
\end{defs}

As with more general compositions of quadratic spaces, the
multiplication law 
$\diamond$ of a composition algebra $\Calg$ induces derived
composition maps~$\diamond_1$ and $\diamond_2$ of $(A,q)$, $(A,q)$,
$(A,q)$, defined by the conditions
\begin{equation}
  \label{eq:compalg1}
  b(x,y\diamond_1z)=b(x\diamond y,z) \qquad\text{and}\qquad
  b(x\diamond_2y,z) = b(x,y\diamond z) \qquad\text{for $x$, $y$,
    $z\in A$.}
\end{equation}
We may therefore define derived composition algebras $\partial\Calg$
and $\partial^2\Calg$ by
\[
  \partial\Calg=(A,q,\diamond_1) \qquad\text{and}\qquad
  \partial^2\Calg=(A,q,\diamond_2).
\]
Composition algebras $\Calg$ such that
$\partial\Calg=\Calg$ are called \emph{symmetric
  composition algebras}. They are characterized by the condition that
\[
  b(x\diamond y,z) = b(x,y\diamond z) \qquad\text{for all $x$, $y$,
    $z\in A$.}
\]
\medbreak\par
By contrast with compositions of pointed quadratic spaces, the
derivation procedure does \emph{not} preserve unitality of composition
algebras. To make this point clear, we determine below the derived
composition maps of a unital composition algebra, using results
from~\cite[Ch.~1]{SpV}. Note that unital composition algebras carry a
canonical involutory isometry $\invo$ 
derived from their pointed quadratic space structure as
in~Definitions~\ref{defs:pcomp}. 

\begin{prop}
  \label{prop:dercompunit}
  Let $\Calg^\bullet=(A,q,e,\diamond)$ be a unital composition algebra.
  \begin{enumerate}
  \item[(a)]
    The derived composition maps $\diamond_1$ and $\diamond_2$ defined
    in~\eqref{eq:compalg1} are given by
    \[
      x\diamond_1y=y\diamond\overline x \qquad\text{and}\qquad
      x\diamond_2y=\overline y \diamond x \qquad\text{for $x$, $y\in
        A$.}
    \]
  \item[(b)]
    For the bilinear map $*\colon A\times A\to A$ defined by
  \[
    x*y=\overline x \diamond \overline y \qquad\text{for $x$, $y\in
      A$,}
  \]
  the triple $S(\Calg^\bullet)=(A,q,*)$ is a symmetric composition
  algebra. 
  \end{enumerate}
\end{prop}

\begin{proof}
  (a)
  Lemma~1.3.2 in \cite{SpV} yields $b(x,y\diamond z) = b(z, \overline
  y\diamond x)$, hence $b(x\diamond_2y,z)= b(\overline y\diamond x,z)$
  for all $x$, $y$, $z\in A$. Since $b$ is nonsingular, it follows
  that $x\diamond_2 y=\overline y\diamond x$.

  Since $\invo$ is an isometry, \cite[Lemma~1.3.2]{SpV} also yields
  $b(\overline{x\diamond y},\overline z) = b(\overline y,
  \overline{\overline x \diamond z})$. Now, \cite[Lemma1.3.1]{SpV}
  shows that $\overline{x\diamond y} = \overline y\diamond \overline x$, 
  hence the definition of $\diamond_1$ yields
  \[
    b(\overline y\diamond\overline x,\overline z) = b(\overline y,
    \overline z\diamond x) = b(\overline z,x\diamond_1\overline y)
    \qquad\text{for all $x$, $y$, $z\in A$.}
  \]
  Since $b$ is nonsingular, it follows that $x\diamond_1\overline y =
  \overline y\diamond\overline x$.

  (b)
  Since $q(\overline x)=q(x)$ for all $x\in A$, it is clear that
  $S(\Calg^\bullet)$ is a composition algebra. To prove
  that the derived maps 
  $*_1$, $*_2$ associated to $*$ are identical to $*$, it suffices to
  prove that $b(x*y,z)=b(x,y*z)$ for all $x$, $y$, $z\in A$, which
  amounts to
  \begin{equation}
    \label{eq:comp13}
    b(\overline x\diamond \overline y,z) = b(x, \overline y\diamond
    \overline z) \qquad\text{for all $x$, $y$, $z\in A$.}
  \end{equation}
  From the definition of $\diamond_1$ in~\eqref{eq:compalg1} and its
  determination in~(a), it follows that
  \[
    b(\overline x\diamond\overline y,z) = b(\overline x, \overline
    y\diamond_1z) = b(\overline x,z\diamond y)
    \qquad\text{for all $x$, $y$, $z\in A$.}
  \]
  Now, $z\diamond y=\overline{\overline y\diamond\overline z}$
  by~\cite[Lemma~1.3.1]{SpV}, and $\invo$ is an isometry, hence the
  rightmost side in the last displayed equation is equal to
  $b(x,\overline y\diamond\overline z)$, which
  proves~\eqref{eq:comp13}. 
\end{proof}

Note that in the context of Proposition~\ref{prop:dercompunit},
$x\diamond_1e=\overline x= e\diamond_2x$ for all $x\in A$. 
Hence $(A,q,e,\diamond_1)$ and $(A,q,e,\diamond_2)$ are \emph{not}
unital composition algebras, unless $\invo=\Id_V$, which occurs only
if $\dim A=1$. 

Symmetric composition algebras $S(\Calg^\bullet)$ derived from unital
composition algebras $\Calg^\bullet$ as in
Proposition~\ref{prop:dercompunit} are called \emph{para-unital}
composition algebras (\emph{para-Hurwitz algebras} in the terminology of
\cite{BoI}). They are characterized by the property that they contain
a \emph{para-unit}, see \cite[(34.8)]{BoI}.
\medbreak
\par

Between algebras, maps that are more general than homomorphisms are
considered, following Albert~\cite{Alb}.

\begin{definition}
  \label{defn:isotopy}
  Let $(A,\diamond)$ and $(\tilde A,\tilde\diamond)$ be $F$-algebras
  (i.e., $F$-vector spaces with a bilinear multiplication). An
  \emph{isotopy} $f\colon(A,\diamond)\to (\tilde A,\tilde\diamond)$ is
  a triple $f=(f_1,f_2,f_3)$ of linear bijections $f_i\colon
  A\to\tilde A$ such that
  \[
    f_3(x\diamond y) =f_1(x)\,\tilde\diamond\, f_2(y)
    \qquad\text{for all $x$, $y\in A$.}
  \]
  An \emph{autotopy} is an isotopy of an algebra to itself. Under the
  composition of maps, autotopies of an algebra form a group
  $\Str(A,\diamond)$ known as the \emph{structure group} of
  $(A,\diamond)$. This group is the set of $F$-rational points of an
  algebraic group scheme $\BStr(A,\diamond)$, which is a closed
  subgroup of $\BGL(A)\times\BGL(A)\times\BGL(A)$.
\end{definition}

For example, in the construction $S$ of
Proposition~\ref{prop:dercompunit}, which yields the symmetric
composition algebra $S(\Calg^\bullet)$ from the unital composition
algebra $\Calg^\bullet$, the algebra $(A,*)$ is isotopic to
$(A,\diamond)$. The following construction, due to Kaplansky~\cite{Kap},
shows that the algebra of every composition algebra is isotopic to the
algebra of a unital composition algebra.

\begin{prop}
  \label{prop:Kaplan}
  Let $\Calg=(A,q,\diamond)$ be a composition algebra. There exists a
  bilinear map $*\colon A\times A\to A$ and a vector $e\in A$ for
  which
  \begin{enumerate}
  \item[(a)]
    $(A,q,e,*)$ is a unital composition algebra, and
  \item[(b)]
    there exists
    an isotopy $f=(f_1,f_2,f_3)\colon (A,\diamond)\to(A,*)$ which is
    also an isomorphism $f\colon\Comp(\Calg)\to\Comp(A,q,*)$ of the
    associated 
    compositions of quadratic spaces as in~\eqref{eq:CompCalg}.
  \end{enumerate}
\end{prop}

\begin{proof}
  Since $\Calg$ is a composition algebra,
  Corollary~\ref{corol:comp1ter} shows
  that there exists $u\in A$ such that $q(u)=1$. The maps
  $\ell_u\colon A\to A$ and $r_u\colon A\to A$ defined by
  $\ell_u(x)=u\diamond x$ and $r_u(x)=x\diamond u$ are isometries of
  $(A,q)$, hence they are invertible. Define
  $*\colon A\times A\to A$ by
  \[
    x*y=r_u^{-1}(x)\diamond \ell_u^{-1}(y) \qquad\text{for $x$, $y\in
      A$},
  \]
  hence $x\diamond y= r_u(x)*\ell_u(y)$ for $x$, $y\in A$.
  It is clear from the definitions that $(\Id_A,r_u,\ell_u)$
  is an isotopy $(A,\diamond)\to(A,*)$ and also a similitude
  $\Comp(\Calg)\to\Comp(A,q,*)$ with multiplier of the form
  $(\lambda_1,\lambda_2,1)$ for some $\lambda_1$, $\lambda_2\in
  F^\times$. Since $\mu(\ell_u)=\mu(r_u)=q(u)=1$, it follows
  from~\eqref{eq:lambdamu} that $\lambda_1=\lambda_2=1$, hence $f$ is
  an isomorphism of compositions of quadratic spaces. Moreover for
  $e=u\diamond u$ we have $r_u^{-1}(e)=u=\ell_u^{-1}(e)$, hence
  \[
    e*x=u\diamond\ell_u^{-1}(x)=
    \ell_u\bigl(\ell_u^{-1}(x)\bigr) 
    = x = r_u\bigl(r_u^{-1}(x)\bigr)=r_u^{-1}(x)\diamond u =
    x*e,
  \]
  hence $(A,q,e,*)$ is a unital composition algebra.
\end{proof}

\begin{corol}
  \label{corol:2.30}
  Every composition of isometric quadratic spaces is isomorphic to a
  composition $\Comp(\Calg)$ as in~\eqref{eq:CompCalg} for some unital
  composition algebra $\Calg$, and also to a composition
  $\Comp(\mathcal{S})$ for some symmetric composition algebra
  $\mathcal S$. Up to isomorphism, there is a unique composition of
  hyperbolic quadratic spaces of dimension~$n$, for $n=2$, $4$ and $8$.
\end{corol}

\begin{proof}
  Let $\Comp=\bigl((V_1,q_1),(V_2,q_2), (V_3,q_3),*_3\bigr)$ be a
  composition of isometric quadratic spaces, and let
  $S(\Comp)=\bigl((V_3,q_3),(V_3,q_3),(V_3,q_3),\circledast_3\bigr)$
  be the composition of quadratic spaces constructed in
  Corollary~\ref{corol:pcomp}, which is isomorphic to $\Comp$. Clearly,
  $S(\Comp)=\Comp(\mathcal{S})$ for $\mathcal{S}$ the composition
  algebra $\mathcal{S}=(V_3,q_3,\circledast_3)$. This composition
  algebra is symmetric since $\partial S(\Comp)=S(\Comp)$. Now,
  Proposition~\ref{prop:Kaplan} 
  yields a unital composition algebra $\Calg=(V_3,q_3,e,*)$ and an
  isomorphism $\Comp(\mathcal{S})\simeq\Comp(\Calg)$. Thus, $\Comp$ is
  isomorphic to $\Comp(\mathcal{S})$ and to $\Comp(\Calg)$.

  For $n=2$, $4$ and $8$, there is up to isomorphism a unique unital
  composition algebra of dimension~$n$ with hyperbolic quadratic form,
  namely the split quadratic \'etale algebra, the split quaternion
  algebra and the split octonion algebra,
  see~\cite[(33.19)]{BoI}. Therefore, there is a unique composition of
  hyperbolic quadratic spaces up to isomorphism.
\end{proof}

By contrast with the last argument, note that there exist more than one
symmetric composition algebra of dimension~$8$ with hyperbolic
quadratic form, see~\cite[(34.37)]{BoI}; but their associated
compositions of quadratic spaces are isomorphic.
\medbreak

We next relate isotopies of composition algebras with similitudes of
the associated compositions of quadratic spaces.

\begin{thm}
  \label{thm:isot}
  Let $\Calg=(A,q,\diamond)$ and $\tilde\Calg=(\tilde A, \tilde q,
  \tilde\diamond)$ be composition algebras and let $\Comp(\Calg)$ and
  $\Comp(\tilde\Calg)$ be the associated compositions of quadratic
  spaces as in~\eqref{eq:CompCalg}. A triple $f=(f_1,f_2,f_3)$ of
  bijective linear maps $f_i\colon A\to\tilde A$ is an isotopy
  $(A,\diamond) \to(\tilde A,\tilde\diamond)$ if and only if it is a
  similitude $f\colon\Comp(\Calg)\to\Comp(\tilde\Calg)$ with
  composition multiplier $(\mu(f_2),\mu(f_1),1)$.
\end{thm}

\begin{proof}
  It is clear by comparing Definitions~\ref{defn:simcomp2} and
  \ref{defn:isotopy} that 
  similitudes $f\colon\Comp(\Calg)\to\Comp(\tilde\Calg)$ with
  composition multiplier of the form $(\lambda_1,\lambda_2,1)$ for
  some $\lambda_1$, $\lambda_2\in F^\times$ are isotopies
  $(A,\diamond)\to(\tilde A,\tilde\diamond)$. To prove the converse, it
  suffices to show that for every isotopy $f=(f_1,f_2,f_3)\colon
  (A,\diamond) \to (\tilde A,\tilde\diamond)$ the maps $f_i$ are
  similitudes. The isotopy condition then shows that $f$ is a
  similitude of compositions $\Comp(\Calg)\to\Comp(\tilde\Calg)$ with
  multiplier of the form $(\lambda_1,\lambda_2,1)$ for some
  $\lambda_1$, $\lambda_2\in F^\times$. Then~\eqref{eq:lambdamu} shows
  that $\lambda_1=\mu(f_2)$ and $\lambda_2=\mu(f_1)$.

  By Proposition~\ref{prop:Kaplan} we can find a bilinear map $*\colon
  A\times A\to A$ and a vector $e\in A$ such that $(A,q,e,*)$ is a
  unital composition algebra and there exists an isotopy
  $g\colon(A,\diamond) \to (A,*)$ whose components are isometries of
  $(A,q)$. The map $f\circ g^{-1}\colon(A,*)\to(\tilde
  A,\tilde\diamond)$ is then an isotopy, which means that for $f\circ
  g^{-1}=(h_1,h_2,h_3)$
  \begin{equation}
    \label{eq:hisot}
    h_3(x*y)=h_1(x)\,\tilde\diamond\, h_2(y)
    \qquad\text{for all $x$, $y\in A$.}
  \end{equation}
  In particular, it follows that for all $x\in A$
  \[
    h_1(e)\,\tilde\diamond\, h_2(x) = h_3(e*x) = h_3(x) = h_3(x*e) =
    h_1(x)\,\tilde\diamond\, h_2(e).
  \]
  Since $\tilde\Calg$ is a composition algebra, these equations imply
  \begin{equation}
    \label{eq:hisot2}
    \tilde q\bigl(h_3(x)\bigr) = \tilde q\bigl(h_1(e)\bigr) \tilde
    q\bigl(h_2(x)\bigr) = \tilde q\bigl(h_1(x)\bigr) \tilde
    q\bigl(h_2(e)\bigr)
    \qquad\text{for all $x\in A$.}
  \end{equation}
  As $h_3$ is bijective, there exist vectors $x\in A$ such that
  $\tilde q\bigl(h(x)\bigr)\neq0$, hence $\tilde
  q\bigl(h_1(e)\bigr)\neq0$ and $\tilde
  q\bigl(h_2(e)\bigr)\neq0$. Equation~\eqref{eq:hisot2} then yields
  $\tilde q\bigl(h_3(e)\bigr)=\tilde q\bigl(h_1(e)\bigr) \tilde
  q\bigl(h_2(e)\bigr)\neq0$. Define $q'\colon A\to F$ by
  \[
    q'(x)=\tilde q\bigl(h_3(e)\bigr)^{-1}\tilde q\bigl(h(x)\bigr)
    \qquad\text{for $x\in A$,}
  \]
  so $q'(e)=1$. Since $\tilde\Calg$ is a composition algebra, we
  obtain from~\eqref{eq:hisot}
  \[
    \tilde q\bigl(h_3(x*y)\bigr) = \tilde q\bigl(h_1(x)\bigr) \tilde
    q\bigl(h_2(y)\bigr) \qquad\text{for all $x$, $y\in A$,}
  \]
  hence by~\eqref{eq:hisot2}
  \[
    \tilde q\bigl(h_3(x*y)\bigr) = \tilde q\bigl(h_1(e)\bigr)^{-1}
    \tilde q\bigl(h_2(e)\bigr)^{-1} \tilde q\bigl(h_3(x)\bigr) \tilde
    q\bigl(h_3(y)\bigr)
    =
    \tilde q\bigl(h_3(e)\bigr)^{-1} \tilde q\bigl(h_3(x)\bigr) \tilde
    q\bigl(h_3(y)\bigr).
  \]
  Therefore,
  \[
    q'(x*y)=q'(x)q'(y) \qquad\text{for all $x$, $y\in A$.}
  \]
  Thus, $(A,q',e,*)$ is a unital composition algebra, just like
  $(A,q,e,*)$. But the quadratic form in a unital composition algebra
  is uniquely determined as the ``generic norm'' of the algebra
  (see~\cite[(33.9)]{BoI} or \cite[Cor.~1.2.4]{SpV}), hence $q'=q$,
  which means that
  \[
    \tilde q\bigl(h_3(x)\bigr) = \tilde q\bigl(h_3(e)\bigr) q(x)
    \qquad\text{for all $x\in A$.}
  \]
  Thus, $h_3\colon (A,q)\to(\tilde A,\tilde q)$ is a similitude with
  multiplier $\tilde q\bigl(h_3(e)\bigr)$. Equation~\eqref{eq:hisot2}
  then yields
  \[
    \tilde q\bigl(h_1(x)\bigr) = \tilde q\bigl(h_2(e)\bigr)^{-1}
    \tilde q\bigl(h_3(e)\bigr) q(x)
    \quad\text{and}\quad
    \tilde q\bigl(h_2(x)\bigr) = \tilde q\bigl(h_1(e)\bigr)^{-1}
    \tilde q\bigl(h_3(e)\bigr) q(x),
  \]
  hence $h_1$ and $h_2$ also are similitudes. Now, $f=h\circ g$ and
  all the components of $g$ are isometries, hence all the components
  of $f$ are similitudes.
\end{proof}

In the case where $\tilde\Calg=\Calg$ Theorem~\ref{thm:isot} has the
following immediate consequence:

\begin{corol}
  \label{corol:isot}
  For any composition algebra $\Calg=(A,q,\diamond)$, let
  $\lambda'\colon \BGO\bigl(\Comp(\Calg)\bigr)\to\BGm$ be the third
  component of the composition multiplier map
  $\lambda_{\Comp(\Calg)}\colon 
  \BGO\bigl(\Comp(\Calg)\bigr)\to\BGm^3$, and let $\mu'\colon
  \BStr(A,\diamond)\to \BGm^2$ be the map defined on rational points
  by mapping every autotopy $(f_1,f_2,f_3)$ to the pair of multipliers
  $\bigl(\mu(f_1),\mu(f_2)\bigr)$. The algebraic group scheme
  $\BStr(A,\diamond)$ is smooth and fits in the following exact
  sequences: 
  \[
    1\to\BStr(A,\diamond) \to \BGO\bigl(\Comp(\Calg)\bigr)
    \xrightarrow{\lambda'} \BGm\to 1
  \]
  and
  \[
    1\to \BOrth\bigl(\Comp(\Calg)\bigr) \to \BStr(A,\diamond)
    \xrightarrow{\mu'} \BGm^2\to 1.
  \]
\end{corol}

\begin{proof}
  Theorem~\ref{thm:isot} identifies $\BStr(A,\diamond)$ as the kernel
  of $\lambda'$. Proposition~\ref{prop:exseqOGOogo} shows that
  $\lambda_{\Comp(\Calg)}$ is a separable morphism, hence $\lambda'$
  also is separable, and it follows by~\cite[(22.13)]{BoI} that
  $\BStr(A,\diamond)$ is smooth. Theorem~\ref{thm:isot} also shows
  that the kernel of $\mu'$ is the kernel of
  the restriction of $\lambda_{\Comp(\Calg)}$ to
  $\BStr(A,\diamond)$, which is $\BOrth\bigl(\Comp(\Calg)\bigr)$ by
  definition. To complete the proof, observe that $\mu'$ is surjective
  because $\lambda_{\Comp(\Calg)}$ is surjective.
\end{proof}

We next turn to automorphisms of an algebra $(A,\diamond)$, which are
linear bijections $f\colon A\to A$ such that $f(x\diamond
y)=f(x)\diamond f(y)$ for all $x$, $y\in A$. They form an algebraic
group scheme $\BAut(A,\diamond)$, which can be viewed as a closed
subgroup of $\BStr(A,\diamond)$ since every automorphism $f$ yields an
autotopy $(f,f,f)$ of $(A,\diamond)$. To relate the condition that
$f_1=f_2=f_3$ in an autotopy $(f_1,f_2,f_3)$ with the shift
isomorphism $\partial\colon\BGO\bigl(\Comp(\Calg)\bigr) \to
\BGO\bigl(\partial\Comp(\Calg)\bigr)$, we view
$\BOrth\bigl(\Comp(\Calg)\bigr)$,
$\BOrth\bigl(\partial\Comp(\Calg)\bigr)$ and
$\BOrth\bigl(\partial^2\Comp(\Calg)\bigr))$ as subgroups of
$\BGL(A)\times \BGL(A)\times\BGL(A)$ and define
\[
  \overline\BOrth\bigl(\Comp(\Calg)\bigr) =
  \BOrth\bigl(\Comp(\Calg)\bigr) \cap
  \BOrth\bigl(\partial\Comp(\Calg)\bigr) \cap
  \BOrth\bigl(\partial^2\Comp(\Calg)\bigr).
\]
The shift isomorphism $\partial$ clearly restricts to an automorphism
of $\overline\BOrth\bigl(\Comp(\Calg)\bigr)$. 

\begin{prop}
  \label{prop:automorph}
  For every composition algebra $\Calg=(A,q,\diamond)$, the group
  $\BAut(A,\diamond)$ is the subgroup of
  $\overline\BOrth\bigl(\Comp(\Calg)\bigr)$ fixed under $\partial$.
\end{prop}

\begin{proof}
  Let $R$ be a commutative $F$-algebra. For every automorphism $f$ of
  $(A,\diamond)_R$ the triple $(f,f,f)$ is an autotopy of
  $(A,\diamond)_R$, hence by Theorem~\ref{thm:isot} a similitude of
  $\Comp(\Calg)_R$ with composition multiplier
  $(\mu(f),\mu(f),1)$. The relations~\eqref{eq:lambdamu} between
  composition multipliers and the multipliers of the components of
  similitudes yield $\mu(f)=\mu(f)^2$, hence $\mu(f)=1$ and $(f,f,f)$
  is an automorphism of $\Comp(\Calg)_R$. By Proposition~\ref{prop:simdef2} we
  then see that $(f,f,f)$ is also an automorphism of
  $\partial\Comp(\Calg)$ and of $\partial^2\Comp(\Calg)$, hence
  $(f,f,f)\in\overline\BOrth\bigl(\Comp(\Calg)\bigr)(R)$, and this
  triple is fixed under $\partial$. On the other hand, if
  $(f_1,f_2,f_3)\in\overline\BOrth\bigl(\Comp(\Calg)\bigr)(R)$ is
  fixed under $\partial$, then $f_1=f_2=f_3$, hence $f_1$ is an
  automorphism of $(A,\diamond)_R$ because $(f_1,f_2,f_3)\in
  \BOrth\bigl(\Comp(\Calg)\bigr)(R)$. 
\end{proof}

When $\Calg=(A,q,\diamond)$ is a symmetric composition algebra, then
$\partial\Comp(\Calg) = \Comp(\Calg)$ by definition, hence
$\overline\BOrth\bigl(\Comp(\Calg)\bigr) =
\BOrth\bigl(\Comp(\Calg)\bigr)$ and Proposition~\ref{prop:automorph}
shows that $\BAut(A,\diamond)$ is the subgroup of
$\BOrth\bigl(\Comp(\Calg)\bigr)$ fixed under $\partial$. When $\dim
A=8$, an alternative description is given
in~\cite[Th.~6.6]{CKT}: $\BAut(A,\diamond)$ is shown to be isomorphic
to the subgroup of $\BPGO^+(q)$ fixed under an outer automorphism of
order~$3$, which is the analogue of $\partial$.
\medbreak

For a unital composition algebra
$\Calg^\bullet=(A,q,e,\diamond)$ with associated para-unital symmetric
composition algebra $S(\Calg^\bullet)=(A,q,*)$ as in
Proposition~\ref{prop:dercompunit}, it follows from functoriality of
the $S$ construction that $\BAut(A,\diamond)\subset \BAut(A,*)$. The
reverse inclusion holds when $\dim A\geq4$
by~\cite[(34.4)]{BoI}. However, the group $\BAut(A,\diamond)$ can also
be described as follows:

\begin{prop}
  Let $\Calg^\bullet=(A,q,e,\diamond)$ be a unital composition
  algebra, and let $\Comp^\bullet(\Calg^\bullet)$ be the associated
  composition of pointed quadratic spaces as
  in~\eqref{eq:pCompCalg}. There is a canonical identification
  $\BAut(A,\diamond)=\BOrth\bigl(\Comp^\bullet(\Calg^\bullet)\bigr)$. 
\end{prop}

\begin{proof}
  Let $R$ be a commutative $F$-algebra. Every automorphism
  $f\in\BAut(A,\diamond)(R)$ leaves $e$ fixed, hence the triple
  $(f,f,f)$ is an automorphism of $\Comp^\bullet(\Calg^\bullet)_R$.
  Therefore, mapping $f$ to $(f,f,f)$ defines an embedding
  $\BAut(A,\diamond)\subset\BOrth\bigl(\Comp^\bullet(\Calg^\bullet)\bigr)$.

  For the reverse inclusion, let
  $(f_1,f_2,f_3)\in\BOrth\bigl(\Comp^\bullet(\Calg^\bullet)\bigr)(R)$.
  Substituting $e$ for $x$ or for $y$ in the
  equation
  \[
    f_3(x\diamond y)=f_1(x)\diamond f_2(y)\qquad\text{for all $x$,
      $y\in A$}
  \]
  yields $f_3(y)=f_2(y)$ and $f_3(x)=f_1(x)$ for all $x$, $y\in A$,
  hence $f_3\in\BAut(A,\diamond)(R)$.
\end{proof}

\section{Trialitarian triples}
\label{chap:tritri}

The focus in this section is on central simple algebras with quadratic
pair of degree~$8$ over an arbitrary field $F$. Altering slightly the
definition in \cite[\S42.A]{BoI} (and extending it to
characteristic~$2$), we define a \emph{trialitarian 
  triple} over $F$ to be a $4$-tuple (!)
\[
  \TT=\bigl((A_1,\sigma_1,\strf_1),\,(A_2,\sigma_2,\strf_2),\,
  (A_3,\sigma_3,\strf_3),\,\varphi_0\bigr)
\]
where $(A_i,\sigma_i,\strf_i)$ is a central simple $F$-algebra with
quadratic pair of degree~$8$ for $i=1$, $2$, $3$, and $\varphi_0$ is
an isomorphism of algebras with quadratic pair
\[
  \varphi_0\colon(C(A_1,\sigma_1,\strf_1),\underline\sigma_1, \underline
  \strf_1) \stackrel\sim\to (A_2,\sigma_2,\strf_2) \times
  (A_3,\sigma_3,\strf_3). 
\]
To simplify notation, we denote by a single letter algebras with
quadratic pair, as in Section~\ref{chap:Cligrps}, and write
$\aqp_i=(A_i,\sigma_i,\strf_i)$ and 
$\Clqp(\aqp_1)=(C(\aqp_1),\underline\sigma_1,\underline\strf_1)$.
\medbreak

If $\widetilde{\TT}=(\tilde\aqp_1,\, \tilde\aqp_2,\,
\tilde\aqp_3,\,\tilde\varphi_0)$ is also a trialitarian triple, 
an \emph{isomorphism of trialitarian triples} $\gamma\colon 
\TT \to \widetilde{\TT}$ is a triple $\gamma=(\gamma_1, \gamma_2,
\gamma_3)$ of isomorphisms of algebras with quadratic pair
\[
  \gamma_i\colon \aqp_i\to\tilde\aqp_i \qquad i=1, 2, 3,
\]
such that the following diagram commutes:
\[
  \xymatrix{\Clqp(\aqp_1)\ar[r]^-{\varphi_0}\ar[d]_{C(\gamma_1)}
    & \aqp_2\times \aqp_3\ar[d]^{\gamma_2\times\gamma_3}\\
    \Clqp(\tilde\aqp_1)
    \ar[r]^-{\widetilde\varphi_0}
    &
    \tilde \aqp_2\times\tilde \aqp_3}
\]

We show in \S\ref{subsec:triplecomp} that every composition $\Comp$ of
quadratic spaces of dimension~$8$ yields a trialitarian triple
$\End(\Comp)$, and that every trialitarian triple of split algebras
has the form $\End(\Comp)$ for some composition $\Comp$ of
dimension~$8$. In \S\ref{subsec:simTT} we discuss the group scheme of
automorphisms
of a trialitarian triple $\TT$: we
show that $\BAut(\TT)$ is smooth (hence an algebraic group) and
introduce algebraic groups $\BOrth(\TT)$, $\BGO(\TT)$, $\BPGO(\TT)$,
extending to the context of trialitarian triples the group schemes
$\BOrth(\Comp)$, $\BGO(\Comp)$, $\BPGO(\Comp)$ defined
in~\S\ref{subsec:simiso} for a composition $\Comp$ of quadratic
spaces. A main result of the section is the construction of derived
trialitarian triples in \S\ref{subsec:dertriple}: to each trialitarian
triple $\TT$ we canonically associate trialitarian triples
$\partial\TT$ and $\partial^2\TT$, in such a way that for split
trialitarian triples $\End(\Comp)$
\[
  \partial\End(\Comp)=\End(\partial\Comp) \qquad\text{and}\qquad
  \partial^2\End(\Comp)=\End(\partial^2\Comp).
\]
This construction is used in \S\ref{subsec:triso} to define for each
trialitarian triple $\TT=(\aqp_1,\aqp_2,\aqp_3,\varphi_0)$
isomorphisms 
\[
  \BOrth(\TT)\simeq\BSpin(\aqp_1) \simeq\BSpin(\aqp_2)
  \simeq\BSpin(\aqp_3)
\]
and
\[
  \BPGO(\TT)\simeq \BPGO^+(\aqp_1)\simeq \BPGO^+(\aqp_2) \simeq
  \BPGO^+(\aqp_3),
\]
which we call the \emph{trialitarian isomorphisms} canonically
attached to the trialitarian triple $\TT$.

We return in \S\ref{subsec:comp8} to the study of compositions of
quadratic spaces, building on the theory of trialitarian triples
developed in the previous subsections to obtain a few more results
about the 
$8$-dimensional case. Specifically, we establish criteria for the
similarity or the isomorphism of compositions of quadratic spaces,
which yield an analogue of the classical Principle of Triality, and we give an explicit description of the cohomological invariants of
$\BSpin_8$.

In the final \S\ref{subsec:strgrp8} we show that the constructions of
\S\ref{subsec:simTT} readily yield a canonical isomorphism between the
structure group of a composition algebra of dimension~$8$ and the
extended Clifford group of its quadratic form.

\subsection{The trialitarian triple of a composition of quadratic
  spaces}
\label{subsec:triplecomp}

Let
\[
  \Comp=\bigl((V_1,q_1),\,(V_2,q_2),\,(V_3,q_3),\,*_3\bigr)
\]
be a composition of quadratic spaces of dimension~$8$ over $F$. Recall
from Proposition~\ref{prop:compdim8} the isomorphism of algebras with
quadratic pair
\[
  C(\alpha)\colon(C(V_1,q_1),\tau_1,\str{g}_1) \xrightarrow\sim
  (\End(V_2\oplus 
  V_3), \sigma_{b_2\perp b_3},\strf_{q_2\perp q_3})
\]
induced by the map
\[
  \alpha\colon x_1\in V_1 \mapsto
  \begin{pmatrix}
    0&r_{x_1}\\ \ell_{x_1}&0
  \end{pmatrix} \in \End(V_2\oplus V_3)
\]
where $\ell_{x_1}\colon V_2\to V_3$ carries $x_2\in V_2$ to
$x_1*_3x_2\in V_3$ and $r_{x_1}\colon V_3\to V_2$ carries $x_3\in V_3$
to $x_3*_2x_1\in V_2$. Its restriction to the even Clifford algebra
also is an isomorphism of algebras with quadratic pair
\[
  C_0(\alpha)\colon (C_0(V_1,q_1),\tau_{01},\str{g}_{01})
  \xrightarrow{\sim} 
  (\End(V_2),\sigma_{b_2},\strf_{q_2})\times
  (\End(V_3),\sigma_{b_3},\strf_{q_3}),
\]
see Proposition~\ref{prop:compdim8}. Therefore, the following is a
trialitarian triple:
\[
  \End(\Comp) = \bigl((\End(V_1),\sigma_{b_1}, \strf_{q_1}),\,
  (\End(V_2),\sigma_{b_2}, \strf_{q_2}),\,
  (\End(V_3), \sigma_{b_3}, \strf_{q_3}),\,
  C_0(\alpha)\bigr).
\]

We next show that the construction of trialitarian
triples from compositions of quadratic spaces is functorial. 

Let $\Comp= \bigl((V_1,q_1),\,(V_2,q_2),\,(V_3,q_3),\,*_3\bigr)$ and
$\tilde\Comp=\bigl((\tilde V_1,\tilde q_1),\,(\tilde V_2,\tilde
q_2),\,(\tilde V_3,\tilde q_3),\,\tilde*_3\bigr)$ be
compositions of quadratic spaces of dimension~$8$. Recall that for
every linear isomorphism $g_i\colon V_i\to \tilde V_i$, we define
\[
  \Int(g_i)\colon\End(V_i)\to\End(\tilde V_i)\qquad\text{by $f\mapsto
    g_i\circ f\circ g_i^{-1}$.}
\]

\begin{prop}
  \label{prop:isoTT}
  For every similitude $(g_1,g_2,g_3)\colon\Comp\to\tilde\Comp$, the triple
  \[
    \Int(g_1,g_2,g_3):=\bigl(\Int(g_1),\Int(g_2),\Int(g_3)\bigr)
    \colon\End(\Comp)\to\End(\tilde\Comp) 
  \]
  is an isomorphism of trialitarian triples.
  Moreover, for every isomorphism of trialitarian triples
  $(\gamma_1,\gamma_2,\gamma_3)\colon\End(\Comp)\to\End(\tilde\Comp)$,
  there exists a similitude $(g_1,g_2,g_3)\colon\Comp\to\tilde\Comp$ such
  that
  \[
    (\gamma_1,\gamma_2,\gamma_3) = \Int(g_1,g_2,g_3).
  \]
\end{prop}

\begin{proof}
  Suppose that $(g_1,g_2,g_3)\colon\Comp\to\tilde\Comp$ is a
  similitude. For $i=1$, $2$, $3$, $\Int(g_i)$
  is an isomorphism of algebras with quadratic pairs
  \[
    \Int(g_i)\colon(\End(V_i),\sigma_{b_i},\strf_{q_i})
    \xrightarrow{\sim} 
    (\End(\tilde V_i),\sigma_{\tilde b_i},\strf_{\tilde q_i}).
  \]
  Note that under the identification
  $C(\End(V_1),\sigma_{b_1},\strf_{q_1}) = C_0(V_1,q_1)$ the
  isomorphism 
  $C\bigl(\Int(g_1)\bigr)$ induced by $\Int(g_1)$ is the isomorphism
  $C_0(g_1)\colon C_0(V_1,q_1)\to C_0(\tilde V_1,\tilde q_1)$ that
  maps $x_1y_1$ 
  to $\mu_1^{-1}g_1(x_1)g_1(y_1)$ for $x_1$, $y_1\in V_1$, where
  $\mu_1$ is the multiplier of $g_1$. Therefore, 
  in order to show that $\bigl(\Int(g_1),\Int(g_2),\Int(g_3)\bigr)$ is
  an 
  isomorphism of trialitarian triples $\End(\Comp)\to\End(\tilde\Comp)$, we
  have to show that the following diagram commutes:
  \begin{equation}
    \label{eq:diag1}
    \begin{aligned}
    \xymatrix{C_0(V_1,q_1)\ar[d]_{C_0(g_1)} \ar[r]^-{C_0(\alpha)}
      & \End(V_2)\times\End(V_3)\ar[d]^{\Int(g_2)\times\Int(g_3)}\\
      C_0(\tilde V_1,\tilde q_1)\ar[r]^-{C_0(\tilde\alpha)}&
      \End(\tilde V_2)\times\End(\tilde V_3)}
    \end{aligned}
  \end{equation}
  Let $(\lambda_1,\lambda_2,\lambda_3)$ be the
  composition 
  multiplier of $(g_1,g_2,g_3)$. For $x_1$, $y_1\in V_1$ we have 
  \[
    C_0(\tilde\alpha)\circ C_0(g_1)(x_1\cdot y_1) = \mu_1^{-1}
    \begin{pmatrix}
      r_{g_1(x_1)}\ell_{g_1(y_1)}&0\\
      0& \ell_{g_1(x_1)}r_{g_1(y_1)}
    \end{pmatrix}
  \]
  and
  \[
    \bigl(\Int(g_2)\times\Int(g_3)\bigr)\circ C_0(\alpha)
    (x_1\cdot y_1) =
    \begin{pmatrix}
      g_2r_{x_1}\ell_{y_1}g_2^{-1}&0\\0&
      g_3\ell_{x_1}r_{y_1}g_3^{-1}
    \end{pmatrix}.
  \]
  By Proposition~\ref{prop:simdef2}, $(g_3,g_1,g_2)$ is a similitude
  of $\partial^2\Comp$, hence for $x_1\in V_1$ and $x_2\in
  V_2$
  \begin{multline*}
    g_2r_{x_1}\ell_{y_1}(x_2) = g_2\bigl((y_1*_3x_2)*_2x_1\bigr) =\\
    \lambda_2^{-1}g_3(y_1*_3x_2)*_2g_1(x_1) =
    \lambda_2^{-1}\lambda_3^{-1}
    \bigl(g_1(y_1)*_3g_2(x_2)\bigr)*_2g_1(x_1).  
  \end{multline*}
  Since $\lambda_2\lambda_3=\mu_1$ by~\eqref{eq:lambdamu}, it follows
  that 
  $g_2r_{x_1}\ell_{y_1} =
  \mu_1^{-1}r_{g_1(x_1)}\ell_{g_1(y_1)}g_2$. Similarly,
  $g_3\ell_{x_1}r_{y_1} = \ell_{g_1(x_1)}r_{g_1(y_1)}g_3$, hence
  diagram~\eqref{eq:diag1} commutes. The first part of the
  proposition is thus proved.

  Now, assume
  $(\gamma_1,\gamma_2,\gamma_3)\colon
  \End(\Comp)\to\End(\tilde\Comp)$ is an
  isomorphism of trialitarian triples. Each $\gamma_i$ is an
  isomorphism
  \[
    \gamma_i\colon(\End(V_i),\sigma_{b_i},\strf_{q_i}) \stackrel\sim\to
    (\End(\tilde V_i),\sigma_{\tilde b_i},\strf_{\tilde q_i});
  \]
  Proposition~\ref{prop:simiso} shows that
  $\gamma_i=\Int(g_i)$ for some similitude $g_i\colon(V_i,q_i) \to
  (\tilde V_i,\tilde q_i)$. We 
  may then also consider the isomorphism
  \[
    \Int\bigl(
    \begin{smallmatrix}
      g_2&0\\0&g_3
    \end{smallmatrix}
    \bigr)
    \colon\End(V_2\oplus V_3)\to \End(\tilde V_2\oplus \tilde V_3), 
  \]
  which makes the following diagram, where the vertical maps are the
  diagonal embeddings, commute:
  \[
    \xymatrix{\End(V_2)\times\End(V_3)\ar[rr]^{\Int(g_2)\times\Int(g_3)}
      \ar@{^{(}->}[d]&
      &\End(\tilde V_2)\times\End(\tilde V_3)\ar@{^{(}->}[d]\\
      \End(V_2\oplus V_3) \ar[rr]^{\Int\bigl(
    \begin{smallmatrix}
      g_2&0\\0&g_3
    \end{smallmatrix}
    \bigr)}&&\End(\tilde V_2\oplus \tilde V_3)}
  \]
  From the hypothesis that $(\gamma_1,\gamma_2,\gamma_3)$ is an
  isomorphism of trialitarian triples, it follows that the
  diagram~\eqref{eq:diag1} commutes. Write $\mu_1$ for the multiplier
  of $g_1$ and consider the linear map
  \[
    \beta\colon V_1\to\End(\tilde V_2\oplus \tilde V_3),\qquad
    x_1\mapsto
    \begin{pmatrix}
      0&r_{g_1(x_1)}\\
      \mu_1^{-1}\ell_{g_1(x_1)}&0
    \end{pmatrix}.
  \]
  For $\tilde x_2\in \tilde V_2$ and $\tilde x_3\in \tilde V_3$, we
  have by~\eqref{eq:compp5} and \eqref{eq:compp7}
  \[
    \mu_1^{-1}(g_1(x_1)\,\tilde*_3\,\tilde x_2)\,\tilde*_2\,g_1(x_1) =
    \mu_1^{-1}\tilde q_1\bigl(g_1(x_1)\bigr)\tilde x_2
  \]
  and
  \[
    \mu_1^{-1}g_1(x_1)\,\tilde*_3\,\bigl(\tilde x_3\,\tilde*_2\,g_1(x_1)\bigr) =
    \mu_1^{-1}\tilde q_1\bigl(g_1(x_1)\bigr)\tilde x_3.
  \]
  Since $\mu_1^{-1}\bigl(\tilde q_1(g_1(x_1))\bigr)=q_1(x_1)$, it
  follows that 
  $\beta(x_1)^2=q_1(x_1)\Id_{\tilde V_2\oplus \tilde V_3}$ for $x_1\in
  V_1$. Therefore, $\beta$ induces an $F$-algebra homomorphism
  \[
    C(\beta)\colon C(V_1,q_1)\to\End(\tilde V_2\oplus \tilde V_3).
  \]
  Since $C(V_1,q_1)$ is a simple algebra, dimension count shows that
  $C(\beta)$ is an isomorphism. For $x_1$, $y_1\in V_1$,
  \[
    C(\beta)(x_1y_1) =
    \begin{pmatrix}
      \mu_1^{-1}r_{g_1(x_1)}\ell_{g_1(y_1)}&0\\
      0&\mu_1^{-1}\ell_{g_1(x_1)}r_{g_1(y_1)}
    \end{pmatrix}
    =C_0(\tilde\alpha)\bigl(C_0(g_1)(x_1y_1)\bigr),
  \]
  hence $C(\beta)\rvert_{C_0(V_1,q_1)} = C_0(\tilde\alpha)\circ C_0(g_1)$.
  Since the diagram~\eqref{eq:diag1} commutes, it follows that
  \[
    C(\beta)\rvert_{C_0(V_1,q_1)} = \Int\bigl(
    \begin{smallmatrix}
      g_2&0\\0&g_3
    \end{smallmatrix}
    \bigr)\circ
    C(\alpha)\rvert_{C_0(V_1,q_1)}.
  \]
  Therefore, $\Int\bigl(
    \begin{smallmatrix}
      g_2&0\\0&g_3
    \end{smallmatrix}
    \bigr)\circ C(\alpha)\circ C(\beta)^{-1}$ is an
  automorphism of $\End(\tilde V_2\oplus \tilde V_3)$ whose restriction to
  $C(\beta)\bigl(C_0(V_1,q_1)\bigr)$ is the identity. This
  automorphism is inner by the Skolem--Noether theorem. Since
  $C(\beta)\bigl(C_0(V_1,q_1)\bigr)=\End(\tilde V_2)\times\End(\tilde V_3)$ we
  must have
  \[
    \Int\bigl(
    \begin{smallmatrix}
      g_2&0\\0&g_3
    \end{smallmatrix}
    \bigr)\circ C(\alpha)\circ C(\beta)^{-1} = \Int\bigl(
    \begin{smallmatrix}
      \nu_2&0\\0&\nu_3
    \end{smallmatrix}\bigr)
    \qquad\text{for some $\nu_2$, $\nu_3\in F^\times$,}
  \]
  hence
  \[
    \Int\bigl(
    \begin{smallmatrix}
      g_2&0\\0&g_3
    \end{smallmatrix}
    \bigr)\circ C(\alpha)(x_1) =
    \bigl(
    \begin{smallmatrix}
      \nu_2&0\\0&\nu_3
    \end{smallmatrix}\bigr) C(\beta)(x_1) \bigl(
    \begin{smallmatrix}
      \nu_2^{-1}&0\\0&\nu_3^{-1}
    \end{smallmatrix}\bigr)
    \qquad\text{for $x_1\in V_1$,}
  \]
  which means that
  \[
    \begin{pmatrix}
      0& g_2r_{x_1}g_3^{-1}\\ g_3\ell_{x_1}g_2^{-1}&0
    \end{pmatrix}
    =
    \begin{pmatrix}
      0&\nu_2\nu_3^{-1}r_{g_1(x_1)}\\
      \nu_2^{-1}\nu_3\mu_1^{-1}\ell_{g_1(x_1)}&0
    \end{pmatrix}.
  \]
  The equation
  $g_3\ell_{x_1}g_2^{-1}=\nu_2^{-1}\nu_3\mu_1^{-1}\ell_{g_1(x_1)}$
  implies that for $x_2\in V_2$
  \[
    g_3(x_1*_3x_2) = \nu_2^{-1}\nu_3\mu_1^{-1} g_1(x_1)\,\tilde*_3\,g_2(x_2).
  \]
  Therefore, $(g_1,g_2,g_3)$ is a similitude
  $\Comp\to\tilde\Comp$.
\end{proof}

We next show that every trialitarian triple of split algebras has the
form $\End(\Comp)$ for some composition $\Comp$ of quadratic spaces of
dimension~$8$. 

\begin{thm}
  \label{thm:splitrial}
  Let $\TT=(\aqp_1,\,\aqp_2,\,\aqp_3,\,\varphi_0)$ be a trialitarian
  triple over an arbitrary field $F$, where
  $\aqp_i=(A_i,\sigma_i,\strf_i)$ for $i=1$, $2$, $3$. If $A_1$, $A_2$
  and $A_3$ are split, then 
  there is a composition $\Comp$ of quadratic spaces of dimension~$8$
  over $F$ such that $\TT\simeq\End(\Comp)$. The composition $\Comp$
  is uniquely determined up to similitude.
\end{thm}

\begin{proof}
  For $i=1$, $2$, $3$, let $A_i=\End(V_i)$ for some $F$-vector space
  $V_i$ of dimension~$8$. Let also $q_i$ be a quadratic form on $V_i$
  to which $(\sigma_i,\strf_i)$ is adjoint. Since $q_1$ is determined
  only 
  up to a scalar factor, we may assume $q_1$ represents~$1$ and pick
  $e_1\in V_1$ such that $q_1(e_1)=1$. The inner automorphism
  $\Int(e_1)$ of the full Clifford algebra $C(V_1,q_1)$ preserves
  $C_0(V_1,q_1)$ and is of order~$2$. It transfers under the
  isomorphism
  \[
    C_0(V_1,q_1)=C(\aqp_1)\xrightarrow{\varphi_0}
    A_2\times 
    A_3= \End(V_2)\times\End(V_3)
  \]
  to an automorphism of $\End(V_2)\times\End(V_3)$ that interchanges
  the two factors. Viewing $\End(V_2)\times\End(V_3)$ as a subalgebra
  diagonally embedded in $\End(V_2\oplus V_3)$, we may find an inner
  automorphism of $\End(V_2\oplus V_3)$ which restricts to
  $\varphi_0\circ\Int(e_1)\circ\varphi_0^{-1}$ by (a slight
  generalization of) the Skolem--Noether Theorem, see \cite[Th.~2,
  p.~A~VIII.252]{Bou}. This inner automorphism is conjugation by an
  operator of the form $\bigl(
  \begin{smallmatrix}
    0&u'\\ u&0
  \end{smallmatrix}\bigr)$ since it interchanges $\bigl(
  \begin{smallmatrix}
    \Id_{V_2}&0\\0&0
  \end{smallmatrix}\bigr)$ and $\bigl(
  \begin{smallmatrix}
    0&0\\0&\Id_{V_3}
  \end{smallmatrix}\bigr)$. Since
  $\varphi_0\circ\Int(e_1)\circ\varphi_0^{-1}$ has order~$2$, it
  follows that $uu'=u'u\in F^\times$, hence $\Int\bigl(
  \begin{smallmatrix}
    0&u^{-1}\\ u&0
  \end{smallmatrix}\bigr)$ has the same restriction to
  $\End(V_2)\times\End(V_3)$ as $\Int\bigl(
  \begin{smallmatrix}
    0&u'\\ u&0
  \end{smallmatrix}\bigr)$. Representing $C(V_1,q_1)$ and
  $\End(V_2\oplus V_3)$ as (generalized) crossed products
  \begin{align*}
    C(V_1,q_1) & = C_0(V_1,q_1)\oplus e_1 C_0(V_1,q_1),
    \\
    \End(V_2\oplus V_3) & = \bigl(\End(V_2)\times\End(V_3)\bigr)
    \oplus\bigl(
    \begin{smallmatrix}
      0&u^{-1}\\ u&0
    \end{smallmatrix}\bigr)
    \bigl(\End(V_2)\times\End(V_3)\bigr),
  \end{align*}
  we may extend $\varphi_0$ to an isomorphism of $F$-algebras
  \[
    \varphi\colon C(V_1,q_1)\to \End(V_2\oplus V_3)
  \]
  by mapping $e_1$ to $\bigl(
  \begin{smallmatrix}
    0&u^{-1}\\ u&0
  \end{smallmatrix}\bigr)$. Let $\tau_1$
  be the involution on $C(V_1,q_1)$ that fixes every vector in $V_1$
  and let 
  $\tau'=\varphi\circ\tau_1\circ\varphi^{-1}$ be the corresponding
  involution on $\End(V_2\oplus V_3)$. The restriction of $\tau_1$ to
  $C_0(V_1,q_1)$ is the canonical involution $\tau_{01}$, and
  $\varphi_0\circ\tau_{01}=
  (\sigma_2\times\sigma_3)\circ\varphi_0$, hence $\tau'$ restricts to
  $\sigma_2$ and $\sigma_3$ on $\End(V_2)$ and $\End(V_3)$. This means
  that
  \[
    (\End(V_2\oplus V_3),\tau')\in (\End(V_2),\sigma_2) \boxplus
    (\End(V_3),\sigma_3),
  \]
  i.e., that $\tau'$ is adjoint to a symmetric bilinear form
  that is the orthogonal sum of a multiple of $b_2$ and a multiple
  of $b_3$. Scaling $q_2$ or $q_3$, we may assume
  $\tau'=\sigma_{b_2\perp b_3}$ is the adjoint involution of $b_2\perp b_3$.

  Under the isomorphism $\varphi$, the odd part
  $C_1(V_1,q_1)=e_1C_0(V_1,q_1)$ is mapped to the odd part of
  $\End(V_2\oplus V_3)$ for the checkerboard grading, hence for each
  $x_1\in V_1$ there exist $\ell_{x_1}\in \Hom(V_2,V_3)$ and
  $r_{x_1}\in\Hom(V_3,V_2)$ such that
  \[
    \varphi(x_1)=
    \begin{pmatrix}
      0&r_{x_1}\\ \ell_{x_1}&0
    \end{pmatrix}\in\End(V_2\oplus V_3).
  \]
  Since $\tau_1(x_1)=x_1$, it follows that $\varphi(x_1)$ is
  $\sigma_{b_2\perp b_3}$-symmetric, hence for all $x_2$, $y_2\in V_2$
  and $x_3$, $y_3\in V_3$
  \[
    (b_2\perp b_3)\left(
      \begin{pmatrix}
        0&r_{x_1}\\ \ell_{x_1}&0
      \end{pmatrix}
      \begin{pmatrix}
        x_2\\ x_3
      \end{pmatrix},\,
      \begin{pmatrix}
        y_2\\ y_3
      \end{pmatrix}
    \right) =
    (b_2\perp b_3)\left(
      \begin{pmatrix}
        x_2\\ x_3
      \end{pmatrix},\,
      \begin{pmatrix}
        0&r_{x_1}\\ \ell_{x_1}&0
      \end{pmatrix}
      \begin{pmatrix}
        y_2\\ y_3
      \end{pmatrix}
    \right).
  \]
  This means that for all $x_2$, $y_2\in V_2$ and $x_3$, $y_3\in V_3$,
  \begin{equation}
    \label{eq:trial19}
    b_2(r_{x_1}(x_3),y_2)=b_3\bigl(x_3,\ell_{x_1}(y_2)\bigr)
    \quad\text{and}\quad
    b_3(\ell_{x_1}(x_2),y_2) = b_2\bigl(x_2, r_{x_1}(y_3)\bigr).
  \end{equation}
  Moreover, the relations $x_1^2=q_1(x_1)$ and
  $x_1y_1+y_1x_1=b_1(x_1,y_1)$ yield for all $x_1$, $y_1\in V_1$
  \begin{equation}
    \label{eq:trial20}
    \ell_{x_1}r_{x_1}=r_{x_1}\ell_{x_1}=q_1(x_1) \quad\text{and}\quad
    \ell_{x_1}r_{y_1}+\ell_{y_1}r_{x_1} = r_{x_1}\ell_{y_1}+
    r_{y_1}\ell_{x_1} = b_1(x_1,y_1).
  \end{equation}

  Recall that the two components $\varphi_\pm$ of $\varphi_0$ are
  homomorphisms of algebras with quadratic pair
  \[
    \varphi_+\colon(C_0(V_1,q_1),\tau_{01},\underline \strf_1)
    \to \aqp_2,\qquad
    \varphi_-\colon(C_0(V_1,q_1),\tau_{01}, \underline \strf_1)
    \to \aqp_3.
  \]
  As observed in Definition~\ref{defn:lift},
  $\varphi_+\bigl(\xclie(q_1)\bigr)\subset\go(q_2)$ and
  $\varphi_-\bigl(\xclie(q_1)\bigr)\subset\go(q_3)$, hence
  \[
    \varphi_+(x_1y_1)\in\go(q_2) \qquad\text{and}\qquad
    \varphi_-(x_1y_1)\in\go(q_3)\qquad\text{for all $x_1$, $y_1\in
      V_1$.}
  \]
  The definition of $\varphi$ yields
  $\varphi_+(x_1y_1)=r_{x_1}\ell_{y_1}$ and $\varphi_-(x_1y_1) =
  \ell_{x_1}r_{y_1}$, hence by~\eqref{eq:trial20}
  \begin{align*}
    \dot\mu\bigl(\varphi_+(x_1y_1)\bigr)
    & = \varphi_+(x_1y_1) + \varphi_+(y_1x_1) = b_1(x_1,y_1)\\
    \intertext{and}
    \dot\mu\bigl(\varphi_-(x_1y_1)\bigr) & = \varphi_-(x_1y_1) +
    \varphi_-(y_1x_1) = b_1(x_1,y_1).
  \end{align*}
  Since $r_{x_1}\ell_{y_1}\in\go(q_2)$ and
  $\ell_{x_1}r_{y_1}\in\go(q_3)$, it follows from
  Proposition~\ref{prop:Lie1bis} 
  that for $x_1$, 
  $y_1\in  V_1$, $x_2\in V_2$ and $x_3\in V_3$
  \begin{equation}
    \label{eq:trial21}
    b_2(r_{x_1}\ell_{y_1}(x_2),x_2) = b_1(x_1,y_1)q_2(x_2)
    \quad\text{and}\quad
    b_3(\ell_{x_1}r_{y_1}(x_3),x_3) = b_1(x_1,y_1)q_3(x_3).
  \end{equation}
  If $x_1\in V_1$ is nonzero, there exists $y_1\in V_1$ such that
  $b_1(x_1,y_1)=1$. From~\eqref{eq:trial19} and \eqref{eq:trial20}
  we derive for all $x_2\in V_2$
  \[
    b_3\bigl(\ell_{x_1}(x_2),
    \ell_{x_1}r_{y_1}\ell_{x_1}(x_2)\bigr) =
    b_2\bigl(r_{x_1}\ell_{x_1}(x_2),r_{y_1}\ell_{x_1}(x_2)\bigr) =
    q_1(x_1) b_2\bigl(x_2,r_{y_1}\ell_{x_1}(x_2)\bigr).
  \]
  But \eqref{eq:trial21} yields
  \[
    b_3\bigl(\ell_{x_1}(x_2),
    \ell_{x_1}r_{y_1}\ell_{x_1}(x_2)\bigr)
    =q_3\bigl(\ell_{x_1}(x_2)\bigr) \quad\text{and}\quad
    b_2\bigl(x_2,r_{y_1}\ell_{x_1}(x_2)\bigr) = q_2(x_2),
  \]
  hence
  \[
    q_3\bigl(\ell_{x_1}(x_2)\bigr) = q_1(x_1)q_2(x_2)
    \qquad\text{for all $x_1\in V_1$, $x_2\in V_2$ with $x_1\neq0$.}
  \]
  This equation obviously also holds for $x_1=0$. Therefore, defining
  \[
    *_3\colon V_1\times V_2\to V_3 \qquad\text{by}\quad
    x_1*_3x_2=\ell_{x_1}(x_2)\qquad\text{for $x_1\in V_1$ and $x_2\in
      V_2$,}
  \]
  we see that $*_3$ is a composition of $(V_1,q_1)$, $(V_2,q_2)$ and
  $(V_3,q_3)$. Let also
  \[
    x_3*_2x_1=r_{x_1}(x_3)\qquad\text{for $x_3\in V_3$ and $x_1\in
      V_1$.}
  \]
  From~\eqref{eq:trial19} it follows that $b_2(x_3*_2x_1,x_2) =
  b_3(x_3,x_1*_3x_2)$ for $x_1\in V_1$, $x_2\in V_2$ and $x_3\in
  V_3$, hence Proposition~\ref{prop:defnder} shows that $*_2$ is the
  derived composition of $(V_3,q_3)$, 
  $(V_1,q_1)$ and $(V_2,q_2)$. Therefore, $\varphi_0=C_0(\alpha)$ for
  $\alpha\colon V_1\to\End(V_2\oplus V_3)$ mapping $x_1\in V_1$ to
  $\bigl(
  \begin{smallmatrix}
    0&r_{x_1}\\ \ell_{x_1}&0
  \end{smallmatrix}
  \bigr)$. We thus see that $\TT=\End(\Comp)$
  for $\Comp=\bigl((V_1,q_1),\,(V_2,q_2),\,(V_3,q_3),\,*_3\bigr)$.
  Proposition~\ref{prop:isoTT} shows that the composition $\Comp$ is
  uniquely determined up to similitude.
\end{proof}

\subsection{Similitudes of trialitarian triples}
\label{subsec:simTT}

Throughout this subsection, we fix a trialitarian triple
\[
  \TT=(\aqp_1,\,\aqp_2,\,\aqp_3,\,\varphi_0)
\]
with $\aqp_i=(A_1,\sigma_i,\strf_i)$ a central simple algebra with
quadratic pair of degree~$8$ over an arbitrary field $F$ for $i=1$,
$2$, $3$. 
The algebraic group scheme $\BAut(\TT)$ of
automorphisms of $\TT$ is defined as
follows: for any
commutative $F$-algebra $R$, the group $\BAut(\TT)(R)$ consists of the
triples $(\gamma_1,\gamma_2,\gamma_3)\in
\BAut(\aqp_1)(R)\times \BAut(\aqp_2)(R) \times \BAut(\aqp_3)(R)$ that
make the following square commute:
\begin{equation}
  \label{eq:diagdefauTT}
  \begin{aligned}
  \xymatrix{C(\aqp_1)_R\ar[r]^-{\varphi_0}\ar[d]_{C(\gamma_1)}&
    A_{2R}\times A_{3R}\ar[d]^{\gamma_2\times\gamma_3}\\
    C(\aqp_1)_R\ar[r]^-{\varphi_0}&A_{2R}\times A_{3R}
  }
  \end{aligned}
\end{equation}
Thus, 
\[
  \BAut(\TT)\subset \BAut(\aqp_1)\times \BAut(\aqp_2) \times
  \BAut(\aqp_3).
\]
Now, recall from \cite[\S23.B]{BoI} that the map $\Int\colon
\BGO(\aqp_i)\to \BAut(\aqp_i)$ defines an isomorphism
$\BPGO(\aqp_i)\xrightarrow{\sim} \BAut(\aqp_i)$. Therefore, we may
consider the inverse image of $\BAut(\TT)$ under the surjective
morphism
\[
  \Int\colon \BGO(\aqp_1) \times \BGO(\aqp_2) \times
  \BGO(\aqp_3) \to \BAut(\aqp_1)\times \BAut(\aqp_2) \times
  \BAut(\aqp_3).
\]

\begin{definition}
  \label{defn:GOTT}
  The algebraic group scheme of \emph{similitudes} of the trialitarian
  triple $\TT$ is
  \[
    \BGO(\TT)=\Int^{-1}\bigl(\BAut(\TT)\bigr) \subset \BGO(\aqp_1)
    \times \BGO(\aqp_2) \times  \BGO(\aqp_3).
  \]
  From this definition, it follows that the map $\Int$ restricts to a
  surjective morphism (see~\cite[(22.4)]{BoI})
  \[
    \Int\colon \BGO(\TT)\to \BAut(\TT).
  \]
  Its kernel is the algebraic group of \emph{homotheties}
  $\BHomot(\TT)=\BGm^3$, which lies in the center of $\BGO(\TT)$.
  We may therefore consider the quotient
  \[
    \BPGO(\TT)=\BGO(\TT)/\BHomot(\TT)\subset \BPGO(\aqp_1)\times
    \BPGO(\aqp_2)\times \BPGO(\aqp_3),
  \]
  and the map $\Int$ yields an
  isomorphism  
  \[
    \overline\Int\colon \BPGO(\TT)\xrightarrow{\sim} \BAut(\TT).
  \]
\end{definition}

Our goal in this subsection is to define a subgroup
$\BOrth(\TT)\subset\BGO(\TT)$ on the same model as the subgroup
$\BOrth(\Comp)$ of the group $\BGO(\Comp)$ of similitudes of a
composition of quadratic spaces, so that when $\TT=\End(\Comp)$ for
some 
composition $\Comp$ of quadratic spaces of dimension~$8$
we may identify
\[
  \BOrth(\TT)=\BOrth(\Comp), \qquad \BGO(\TT)=\BGO(\Comp)
  \quad\text{and}\quad \BPGO(\TT)=\BPGO(\Comp);
\]
see Proposition~\ref{prop:simEndComp}. Moreover, for an arbitrary
trialitarian triple $\TT$, we 
relate $\BGO(\TT)$ to the extended Clifford group $\BOmega(\aqp_1)$ to
obtain canonical isomorphisms
\[
  \BSpin(\aqp_1)\xrightarrow{\sim} \BOrth(\TT) \qquad\text{and}\qquad
  \BPGO(\TT)\xrightarrow{\sim}\BPGO^+(\aqp_1),
\]
see Theorems~\ref{thm:psiTTrev} and \ref{thm:SpinOTT}.
\medbreak

A key tool is the following construction: let
$\varphi_+\colon\Clqp(\aqp_1)\to \aqp_2$ and 
$\varphi_-\colon\Clqp(\aqp_1)\to \aqp_3$ be the two components of the
isomorphism $\varphi_0\colon\Clqp(\aqp_1)\to\aqp_2\times\aqp_3$, which
is part of the structure of $\TT$. Recall
from~\eqref{eq:commdiaghomoClif2} that $\varphi_+$ and $\varphi_-$
restrict to morphisms
\[
  \varphi_+\colon\BOmega(\aqp_1)\to \BGO^+(\aqp_2)
  \qquad\text{and}\qquad
  \varphi_-\colon\BOmega(\aqp_1)\to \BGO^+(\aqp_3).
\]
Combine $\varphi_+$ and $\varphi_-$ with the morphism
$\chi_0\colon\BOmega(\aqp_1)\to\BGO^+(\aqp_1)$ of~\S\ref{subsec:ClGrp}
to obtain a morphism
\begin{equation}
  \label{eq:defnpsiTT}
  \psi_\TT\colon \BOmega(\aqp_1)\to\BGO(\TT)
\end{equation}
as follows: for every commutative $F$-algebra $R$ and
$\xi\in\BOmega(\aqp_1)(R)$, let
\[
  \psi_\TT(\xi) = \bigl(\chi_0(\xi),\, \varphi_+(\xi),\,
  \varphi_-(\xi)\bigr) \in \BGO^+(\aqp_1)(R) \times \BGO^+(\aqp_2)(R)
  \times \BGO^+(\aqp_3)(R).
\]
Proposition~\ref{prop:ClifinXClifbis} shows that
$C\bigl(\Int(\chi_0(\xi))\bigr) = \Int(\xi)$, hence
\[
  \varphi_0\circ C\bigl(\Int(\chi_0(\xi))\bigr) \circ \varphi_0^{-1} =
  \Int\bigl(\varphi_0(\xi)\bigr) = \Int\bigl(\varphi_+(\xi)\bigr)
  \times \Int\bigl(\varphi_-(\xi)\bigr),
\]
which means that $\bigl(\Int(\chi_0(\xi)),\, \Int(\varphi_+(\xi)), \,
\Int(\varphi_-(\xi))\bigr)$ lies in $\BAut(\TT)(R)$, and therefore
$\psi_\TT(\xi)\in \BGO(\TT)(R)$. Note that $\psi_\TT$ is injective,
since $\bigl(\varphi_+(\xi),\,\varphi_-(\xi)\bigr)=\varphi_0(\xi)$ and
$\varphi_0$ is an isomorphism.

We first use the map $\psi_\TT$ to prove:

\begin{thm}
  \label{thm:psiTTrev}
  Projection on the first component $\pi_\TT\colon\BPGO(\TT) \to
  \BPGO(\aqp_1)$ defines an isomorphism
  \[
    \BPGO(\TT)\xrightarrow{\sim} \BPGO^+(\aqp_1).
  \]
\end{thm}

\begin{proof}
  Let $R$ be a commutative $F$-algebra and
  $(\gamma_1,\gamma_2,\gamma_3)\in\BAut(\TT)(R)$. Since $\varphi_0$ is
  an isomorphism, $\gamma_2$ and $\gamma_3$ are uniquely determined by
  $\gamma_1$ and commutativity of the
  diagram~\eqref{eq:diagdefauTT}. Therefore, $\pi_\TT$ is
  injective. Moreover, commutativity of the
  diagram~\eqref{eq:diagdefauTT} shows that $C(\gamma_1)$ leaves the
  center of $C(\aqp_1)$ fixed, which means that $\gamma_1$ lies in the
  connected component of the identity $\BAut^+(\aqp_1)(R)$. Therefore,
  the image of $\pi_\TT$ lies in~$\BPGO^+(\aqp_1)$.

  To complete the proof, we show that $\pi_\TT$ is surjective on
  $\BPGO^+(\aqp_1)$. Since $\BPGO^+(\aqp_1)$ is smooth, it suffices to
  consider rational points over an algebraic closure $\Falg$ of $F$,
  by~\cite[(22.3)]{BoI}. Recall from
  Proposition~\ref{prop:chi0ontobis} that $\chi_0$ is surjective. For
  every 
  $g_1\in\BGO^+(\aqp_1)(\Falg)$, we may therefore find
  $\xi\in\BOmega(\aqp_1)(\Falg)$ such that $\chi_0(\xi)=g_1$. Then
  $\psi_\TT(\xi)\in \BGO(\TT)(\Falg)$, and its image
  $\overline\psi_\TT(\xi)$ in $\BPGO(\TT)(\Falg)$ satisfies
  $\pi_\TT\bigl(\overline\psi_\TT(\xi)\bigr) = g_1\Falg^\times$.
  We have thus found an element in
  $\BPGO(\TT)(\Falg)$ that maps under $\pi_\TT$ to any given
  $g_1\Falg^\times\in\BPGO^+(\aqp_1)(\Falg)$, hence $\pi_\TT$ is
  surjective. 
\end{proof}

\begin{corol}
  \label{cor:smoothBPGO}
  The algebraic group schemes $\BGO(\TT)$ and $\BPGO(\TT)$ are smooth
  and connected.
\end{corol}

\begin{proof}
  That $\BPGO(\TT)$ is smooth and connected readily follows from the
  theorem, since $\BPGO^+(\aqp_1)$ is smooth and connected
  by~\cite[\S23.B]{BoI}. Then $\BGO(\TT)$ is also smooth and connected
  because $\BPGO(\TT)=\BGO(\TT)/\BHomot(\TT)$ with $\BHomot(\TT)$
  smooth and connected, see~\cite[(22.12)]{BoI}.
\end{proof}

We next use $\psi_\TT$ to determine the structure of $\BGO(\TT)$. Let $Z_1\simeq F\times F$ denote the center of $C(\aqp_1)$, and recall
that $R_{Z_1/F}(\BGm)$ lies in the center of $\BOmega(\aqp_1)$
(see~\eqref{eq:commdiagClif}). For every commutative $F$-algebra $R$
and $z\in(Z_1)_R^\times$, Proposition~\ref{prop:ClifinXClifbis} yields
$\chi_0(z)=N_{Z_1/F}(z)$, while $\varphi_+(z)$, $\varphi_-(z)\in
R^\times$. Therefore,
\begin{equation}
  \label{eq:psiTTonZ}
    \psi_\TT(z) = (N_{Z_1/F}(z),\,\varphi_+(z),\,\varphi_-(z)\bigr)
    \in R^\times\times R^\times\times R^\times =\BHomot(\TT)(R).
\end{equation}

\begin{prop}
  \label{prop:OmegaH}
  The morphism $\psi_\TT$ and the inclusion $i\colon \BHomot(\TT)\to
  \BGO(\TT)$ define an isomorphism
  \[
    \overline{\psi_\TT\times i}\colon \bigl(\BOmega(\aqp_1)\times
    \BHomot(\TT)\bigr)/ R_{Z_1/F}(\BGm) \xrightarrow{\sim} \BGO(\TT),
  \]
  where $R_{Z_1/F}(\BGm)$ is embedded into the product canonically in
  the first factor and by the inversion followed by $\psi_\TT$ in the
  second (so the copies of $R_{Z_1/F}(\BGm)$ in $\BOmega(\aqp_1)$ and
  in $\BHomot(\TT)$ are identified in the quotient).
\end{prop}

\begin{proof}
  It is clear from the definition of the quotient that the morphism
  $\overline{\psi_\TT\times i}$ is defined.
  To prove that it is injective, consider an arbitrary commutative
  $F$-algebra $R$ and pick $\xi\in\BOmega(\aqp_1)(R)$ and
  $\nu=(\nu_1,\nu_2,\nu_3)\in \BHomot(\TT)(R)$ such that
  $\psi_\TT(\xi)\cdot\nu=1$ in $\BGO(\TT)(R)$, i.e.,
  \[
    \chi_0(\xi)=\nu_1^{-1},\qquad \varphi_+(\xi) = \nu_2^{-1}
    \quad\text{and}\quad \varphi_-(\xi)=\nu_3^{-1}.
  \]
  The last two equations show that $\varphi_0(\xi)=(\nu_2^{-1},
  \nu_3^{-1})$ in $(A_2)_R\times (A_3)_R$. Since $\varphi_0$ is an
  isomorphism, it follows that $\xi$ lies in
  $(Z_1)_R^\times$, hence $(\xi,\nu)$ is trivial in the quotient,
  for $\nu=\psi_\TT(\xi)^{-1}$.

  To complete the proof, it remains to show that
  $\overline{\psi_\TT\times i}$ is 
  surjective. Since $\BGO(\TT)$ is smooth by
  Corollary~\ref{cor:smoothBPGO}, it suffices to consider the groups
  of rational points over an algebraic closure $\Falg$ of $F$. Let
  $g=(g_1,g_2,g_3)\in\BGO(\TT)(\Falg)$. Note that $g_1$, $g_2$ and
  $g_3$ are proper similitudes, because $\BGO(\TT)$ is connected by
  Corollary~\ref{cor:smoothBPGO}. We know from
  Proposition~\ref{prop:chi0ontobis} that $\chi_0\colon
  \BOmega(\aqp_1)\to \BGO^+(\aqp_1)$ is surjective, hence we may find
  $\xi\in\BOmega(\aqp_1)(\Falg)$ such that $\chi_0(\xi)=g_1$. Then
  $\psi_\TT(\xi)=(g_1,g'_2,g'_3)$ for some
  $g'_2\in\BGO^+(\aqp_2)(\Falg)$ and
  $g'_3\in\BGO^+(\aqp_3)(\Falg)$. As $g$ and $\psi_\TT(\xi)$ belong to
  $\BGO(\TT)(\Falg)$, the following diagrams commute:
  \[
  \xymatrix{C(\aqp_1)_{\Falg}\ar[r]^-{\varphi_0}\ar[d]_{C(\Int(g_1))}
    & (A_2)_{\Falg}\times (A_3)_{\Falg}\ar[d]^{\Int(g_2)\times\Int(g_3)}\\
    C(\aqp_1)_{\Falg}
    \ar[r]^-{\varphi_0}
    &
    (A_2)_{\Falg}\times (A_3)_{\Falg}}
  \qquad
  \xymatrix{C(\aqp_1)_{\Falg}\ar[r]^-{\varphi_0}\ar[d]_{C(\Int(g_1))}
    & (A_2)_{\Falg}\times (A_3)_{\Falg}\ar[d]^{\Int(g'_2)\times\Int(g'_3)}\\
    C(\aqp_1)_{\Falg}
    \ar[r]^-{\varphi_0}
    &
    (A_2)_{\Falg}\times (A_3)_{\Falg}}
\]
Therefore, $\Int(g_2)=\Int(g'_2)$ and $\Int(g_3)=\Int(g'_3)$, which
implies that $g_2=g'_2\nu_2$ and $g_3=g'_3\nu_3$ for some $\nu_2$,
$\nu_3\in\Falg^\times$. With
$\nu=(1,\nu_2,\nu_3)\in\BHomot(\TT)(\Falg)$ we then have
\[
  \psi_\TT(\xi)\cdot\nu=(g_1,g'_2\nu_2,g'_3\nu_3) = (g_1,g_2,g_3) = g.
\]
Surjectivity of $\overline{\psi_\TT\times i}$ follows.
\end{proof}

\begin{corol}
  \label{cor:rhoTT}
  Let $m\colon \BHomot(\TT)\to\BGm$ denote the multiplication map
  carrying $(\nu_1,\nu_2,\nu_3)$ to $\nu_1\nu_2\nu_3$. There is a
  morphism $\rho_\TT\colon \BGO(\TT)\to\BGm$ uniquely determined by
  the condition that the following diagram commutes:
  \begin{equation}
    \label{eq:defrho}
    \begin{aligned}
    \xymatrix{
      \BOmega(\aqp_1)\times\BHomot(\TT)
      \ar[rr]^-{\psi_\TT\times i}
      \ar[dr]_{(\mu\circ\chi_0)\times m}
      &&
      \BGO(\TT)\ar[dl]^{\rho_\TT}\\
      &\BGm&}
  \end{aligned}
  \end{equation}
\end{corol}

\begin{proof}
  Proposition~\ref{prop:OmegaH} identifies $\BGO(\TT)$ with a quotient
  of $\BOmega(\aqp_1)\times\BHomot(\TT)$ by $R_{Z_1/F}(\BGm)$, hence
  to prove the existence and uniqueness of $\rho_\TT$ it suffices to
  show that $(\mu\circ\chi_0)$ and $m$ coincide on the images of
  $R_{Z_1/F}(\BGm)$ in $\BOmega(\TT)$ by inclusion and in
  $\BHomot(\TT)$ by $\psi_\TT$.
  For every commutative $F$-algebra $R$ and
  $z\in(Z_1)_R^\times$ 
  we have $\chi_0(z)=N_{Z_1/F}(z)$ by
  Proposition~\ref{prop:ClifinXClifbis}, hence
  $(\mu\circ\chi_0)(z)=N_{Z_1/F}(z)^2$. On the other hand
  $N_{Z_1/F}(z)=\varphi_+(z)\varphi_-(z)$, hence
  \[
    (m\circ\psi_\TT)(z)=
    m\bigl(N_{Z_1/F}(z),\,\varphi_+(z),\,\varphi_-(z)\bigr) =
    N_{Z_1/F}(z)^2.
  \]
  Thus, $(\mu\circ\chi_0)$ and $m$ coincide on the images of
  $R_{Z_1/F}(\BGm)$. 
\end{proof}

\begin{definition}
  \label{defn:lambdaOTT}
  A morphism
  $\lambda_\TT\colon\BGO(\TT)\to\BGm^3$ is defined as follows:
  for every commutative $F$-algebra $R$ and
  $g=(g_1,g_2,g_3)\in\BGO(\TT)(R)$, set
  \[
    \lambda_\TT(g) = \bigl(\rho_\TT(g)\mu(g_1)^{-1},\,
    \rho_\TT(g)\mu(g_2)^{-1},\,
    \rho_\TT(g)\mu(g_3)^{-1}\bigr) \in R^\times \times R^\times \times
    R^\times.
  \]
  From the definition of $\rho_\TT$, it follows that for
  $\nu=(\nu_1,\nu_2,\nu_3)\in\BHomot(\TT)(R)$
  \[
    \rho_\TT(\nu) = m(\nu) = \nu_1\nu_2\nu_3,
  \]
  hence
  \begin{equation}
    \label{eq:lambdaTTH}
    \lambda_\TT(\nu) = (\nu_2\nu_3\nu_1^{-1},\,
    \nu_3\nu_1\nu_2^{-1},\, \nu_1\nu_2\nu_3^{-1}).
  \end{equation}
  The definition of $\rho_\TT$ also yields
  $\rho_\TT\bigl(\psi_\TT(\xi)\bigr) = \mu\bigl(\chi_0(\xi)\bigr)$ for
  $\xi\in\BOmega(\aqp_1)(R)$. Letting
  $\underline\mu\colon\BOmega(\aqp_1)\to R_{Z_1/F}(\BGm)$ denote the
  multiplier map, we have by Proposition~\ref{prop:ClifinXClifbis}
  \[
    \mu\bigl(\chi_0(\xi)\bigr) =
    N_{Z_1/F}\bigl(\underline\mu(\xi)\bigr) =
    \varphi_+\bigl(\underline\mu(\xi)\bigr) \cdot
    \varphi_-\bigl(\underline\mu(\xi)\bigr).
  \]
  As $\varphi_0$ is an isomorphism of algebras with quadratic pair, we
  also have
  \begin{equation}
    \label{eq:lambdaTTpsi0}
    \bigl(\varphi_+(\underline\mu(\xi)),\,
    \varphi_-(\underline\mu(\xi))\bigr) =
    \varphi_0\bigl(\underline\mu(\xi)\bigr) =
    \bigl(\mu(\varphi_+(\xi)),\, \mu(\varphi_-(\xi))\bigr).
  \end{equation}
  Therefore, the definition of $\lambda_\TT$ yields
  \begin{equation}
        \label{eq:lambdaTTpsi}
  \begin{aligned}
    \lambda_\TT\bigl(\psi_\TT(\xi)\bigr)
    &= \bigl(
    1,\,\mu\bigl(\chi_0(\xi)\bigr)\mu\bigl(\varphi_+(\xi)\bigr)^{-1},
    \,
    \mu\bigl(\chi_0(\xi)\bigr)\mu\bigl(\varphi_-(\xi)\bigr)^{-1}\bigr)
    \\
    &= \bigl(1,\, \mu\bigl(\varphi_-(\xi)\bigr),\,
    \mu\bigl(\varphi_+(\xi)\bigr)\bigr). 
  \end{aligned}
  \end{equation}
\end{definition}

\begin{definition}
  Let
  \[
    \BOrth(\TT)=\ker(\lambda_\TT\colon \BGO(\TT)\to\BGm^3).
  \]
  As in the proof of Proposition~\ref{prop:exseqOGOogo}, it follows
  from~\eqref{eq:lambdaTTH}
  that the map $\lambda_\TT\colon\BHomot(\TT)\to\BGm^3$, hence also
  $\lambda_\TT\colon\BGO(\TT)\to\BGm^3$, is surjective. Therefore, the
  following sequence is exact: 
  \[
    1\to\BOrth(\TT)\to \BGO(\TT) \xrightarrow{\lambda_\TT} \BGm^3\to 1
  \]

  Now, let $\BZO(\TT)$ be the kernel of the canonical map
  $\BOrth(\TT)\to \BPGO(\TT)$, which is the composition of the
  inclusion $\BOrth(\TT)\subset\BGO(\TT)$ and the canonical
  epimorphism $\BGO(\TT)\to\BPGO(\TT)$. Thus, letting $m$ be the
  multiplication map $(\nu_1,\nu_2,\nu_3)\mapsto \nu_1\nu_2\nu_3$,
  \[
  \BZO(\TT)=\BHomot(\TT)\cap\BOrth(\TT) = \ker(m\colon
  \Bmu_2\times\Bmu_2\times\Bmu_2\to\Bmu_2)\simeq \Bmu_2\times\Bmu_2.
\]
The same arguments as in Proposition~\ref{prop:exdiagGOComp} yield the
following commutative diagram with exact rows and columns:
\begin{equation}
  \label{eq:diagZOPGO}
  \begin{aligned}
    \xymatrix{&1\ar[d]&1\ar[d]&&\\
      1\ar[r]&\BZO(\TT)\ar[r]\ar[d]&\BOrth(\TT)\ar[r]\ar[d]&
      \BPGO(\TT)\ar[r]\ar@{=}[d]&1\\
      1\ar[r]&\BHomot(\TT)\ar[r]\ar[d]_{\lambda_\TT}&
      \BGO(\TT)\ar[r]\ar[d]^{\lambda_\TT}&\BPGO(\TT)\ar[r]&1\\
      &\BGm^3\ar@{=}[r]\ar[d]&\BGm^3\ar[d]&&\\
      &1&1&&}
  \end{aligned}
\end{equation}
\end{definition}

Now, we show that the definitions above are compatible with the
corresponding definitions for compositions of quadratic spaces
in~\S\ref{subsec:simiso}. 

\begin{prop}
  \label{prop:simEndComp}
  For $\Comp$ any composition of quadratic spaces of dimension~$8$ and
  $\TT=\End(\Comp)$,
  canonical isomorphisms yield identifications
  \[
    \BHomot(\Comp) = \BHomot(\TT), \quad
    \BOrth(\Comp) = \BOrth(\TT),\quad
    \BGO(\Comp) = \BGO(\TT), \quad
    \BPGO(\Comp) = \BPGO(\TT).
  \]
  Moreover, the following diagram commutes:
  \begin{equation}
    \label{eq:diaglambdalambda}
    \begin{aligned}
    \xymatrix{
      \BGO(\Comp)\ar@{=}[r]\ar[d]_{\lambda_\Comp}&
      \BGO(\TT)\ar[d]^{\lambda_\TT}\\ 
      \BGm^3\ar@{=}[r]&\BGm^3}
    \end{aligned}
  \end{equation}
\end{prop}

\begin{proof}
  Let $R$ be a commutative $F$-algebra. For every
  $(g_1,g_2,g_3,\lambda_3)\in\BGO(\Comp)(R)$ the triple
  $(g_1,g_2,g_3)$ lies in $\BGO(\TT)(R)$, as seen in the first part of
  the proof of Proposition~\ref{prop:isoTT}. Since $\lambda_3$ is
  uniquely 
  determined by $g_1$, $g_2$ and $g_3$, mapping
  $(g_1,g_2,g_3,\lambda_3)$ to $(g_1,g_2,g_3)$ defines an injective
  map $\BGO(\Comp)\to\BGO(\TT)$.
  Proposition~\ref{prop:isoTT} also shows that for $\Falg$ an
  algebraic closure of $F$ the map $\BGO(\Comp)(\Falg)\to
  \BGO(\TT)(\Falg)$ is surjective. This is
  sufficient to prove that the map $\BGO(\Comp)\to
  \BGO(\TT)$ is surjective, since
  $\BGO(\TT)$ is smooth by
  Corollary~\ref{cor:smoothBPGO}. We have thus obtained
  a canonical isomorphism $\BGO(\Comp)\xrightarrow{\sim}
  \BGO(\TT)$. This isomorphism maps $\BHomot(\Comp)$
  to $\BHomot(\TT)$, hence it induces an isomorphism
  $\BPGO(\Comp)\xrightarrow{\sim}\BPGO(\TT)$.
  
  In order to prove that the isomorphism $\BGO(\Comp) = \BGO(\TT)$
  also maps $\BOrth(\Comp)$ 
  to $\BOrth(\TT)$, it suffices to prove that the
  diagram~\eqref{eq:diaglambdalambda} is commutative. For this, we use
  the description of $\BGO(\TT)$ in Proposition~\ref{prop:OmegaH} as a
  quotient of the product of $\BOmega(\aqp_1)$ and $\BHomot(\TT)$. It
  is clear from~\eqref{eq:lambdaTTH} that $\lambda_\Comp$ and
  $\lambda_\TT$  coincide on the image of
  $\BHomot(\Comp)=\BHomot(\TT)$ in $\BGO(\Comp)=\BGO(\TT)$. Therefore,
  it suffices to consider the image of $\BOmega(\aqp_1)$ under
  $\psi_\TT$.
  
  Let
  $\Comp=\bigl((V_1,q_1),\,(V_2,q_2),\,(V_3,q_3),\,*_3\bigr)$, so
  $\varphi_0=C_0(\alpha)$,
  $A_1=\End V_1$ and
  $\BOmega(\aqp_1)=\BOmega(q_1)$.
  Let $R$ be a commutative $F$-algebra and
  let $\xi\in\BOmega(q_1)(R)$.
  To simplify notation, write $g_1=\chi_0(\xi)\in\BGO^+(q_1)(R)$,
  $g_2=C_+(\alpha)(\xi)\in\BGO^+(q_2)(R)$ 
  and $g_3=C_-(\alpha)(\xi)\in\BGO^+(q_3)(R)$, so
  \[
    \psi_\TT(\xi) = (g_1,g_2,g_3).
  \]
  Now, \eqref{eq:lambdaTTpsi0} yields
  $C_0(\alpha)\bigl(\underline\mu(\xi)\bigr) = 
  \bigl(\mu(g_2),\mu(g_3)\bigr)$, and by~\eqref{eq:defchi0}, we
  have $g_1(x_1)=\iota\bigl(\underline\mu(\xi)\bigr)\xi x_1\xi^{-1}$
  for every $x_1\in V_{1R}$. By taking the image of each side of the
  last equation under $C(\alpha)$, we obtain
  \begin{equation*}
    \begin{pmatrix}
      0&r_{g_1(x_1)}\\ \ell_{g_1(x_1)}&0
    \end{pmatrix}
    =
    \begin{pmatrix}
      \mu(g_3)&0\\0&\mu(g_2)
    \end{pmatrix}
    \begin{pmatrix}
      g_2&0\\0&g_3
    \end{pmatrix}
    \begin{pmatrix}
      0&r_{x_1}\\ \ell_{x_1}&0
    \end{pmatrix}
    \begin{pmatrix}
      g_2^{-1}&0\\0&g_3^{-1}
    \end{pmatrix}.
  \end{equation*}
  This equation yields
  \[
    r_{g_1(x_1)}g_3=\mu(g_3)g_2r_{x_1} \qquad\text{and}\qquad
    \ell_{g_1(x_1)}g_2=\mu(g_2)g_3\ell_{x_1} \qquad\text{for all
      $x_1\in V_{1R}$},
  \]
  which means that for all $x_1\in V_{1R}$, $x_2\in V_{2R}$ and
  $x_3\in V_{3R}$ 
  \[
    g_3(x_3)*_2g_1(x_1) = \mu(g_3)g_2(x_3*_2x_1)
    \qquad\text{and}\qquad
    g_1(x_1)*_3g_2(x_2) = \mu(g_2)g_3(x_1*_3x_2).
  \]
  These equations show that
  $\bigl(g_1,g_2,g_3,\mu(g_2)\bigr)\in\BGO(\Comp)(R)$, hence
  by~\eqref{eq:lambdamu}
  \[
    \lambda_\Comp\bigl(\psi_\TT(\xi)\bigr) =
    \bigl(1,\mu(g_3),\mu(g_2)\bigr).
  \]
  Therefore, $\lambda_\Comp\bigl(\psi_\TT(\xi)\bigr) =
  \lambda_\TT\bigl(\psi_\TT(\xi)\bigr)$ by~\eqref{eq:lambdaTTpsi}, and
  the proof is complete.
\end{proof}

\begin{corol}
  \label{cor:OTTsmooth}
  For every trialitarian triple $\TT$, the algebraic group scheme
  $\BOrth(\TT)$ is smooth. 
\end{corol}

\begin{proof}
   Over an algebraic closure $\Falg$ of $F$ the trialitarian
  triple $\TT$ is split, hence by Theorem~\ref{thm:splitrial} we may
  find a composition $\Comp$ of 
  quadratic spaces of dimension~$8$ over $\Falg$ such that
  $\TT_{\Falg}\simeq\End(\Comp)$. Then
  $\BOrth(\TT_{\Falg})$ is isomorphic to $\BOrth(\Comp)$,
  which is smooth by Proposition~\ref{prop:exseqOGOogo}, hence
  $\BOrth(\TT)$ is smooth by~\cite[(21.10)]{BoI}.
\end{proof}

The final result in this subsection elucidates the structure of
$\BOrth(\TT)$.

\begin{thm}
  \label{thm:SpinOTT}
  For every trialitarian triple $\TT$, the morphism $\psi_\TT$
  restricts to an isomorphism
  \[
    \psi_\TT\colon \BSpin(\aqp_1) \xrightarrow{\sim} \BOrth(\TT).
  \]
\end{thm}

\begin{proof}
  When defining $\psi_\TT$, we already observed that this morphism is
  injective. 
  Recall from \S\ref{subsec:ClGrp} that $\BSpin(\aqp_1)$ is the kernel
  of $\underline\mu\colon\BOmega(\aqp_1)\to
  R_{Z_1/F}(\BGm)$. Therefore, \eqref{eq:lambdaTTpsi0} and
  \eqref{eq:lambdaTTpsi} show that $\psi_\TT$ map $\BSpin(\aqp_1)$ to
  $\BOrth(\TT)$.

  To prove that $\psi_\TT$ maps $\BSpin(\aqp_1)$ onto $\BOrth(\TT)$,
  it suffices to consider the groups of rational points over an
  algebraic closure $\Falg$ of $F$, because we know by
  Corollary~\ref{cor:OTTsmooth} that $\BOrth(\TT)$ is
  smooth. Proposition~\ref{prop:OmegaH} shows that $\psi_\TT\times
  i\colon \BOmega(\TT)(\Falg)\times \BHomot(\TT)(\Falg) \to
  \BGO(\TT)(\Falg)$ is surjective, hence for any
  $g\in\BOrth(\TT)(\Falg)$ we may find $\xi\in\BOmega(\TT)(\Falg)$ and
  $\nu=(\nu_1,\nu_2,\nu_3)\in\BHomot(\TT)(\Falg)$ such that
  $\psi_\TT(\xi)\cdot\nu=g$. Taking the image of each side under
  $\lambda_\TT$ and using~\eqref{eq:lambdaTTH}
  and~\eqref{eq:lambdaTTpsi}, we obtain 
  \[
    (1,\lambda_2,\lambda_3)\cdot(\nu_2\nu_3\nu_1^{-1},\,
    \nu_3\nu_1\nu_2^{-1},\, \nu_1\nu_2\nu_3^{-1}) = (1,1,1)
  \]
  for some $\lambda_2$, $\lambda_3\in\Falg^\times$, hence
  $\nu_1=\nu_2\nu_3$. Therefore, $\nu=\psi_\TT(z)$ for
  $z\in(Z_1)^\times_{\Falg}$ such that $\varphi_0(z)=(\nu_2,\nu_3)$,
  and $\psi_\TT(\xi z)=g$. Since $\lambda_\TT(g)=(1,1,1)$,
  \eqref{eq:lambdaTTpsi0} and \eqref{eq:lambdaTTpsi} show that
  $\underline\mu(\xi z)=1$, hence $\xi
  z\in\BSpin(\aqp_1)(\Falg)$. Thus, $\psi_\TT$ maps $\BSpin(\aqp_1)$
  onto $\BOrth(\TT)$.
\end{proof}

\begin{corol}
  \label{corol:commdiagSpinPGO}
  The following diagram, in which the vertical maps are isomorphisms,
  is commutative with exact rows:
  \begin{equation}
    \label{eq:commdiagSpinO}
    \begin{aligned}
    \xymatrix{1\ar[r]& R_{Z_1/F}(\Bmu_2)\ar[d]_{\psi_\TT}\ar[r]&
      \BSpin(\aqp_1)\ar[r]^-{\chi'}\ar[d]_{\psi_\TT}&
      \BPGO^+(\aqp_1)\ar[r]&1\\
      1\ar[r]&\BZO(\TT)\ar[r]&\BOrth(\TT)\ar[r]&
      \BPGO(\TT)\ar[r]\ar[u]_{\pi_\TT}&1}
    \end{aligned}
  \end{equation}
\end{corol}

\begin{proof}
  The upper sequence is~\eqref{eq:exseqBSpinBPGO}, and the lower
  sequence is from~\eqref{eq:diagZOPGO}. Commutativity of the right
  square follows from the definition of $\chi'$ as the composition of
  $\chi_0$ with the canonical map $\BGO^+(\aqp_1)\to\BPGO^+(\aqp_1)$,
  and bijectivity of the vertical maps is proved in
  Theorems~\ref{thm:psiTTrev} and \ref{thm:SpinOTT}.
\end{proof}

\subsection{Derived trialitarian triples}
\label{subsec:dertriple}

To every trialitarian triple $\TT=(\aqp_1,\aqp_2,\aqp_3,\varphi_0)$ we
attach in this subsection two derived trialitarian triples
\[
  \partial\TT = (\aqp_2,\aqp_3,\aqp_1, \varphi'_0)
  \quad\text{and}\quad
  \partial^2\TT = (\aqp_3,\aqp_1,\aqp_2,\varphi''_0)
\]
in such a way that for every composition $\Comp$ of quadratic spaces
of dimension~$8$
\[
  \partial\End(\Comp)=\End(\partial\Comp) \quad\text{and}\quad
  \partial^2\End(\Comp) = \End(\partial^2\Comp).
\]
The two components of the isomorphisms
\[
  \varphi_0'\colon \Clqp(\aqp_2)\to \aqp_3\times\aqp_1
  \quad\text{and}\quad
  \varphi_0''\colon\Clqp(\aqp_3)\to \aqp_1\times\aqp_2
\]
are determined as lifts (in the sense of Definition~\ref{defn:lift})
of Lie algebra homomorphisms
\[
  \theta'_+\colon\pgo(\aqp_2)\to\pgo(\aqp_3)
  \quad\text{and}\quad
  \theta''_+\colon\pgo(\aqp_3)\to\pgo(\aqp_1),
\]
\[
  \theta'_-\colon\pgo(\aqp_2)\to\pgo(\aqp_1)
  \quad\text{and}\quad
  \theta''_-\colon\pgo(\aqp_3)\to\pgo(\aqp_2).
\]

Our main result is the following:

\begin{thm}
  \label{thm:trial5rev}
  Let $\TT=(\aqp_1,\aqp_2,\aqp_3,\varphi_0)$ be a trialitarian triple
  over an arbitrary field $F$, and let
  \[
    \theta_+\colon\pgo(\aqp_1)\to\pgo(\aqp_2)
    \quad\text{and}\quad
    \theta_-\colon\pgo(\aqp_1)\to\pgo(\aqp_3)
  \]
  denote the Lie algebra homomorphisms induced (as per
  Definition~\ref{defn:lift}) by the two components of $\varphi_0$:
  \[
    \varphi_+\colon \Clqp_+(\aqp_1)\to \aqp_2
    \quad\text{and}\quad
    \varphi_-\colon \Clqp_-(\aqp_1)\to \aqp_3.
  \]
  The homomorphisms $\theta_+$ and $\theta_-$ are isomorphisms, and
  the following Lie algebra homomorphisms are liftable:
  \[
    \theta'_+=\theta_-\circ\theta_+^{-1},\qquad
    \theta'_-=\theta_+^{-1},\qquad
    \theta''_+=\theta_-^{-1},\qquad
    \theta''_-=\theta_+\circ\theta_-^{-1}.
  \]
  Moreover, $\theta'_+$ and $\theta'_-$ on one side, and $\theta''_+$
  and $\theta''_-$ on the other side, are of opposite signs (see
  Definition~\ref{defn:lift}). 
\end{thm}

Corollary~\ref{corol:trial3} shows that we may extend scalars to a
Galois extension of $F$ in order to show that a Lie algebra
homomorphism is liftable. We may thus reduce to split trialitarian
triples, i.e., triples of the form $\End(\Comp)$. We investigate this
case first. The proof of Theorem~\ref{thm:trial5rev} will quickly
follow after~\eqref{eq:diagtrial}.
\medbreak

Let $\Comp=\bigl((V_1,q_1),\,(V_2,q_2),\,(V_3,q_3),\,*_3\bigr)$ be a
composition of quadratic spaces of dimension~$8$ over an arbitrary
field~$F$. Recall from Proposition~\ref{prop:pgoComp} that the Lie
algebra $\pgo(\Comp)$ can be described as a subalgebra of
$\pgo(q_1)\times\pgo(q_2)\times \pgo(q_3)$. Let
\[
  \pi_1\colon\pgo(\Comp)\to\pgo(q_1),\qquad
  \pi_2\colon\pgo(\Comp)\to\pgo(q_2),\qquad
  \pi_3\colon\pgo(\Comp)\to\pgo(q_3)
\]
denote the projections on the three components, and let
\[
  \theta_+\colon\pgo(q_1)\to\pgo(q_2)
  \qquad\text{and}\qquad
  \theta_-\colon\pgo(q_1)\to\pgo(q_3)
\]
be the Lie algebra homomorphisms induced by the two components of
$C_0(\alpha)$,
\[
  C_+(\alpha)\colon C_+(V_1,q_1)\to \End(V_2)
  \qquad\text{and}\qquad
  C_-(\alpha)\colon C_-(V_1,q_1)\to \End(V_3).
\]

\begin{lemma}
  \label{lem:dertri1}
  The following diagram, where all the maps are isomorphisms, is
  commutative:
  \begin{equation}
    \label{eq:diagdertri1}
    \begin{aligned}
    \xymatrix{
    \pgo(q_1) \ar[rr]^{\theta_+}
    \ar[dddr]_{\theta_-}
    && \pgo(q_2)
    \\
    &\pgo(\Comp)\ar[ul]_{\pi_1}
    \ar[ur]^{\pi_2}
      \ar[dd]_<<<<<<{\pi_3}&\\
    &&\\
    &\pgo(q_3)&
  }
  \end{aligned}
  \end{equation}
\end{lemma}

\begin{proof}
  First, observe that $\pi_1$ is the differential of the morphism
  $\pi_{\End(\Comp)}$ under the identification $\BPGO(\Comp) =
  \BPGO\bigl(\End(\Comp)\bigr)$ of
  Proposition~\ref{prop:simEndComp}. The morphism $\pi_{\End(\Comp)}$
  is an isomorphism by Theorem~\ref{thm:psiTTrev}, hence $\pi_1$ is an
  isomorphism. Similarly, $\pi_2$ is the differential of the
  isomorphism obtained by the composition
  \[
    \BPGO(\Comp)\xrightarrow{\partial}\BPGO(\partial\Comp)
    \xrightarrow{\pi_{\End(\partial\Comp)}} \BPGO^+(q_2),
  \]
  hence $\pi_2$ is an isomorphism. Likewise, $\pi_3$ is an
  isomorphism.
  
  Now, recall from~\eqref{eq:commdiagxcliehomo} that $\theta_+$ and
  $\theta_-$ are defined by the following commutative diagrams, where
  $\overline{C_+(\alpha)}$ and $\overline{C_-(\alpha)}$ are obtained
  by composing $C_+(\alpha)$ and $C_-(\alpha)$ with the canonical
  homomorphisms $\go(q_2)\to\pgo(q_2)$ or $\go(q_3)\to\pgo(q_3)$:
  \[
    \xymatrix{
      \xclie(q_1)\ar[r]^{\dot\chi'}
      \ar[d]_{\overline{C_+(\alpha)}}&\pgo(q_1)\ar[dl]^{\theta_+}\\
      \pgo(q_2)&
    }
    \qquad
    \xymatrix{
      \xclie(q_1)\ar[r]^{\dot\chi'}
      \ar[d]_{\overline{C_-(\alpha)}}&\pgo(q_1)\ar[dl]^{\theta_-}\\
      \pgo(q_3)&
    }
  \]
  Therefore,
  \begin{equation}
    \label{eq:dertri1}
    \overline{C_+(\alpha)}=\theta_+\circ\dot\chi'
    \qquad\text{and}\qquad
    \overline{C_-(\alpha)}=\theta_-\circ\dot\chi'.
  \end{equation}

  Next, define a Lie algebra homomorphism
\[
  \Psi_\Comp\colon\xclie(q_1)\to\pgo(\Comp)
\]
by composing
the differential $\dot\psi_{\End(\Comp)}\colon \xclie(q_1)\to
\go(\Comp)$ of the morphism $\psi_{\End(\Comp)}\colon \BOmega(q_1) \to
\BGO(\Comp)$ of \S\ref{subsec:simTT} with the canonical map
$\go(\Comp)\to\pgo(\Comp)$. Explicitly,
\[
  \Psi_\Comp(\xi) = (\dot\chi_0(\xi)+F,\; C_+(\alpha)(\xi)+F,\;
  C_-(\alpha)(\xi)+F)\qquad\text{for $\xi\in\xclie(q_1)$,}
\]
or, since
Proposition~\ref{prop:xclie2new} shows that $\dot\chi'(\xi) =
\dot\chi_0(\xi)+F$,
\[
  \Psi_\Comp(\xi) = \bigl(\dot\chi'(\xi),\;
  \overline{C_+(\alpha)}(\xi),\; 
  \overline{C_-(\alpha)}(\xi)\bigr)
  \qquad\text{for $\xi\in\xclie(q_1)$.}
\]
It follows
from the definitions that the following diagrams are commutative:
\[
  \xymatrix{
    \xclie(q_1)\ar[dr]^{\Psi_\Comp}\ar[r]^{\dot\chi'}
    \ar[d]_{\overline{C_+(\alpha)}}&
    \pgo(q_1)\\
    \pgo(q_2)&\pgo(\Comp)\ar[u]_{\pi_1} \ar[l]^{\pi_2}}
  \qquad\qquad
  \xymatrix{
    \xclie(q_1)\ar[dr]^{\Psi_\Comp}\ar[r]^{\dot\chi'}
    \ar[d]_{\overline{C_-(\alpha)}}&
    \pgo(q_1)\\
    \pgo(q_3)&\pgo(\Comp)\ar[u]_{\pi_1} \ar[l]^{\pi_3}}
\]
Therefore,
\[
  \overline{C_+(\alpha)}=\pi_2\circ\Psi_\Comp,
  \qquad
  \dot\chi'=\pi_1\circ\Psi_\Comp,
  \qquad
  \overline{C_-(\alpha)}=\pi_3\circ\Psi_\Comp.
\]
Substituting in~\eqref{eq:dertri1} yields
\[
  \pi_2\circ\Psi_\Comp=\theta_+\circ\pi_1\circ\Psi_\Comp
  \qquad\text{and}\qquad
  \theta_-\circ\pi_1\circ\Psi_\Comp=
  \pi_3\circ\Psi_\Comp.
\]
We know from Proposition~\ref{prop:xclie2new} that $\dot\chi'$ is
surjective, hence $\Psi_\Comp$ also is surjective since $\pi_1$ is an
isomorphism. Therefore, the last displayed equations yield
$\pi_2=\theta_+\circ\pi_1$ and $\pi_3=\theta_-\circ\pi_1$, proving the
commutativity of diagram~\eqref{eq:diagdertri1}. Bijectivity of
$\theta_+$ and $\theta_-$ follows, since
$\theta_+=\pi_2\circ\pi_1^{-1}$ and $\theta_-=\pi_3\circ\pi_1^{-1}$.
\end{proof}

We next apply Lemma~\ref{lem:dertri1} to the derived compositions
$\partial\Comp$ and $\partial^2\Comp$. Recall from~\eqref{eq:alpha'}
and \eqref{eq:alpha''} the trialitarian triples
\[
  \End(\partial\Comp) = \bigl((\End(V_2),\sigma_{b_2}, \strf_{q_2}),\,
  (\End(V_3),\sigma_{b_3}, \strf_{q_3}),\,
  (\End(V_1), \sigma_{b_1}, \strf_{q_1}),\,
  C_0(\alpha')\bigr)
\]
and
\[
  \End(\partial^2\Comp) = \bigl((\End(V_3),\sigma_{b_3},
  \strf_{q_3}),\, 
  (\End(V_1),\sigma_{b_1}, \strf_{q_1}),\,
  (\End(V_2), \sigma_{b_2}, \strf_{q_2}),\,
  C_0(\alpha'')\bigr).
\]
Let
\[
  \theta'_+\colon \pgo(q_2)\to\pgo(q_3)
  \qquad\text{and}\qquad
  \theta'_-\colon \pgo(q_2)\to\pgo(q_1)
\]
be the Lie algebra isomorphisms induced by $C_+(\alpha')$ and
$C_-(\alpha')$ respectively, and
\[
  \theta''_+\colon\pgo(q_3)\to\pgo(q_1)
  \qquad\text{and}\qquad
  \theta''_-\colon\pgo(q_3)\to\pgo(q_2)
\]
those induced by $C_+(\alpha'')$ and $C_-(\alpha'')$. Let also
\[
  \pi'_1\colon\pgo(\partial\Comp)\to\pgo(q_2),\qquad
  \pi'_2\colon\pgo(\partial\Comp)\to\pgo(q_3),\qquad
  \pi'_3\colon\pgo(\partial\Comp)\to\pgo(q_1)
\]
and
\[
  \pi''_1\colon\pgo(\partial^2\Comp)\to\pgo(q_3),\qquad
  \pi''_2\colon\pgo(\partial^2\Comp)\to\pgo(q_1),\qquad
  \pi''_3\colon\pgo(\partial^2\Comp)\to\pgo(q_2)
\]
be the projections on the various components of $\pgo(\partial\Comp)$
and $\pgo(\partial^2\Comp)$. Lemma~\ref{lem:dertri1} yields
\begin{equation}
  \label{eq:dertri2}
  \theta'_+=\pi'_2\circ{\pi'_1}^{-1},\qquad
  \theta'_-=\pi'_3\circ{\pi'_1}^{-1},\qquad
  \theta''_+=\pi''_2\circ{\pi''_1}^{-1},\qquad
  \theta''_-=\pi''_3\circ{\pi''_1}^{-1}.
\end{equation}

\begin{prop}
  \label{prop:dertri}
  The following equations hold:
  \[
    \theta'_+=\theta_-\circ\theta_+^{-1},\qquad
    \theta'_-=\theta_+^{-1},\qquad
    \theta''_+=\theta_-^{-1},\qquad
    \theta''_-=\theta_+\circ\theta_-^{-1}.
  \]
\end{prop}

\begin{proof}
  The switch maps $\partial$ fit in the following
  commutative diagrams:
  \[
    \xymatrix{
      \pgo(\Comp) \ar[rr]^{\partial}
      \ar[dr]^{\pi_1}
      && \pgo(\partial\Comp)
      \ar[dl]_{\pi'_3}
      \ar[dddl]^{\partial}
    \\
    &\pgo(q_1)&\\
    &&\\
    &\pgo(\partial^2\Comp)\ar[uu]_{\pi''_2}\ar[uuul]^{\partial}&
  }
  \qquad
  \xymatrix{
      \pgo(\Comp) \ar[rr]^{\partial}
      \ar[dr]^{\pi_2}
      && \pgo(\partial\Comp)
      \ar[dl]_{\pi'_1}
      \ar[dddl]^{\partial}
    \\
    &\pgo(q_2)&\\
    &&\\
    &\pgo(\partial^2\Comp)\ar[uu]_{\pi''_3}\ar[uuul]^{\partial}&
  }
\]
and
\[
  \xymatrix{
      \pgo(\Comp) \ar[rr]^{\partial}
      \ar[dr]^{\pi_3}
      && \pgo(\partial\Comp)
      \ar[dl]_{\pi'_2}
      \ar[dddl]^{\partial}
    \\
    &\pgo(q_3)&\\
    &&\\
    &\pgo(\partial^2\Comp)\ar[uu]_{\pi''_1}\ar[uuul]^{\partial}&
  }
\]
Substituting $\pi'_1=\pi_2\circ\partial^2$,
$\pi'_2=\pi_3\circ\partial^2$, $\pi'_3=\pi_1\circ\partial^2$ and
$\pi''_1=\pi_3\circ\partial$, $\pi''_2=\pi_1\circ\partial$,
$\pi''_3=\pi_2\circ\partial$ in~\eqref{eq:dertri2} yields
\[
  \theta'_+=\pi_3\circ\pi_2^{-1},\qquad
  \theta'_-=\pi_1\circ\pi_2^{-1},\qquad
  \theta''_+=\pi_1\circ\pi_3^{-1},\qquad
  \theta''_-=\pi_2\circ\pi_3^{-1}.
\]
The proposition follows by Lemma~\ref{lem:dertri1}.
\end{proof}

The maps $\theta_\pm$, $\theta'_\pm$, $\theta''_\pm$ thus fit in the
following commutative diagram, in which all the maps are isomorphisms:
\begin{equation}
  \label{eq:diagtrial}
  \begin{aligned}
  \xymatrix{
    \pgo(q_1) \ar@/^/[rr]^{\theta_+}
    \ar[dddr]^{\theta_-}
    && \pgo(q_2)\ar[ll]^{\theta'_-}
    \ar@/^/[dddl]^{\theta'_+}
    \\
    &\pgo(\Comp)\ar[ul]_{\pi_1}
    \ar[ur]^{\pi_2}
      \ar[dd]_<<<<<<{\pi_3}&\\
    &&\\
    &\pgo(q_3) \ar@/^/[uuul]^{\theta''_+}
    \ar[uuur]^{\theta''_-}&
  }
  \end{aligned}
\end{equation}

\begin{proof}[Proof of Theorem~\ref{thm:trial5rev}]
  Corollary~\ref{corol:trial3} shows that it suffices to prove the
  claim after a Galois scalar extension that splits the trialitarian
  triple $\TT$. We may thus assume that $\TT=\End(\Comp)$ for some
  composition $\Comp$ of quadratic spaces of dimension~$8$. Then
  Proposition~\ref{prop:dertri} shows that
  $\theta_-\circ\theta_+^{-1}$ and $\theta_+^{-1}$ (resp.\
  $\theta_-^{-1}$ and $\theta_+\circ\theta_-^{-1}$) are the Lie algebra
  homomorphisms induced by the isomorphisms $C_+(\alpha')$ and
  $C_-(\alpha')$ (resp.\ $C_+(\alpha'')$ and $C_-(\alpha'')$) of the
  trialitarian triple $\End(\partial\Comp)$ (resp.\
  $\End(\partial^2\Comp)$), hence they are liftable by
  definition. Moreover, $\theta_-\circ\theta_+^{-1}$ and
  $\theta_+^{-1}$ (resp.\   $\theta_-^{-1}$,
  $\theta_+\circ\theta_-^{-1}$) are of opposite signs, hence the proof
  is complete.
\end{proof}

\begin{definition}
  \label{defn:dertri}
  Given any trialitarian triple
  $\TT=(\aqp_1,\,\aqp_2,\,\aqp_3,\,\varphi_0)$ with Lie algebra
  isomorphisms
  \[
    \theta_+\colon\pgo(\aqp_1)\to\pgo(\aqp_2)
    \qquad\text{and}\qquad
    \theta_-\colon\pgo(\aqp_1)\to\pgo(\aqp_3)
  \]
  induced by the components $\varphi_+\colon\Clqp(\aqp_1)\to\aqp_2$
  and $\varphi_-\colon\Clqp(\aqp_1)\to\aqp_3$ of $\varphi_0$, the pair
  of opposite Lie algebra isomorphisms $(\theta'_+,\,\theta'_-)=
  (\theta_-\circ\theta_+^{-1},\,\theta_+^{-1})$ (resp.\
  $(\theta''_+,\,\theta''_-)=
  (\theta_-^{-1},\,\theta_+\circ\theta_-^{-1})$)
  lifts by 
  Theorem~\ref{thm:trial5rev} to an isomorphism
  \[
    \varphi_0'\colon\Clqp(\aqp_2)\to \aqp_3\times\aqp_1
    \qquad\qquad\text{(resp.\ $\varphi_0''\colon \Clqp(\aqp_3)\to
      \aqp_1\times\aqp_2$)}
  \]
  that defines a trialitarian triple
  \[
    \partial\TT = (\aqp_2,\,\aqp_3,\,\aqp_1,\,\varphi_0')
    \qquad\qquad\text{(resp.\ $\partial^2\TT =
      (\aqp_3,\,\aqp_1,\,\aqp_2,\,\varphi''_0)$)}.
  \]
  The trialitarian triples $\partial\TT$ and $\partial^2\TT$ are
  called the \emph{derived trialitarian triples} of $\TT$.

Note that $\theta'_-\circ{\theta'_+}^{-1}=\theta''_+$ and
${\theta'_+}^{-1} = \theta_-''$, hence $\partial(\partial\TT) =
\partial^2\TT$. Similarly, $\partial^2(\partial\TT) = \TT =
\partial(\partial^2\TT)$ and $\partial^2(\partial^2\TT)=\partial\TT$.

From the proof of Theorem~\ref{thm:trial5rev}, it is clear that
for every composition $\Comp$ of quadratic
spaces of dimension~$8$,
\[
  \partial\End(\Comp) = \End(\partial\Comp) \qquad\text{and}\qquad
  \partial^2\End(\Comp) = \End(\partial^2\Comp).
\]
\end{definition}

We next establish the functoriality of the $\partial$ operation.

\begin{prop}
  \label{prop:funcTT}
  If $\gamma=(\gamma_1,\gamma_2,\gamma_3)\colon\TT\to\widetilde\TT$ is
  an isomorphism of trialitarian triples, then
  $\partial\gamma:=(\gamma_2,\gamma_3,\gamma_1)$ is an isomorphism of
  trialitarian triples $\partial\TT\to\partial\widetilde\TT$.
\end{prop}

\begin{proof}
  Let $(\theta_+,\theta_-)$ (resp.\
  $(\widetilde\theta_+,\widetilde\theta_-)$) be the pair of liftable
  homomorphisms attached to $\TT$ (resp.\ $\widetilde\TT$). The
  hypothesis that $\gamma$ is an isomorphism means that
  \[
    \gamma_2\circ\theta_+=\widetilde\theta_+\circ\gamma_1
    \qquad\text{and}\qquad
    \gamma_3\circ\theta_-=\widetilde\theta_-\circ\gamma_1.
  \]
  It then follows that
  \[
    \gamma_3\circ(\theta_-\circ\theta_+^{-1}) =
    \widetilde\theta_-\circ\gamma_1\circ\theta_+^{-1} =
    (\widetilde\theta_- \circ \widetilde\theta_+^{-1})\circ\gamma_2
    \quad\text{and}\quad
    \gamma_1\circ\theta_+^{-1} = \widetilde\theta_+^{-1}\circ\gamma_2.
  \]
  Since $(\theta_-\circ\theta_+^{-1},\,\theta_+^{-1})$ and
  $(\widetilde\theta_-\circ\widetilde\theta_+^{-1},\,
  \widetilde\theta_+^{-1})$ are the pairs of 
  liftable homomorphisms attached to $\partial\TT$ and
  $\partial\widetilde\TT$ respectively, it follows that
  $(\gamma_2,\gamma_3,\gamma_1)$ is an isomorphism
  $\partial\TT\to\partial\widetilde\TT$. 
\end{proof}

For the next corollary, observe that each trialitarian triple
$\TT=(\aqp_1,\,\aqp_2,\,\aqp_3,\,\varphi_0)$ yields a polarization of
$\aqp_1$ in the sense of Definition~\ref{defn:orientationaqp}: the
primitive idempotents in the center of $C(\aqp_1)$ are designated as
$z_{1+}$ and $z_{1-}$ according to the following convention:
\[
  \varphi_0(z_{1+})=(1,0) \qquad\text{and} \qquad
  \varphi_0(z_{1-})=(0,1),
\]
so that the two components of $\varphi_0$ are $\varphi_+\colon
\Clqp_+(\aqp_1)\xrightarrow{\sim}\aqp_2$ and $\varphi_-\colon
\Clqp_-(\aqp_1) \xrightarrow{\sim}\aqp_3$. Similarly, the maps
$\varphi_0'$ and $\varphi_0''$ of the derived trialitarian triples
$\partial\TT$ and $\partial^2\TT$ yield polarizations of $\aqp_2$ and
$\aqp_3$ so that
\[
  \varphi'_0(z_{2+})=(1,0),\quad \varphi'_0(z_{2-})=(0,1),\qquad
  \varphi''_0(z_{3+})=(1,0),\quad \varphi''_0(z_{3-})=(0,1),
\]
just as in the case of compositions of quadratic spaces: see
Remark~\ref{rem:polarcomp}.

\begin{corol}
  \label{corol:funcTT}
  Let $\TT=(\aqp_1,\,\aqp_2,\,\aqp_3,\,\varphi_0)$ and $\widetilde\TT=
  (\tilde\aqp_1,\,\tilde\aqp_2,\,\tilde\aqp_3,\,\tilde\varphi_0)$ be
  trialitarian triples. There are canonical one-to-one correspondences
  between the following sets:
  \begin{enumerate}
  \item[(i)]
    isomorphisms of trialitarian triples $\TT\to\widetilde\TT$;
  \item[(ii)]
    isomorphisms of algebras with quadratic pair
    $\aqp_1\to\tilde\aqp_1$ preserving the polarizations induced by
    $\TT$ and $\widetilde\TT$;
  \item[(iii)]
    isomorphisms of algebras with quadratic pair
    $\aqp_2\to\tilde\aqp_2$ preserving the polarizations induced by
    $\partial\TT$ and $\partial\widetilde\TT$;
  \item[(iv)]
    isomorphisms of algebras with quadratic pair
    $\aqp_3\to\tilde\aqp_3$ preserving the polarizations induced by
    $\partial^2\TT$ and $\partial^2\widetilde\TT$.
  \end{enumerate}
\end{corol}

\begin{proof}
  By definition, an isomorphism $\gamma\colon\TT\to\widetilde\TT$ is a
  triple $(\gamma_1,\gamma_2,\gamma_3)$ where each $\gamma_i$ is an
  isomorphism $\aqp_i\to\tilde\aqp_i$ and the following square
  commutes:
  \begin{equation}
    \label{eq:sq}
    \begin{aligned}
      \xymatrix{\Clqp(\aqp_1)\ar[r]^-{\varphi_0}\ar[d]_{C(\gamma_1)} &
        \aqp_2\times\aqp_3\ar[d]^{\gamma_2\times\gamma_3} \\
        \Clqp(\tilde\aqp_1)\ar[r]^-{\tilde\varphi_0}&
        \tilde\aqp_2\times \tilde\aqp_3}
    \end{aligned}
  \end{equation}
  Commutativity of this square implies that $\gamma_1$ preserves the
  polarizations of $\aqp_1$ and $\tilde\aqp_1$ induced by $\TT$ and
  $\widetilde\TT$.

  Conversely, if $\gamma_1\colon\aqp_1\to\tilde\aqp_1$ is an
  isomorphism of algebras with quadratic pair preserving
  polarizations, then there are isomorphisms
  $\gamma_2\colon\aqp_2\to\tilde\aqp_2$ and $\gamma_3\colon\aqp_3\to
  \tilde\aqp_3$ uniquely determined by the condition that the
  square~\eqref{eq:sq} commute. The triple
  $(\gamma_1,\gamma_2,\gamma_3)$ is then an isomorphism
  $\TT\to\widetilde\TT$. Thus, the sets described in~(i) and (ii) are
  in bijection under the map carrying
  $\gamma=(\gamma_1,\gamma_2,\gamma_3)$ to $\gamma_1$. Similarly, the
  set in~(iii) is in bijection with the set of isomorphisms
  $\partial\TT\to\partial\widetilde\TT$, hence, by
  Proposition~\ref{prop:funcTT}, with the set of isomorphisms
  $\TT\to\widetilde\TT$: to each isomorphism 
  $\gamma\colon\TT\to\widetilde\TT$ corresponds the second component
  $\gamma_2\colon\aqp_2\to\tilde\aqp_2$. Likewise, mapping $\gamma$ to
  $\gamma_3$ defines a one-to-one correspondence between~(i) and (iv).
\end{proof}

In the particular case where $\widetilde\TT=\TT$,
Corollary~\ref{corol:funcTT} yields isomorphisms between the group of
automorphisms of $\TT$ and the groups of polarization-preserving
automorphisms of $\aqp_1$, $\aqp_2$ and $\aqp_3$, which are
$\PGO^+(\aqp_1)$, $\PGO^+(\aqp_2)$ and $\PGO^+(\aqp_3)$. We discuss
this case in detail in the next subsection.

\subsection{Trialitarian isomorphisms}
\label{subsec:triso}

Throughout this subsection, we fix a trialitarian triple
$\TT=(\aqp_1,\,\aqp_2,\,\aqp_3,\,\varphi_0)$. We show how to attach to
$\TT$ canonical isomorphisms, which we call \emph{trialitarian
  isomorphisms}:
\[
  \BSpin(\aqp_1)\simeq\BSpin(\aqp_2)\simeq\BSpin(\aqp_3)
  \qquad\text{and}\qquad
  \BPGO^+(\aqp_1)\simeq\BPGO^+(\aqp_2)\simeq\BPGO^+(\aqp_3).
\]
\medbreak

Proposition~\ref{prop:funcTT} shows that the switch map $\partial$
yields an isomorphism
\[
  \partial\colon\BGO(\TT)\to\BGO(\partial\TT).
\]
This isomorphism maps $\BHomot(\TT)$ to $\BHomot(\partial\TT)$, hence
it induces a switch isomorphism $\partial\colon\BPGO(\TT)\to
\BPGO(\partial\TT)$. The following proposition shows that $\partial$
also maps $\BOrth(\TT)$ to $\BOrth(\partial\TT)$:

\begin{prop}
  \label{prop:shiftlambda}
  The following diagram is commutative:
  \[
    \xymatrix{
      \BGO(\TT)\ar[r]^{\partial}\ar[d]_{\lambda_\TT}&
      \BGO(\partial\TT)\ar[d]^{\lambda_{\partial\TT}}\\
      \BGm^3\ar[r]^{\partial}&\BGm^3
    }
  \]
\end{prop}

\begin{proof}
  When $\TT=\End(\Comp)$ for some composition $\Comp$ of quadratic
  spaces of dimension~$8$, then $\BGO(\TT) = \BGO(\Comp)$ and
  $\BGO(\partial\TT) = \BGO(\partial\Comp)$ by
  Proposition~\ref{prop:simEndComp}, and commutativity of the diagram
  is clear from~\eqref{eq:diaglambdalambda}. Commutativity for an
  arbitrary 
  trialitarian triple follows by scalar extension to a splitting
  field. 
\end{proof}

Recall from Corollary~\ref{corol:commdiagSpinPGO} the
diagram~\eqref{eq:commdiagSpinO} relating the groups
$\BSpin(\aqp_1)$ and $\BPGO^+(\aqp_1)$ to $\BOrth(\TT)$ and
$\BPGO(\TT)$. Substituting $\partial\TT$ for $\TT$ in that diagram, we
obtain another commutative diagram, which involves $\aqp_2$ and
$\partial\TT$ instead of $\aqp_1$ and $\TT$. We may connect this new
diagram to~\eqref{eq:commdiagSpinO} by means of the shift map to
obtain the following commutative diagram with exact rows, where all
the vertical maps are isomorphisms and $Z_2$ denotes the center of
$C(\aqp_2)$:

\[
  \xymatrix{1\ar[r]& R_{Z_1/F}(\Bmu_2)\ar[d]_{\psi_\TT}\ar[r]&
      \BSpin(\aqp_1)\ar[r]^-{\chi'}\ar[d]_{\psi_\TT}&
      \BPGO^+(\aqp_1)\ar[r]&1\\
      1\ar[r]&\BZO(\TT)\ar[r]\ar[d]_{\partial}&
      \BOrth(\TT)\ar[r]\ar[d]_{\partial}&
      \BPGO(\TT)\ar[r]\ar[d]_{\partial}\ar[u]^{\pi_\TT}&1\\
      1\ar[r]&\BZO(\partial\TT)\ar[r]&\BOrth(\partial\TT)\ar[r]&
      \BPGO(\partial\TT)\ar[r]\ar[d]_{\pi_{\partial\TT}}&1\\
      1\ar[r]& R_{Z_2/F}(\Bmu_2)\ar[u]^{\psi_{\partial\TT}}\ar[r]&
      \BSpin(\aqp_2)\ar[r]^-{\chi'}\ar[u]^{\psi_{\partial\TT}}&
      \BPGO^+(\aqp_2)\ar[r]&1
      }
\]
Define
\[
  \Sigma_\TT\colon
  \BSpin(\aqp_1)\xrightarrow{\sim}\BSpin(\aqp_2)
  \quad\text{and}\quad \Theta_\TT\colon
  \BPGO^+(\aqp_1)\xrightarrow{\sim}\BPGO^+(\aqp_2)
\]
by composing the
vertical isomorphisms: 
$
  \Sigma_{\TT} =
  \psi_{\partial\TT}^{-1}\circ\partial\circ\psi_\TT$ and 
  $\Theta_\TT = \pi_{\partial\TT}\circ\partial\circ\pi_\TT^{-1}$.
Forgetting the two central lines of the last diagram, we obtain a
commutative diagram with exact rows:
\begin{equation}
  \label{eq:commdiagSpin}
  \begin{aligned}
  \xymatrix{1\ar[r]& R_{Z_1/F}(\Bmu_2)\ar[d]_{\Sigma_\TT}\ar[r]&
      \BSpin(\aqp_1)\ar[r]^-{\chi'}\ar[d]_{\Sigma_\TT}&
      \BPGO^+(\aqp_1)\ar[r]\ar[d]_{\Theta_\TT}&1\\
      1\ar[r]& R_{Z_2/F}(\Bmu_2)\ar[r]&
      \BSpin(\aqp_2)\ar[r]^-{\chi'}&
      \BPGO^+(\aqp_2)\ar[r]&1
    }
    \end{aligned}
\end{equation}

Applying the construction above to $\partial\TT$ and
$\partial^2\TT$ instead of $\TT$, we obtain isomorphisms
$\Sigma_{\partial\TT}\colon \BSpin(\aqp_2) \xrightarrow{\sim}
  \BSpin(\aqp_3)$ and
  $\Sigma_{\partial^2\TT} \colon \BSpin(\aqp_3) \xrightarrow{\sim}
  \BSpin(\aqp_1)$
such that $\Sigma_{\partial^2\TT}\circ\Sigma_{\partial\TT}\circ
\Sigma_\TT=\Id$,
which make the following diagram with exact rows commute:
\[
  \xymatrix{
    &\vdots\ar[d]&\vdots\ar[d]&\vdots\ar[d]& \\
    1\ar[r]&R_{Z_1/F}(\Bmu_2)\ar[r]\ar[d]_{\Sigma_\TT}&
    \BSpin(\aqp_1)\ar[r]^-{\chi'}\ar[d]_{\Sigma_\TT}&
    \BPGO^+(\aqp_1)\ar[r]\ar[d]_{\Theta_\TT}&1\\
    1\ar[r]&R_{Z_2/F}(\Bmu_2)\ar[r]\ar[d]_{\Sigma_{\partial\TT}}&
    \BSpin(\aqp_2)\ar[r]^-{\chi'}\ar[d]_{\Sigma_{\partial\TT}}&
    \BPGO^+(\aqp_2)\ar[r]\ar[d]_{\Theta_{\partial\TT}}&1\\
    1\ar[r]&R_{Z_3/F}(\Bmu_2)\ar[r]\ar[d]_{\Sigma_{\partial^2\TT}}&
    \BSpin(\aqp_3)\ar[r]^-{\chi'}\ar[d]_{\Sigma_{\partial^2\TT}}&
    \BPGO^+(\aqp_3)\ar[r]\ar[d]_{\Theta_{\partial^2\Comp}}&1\\
    1\ar[r]&R_{Z_1/F}(\Bmu_2)\ar[r]\ar[d]&
    \BSpin(\aqp_1)\ar[r]^-{\chi'}\ar[d]&
    \BPGO^+(\aqp_1)\ar[r]\ar[d]&1\\
    &\vdots&\vdots&\vdots& }
\]

Letting $\pi_i\colon\BPGO(\TT)\to\BPGO^+(\aqp_i)$ denote the
projection on the $i$-th component, we have
\[
  \pi_1=\pi_\TT,\qquad
  \pi_2= \pi_{\partial\TT}\circ\partial,\qquad
  \pi_3=\pi_{\partial^2\TT}\circ\partial^2,
\]
hence the following diagram, in which all the maps are isomorphisms,
is commutative: 
\begin{equation}
  \label{eq:diagsimcompdim8}
  \begin{aligned}
    \xymatrix{\BPGO^+(\aqp_1)\ar[rr]^{\Theta_\TT}&
      &\BPGO^+(\aqp_2)\ar[dddl]^{\Theta_{\partial\TT}}\\
&\BPGO(\TT)\ar[ul]_{\pi_1} \ar[ur]^{\pi_2}
\ar[dd]_{\pi_3}& \\
&&\\
&\BPGO^+(\aqp_3)\ar[uuul]^{\Theta_{\partial^2\TT}}&
}
\end{aligned}
\end{equation}

Similarly, defining $\psi_i\colon\BSpin(\aqp_i)\to\BOrth(\TT)$ for
$i=1$, $2$, $3$ by
\[
  \psi_1=\psi_\TT,\qquad
  \psi_2=\partial^{-1}\circ\psi_{\partial\TT}, \qquad
  \psi_3=\partial^{-2}\circ\psi_{\partial^2\TT},
\]
we obtain the following commutative diagram similar
to~\eqref{eq:diagsimcompdim8}, where all the maps are isomorphisms: 
\[
  \xymatrix{\BSpin(\aqp_1)\ar[rr]^{\Sigma_\TT}\ar[dr]^{\psi_1}
    &&\BSpin(\aqp_2)\ar[dddl]^{\Sigma_{\partial\TT}}\ar[dl]_{\psi_2}\\
&\BOrth(\TT)& \\
&&\\
&\BSpin(\aqp_3)\ar[uuul]^{\Sigma_{\partial^2\TT}}\ar[uu]^{\psi_3}&
}
\]

Restricting to the central subgroups, we also obtain a commutative
diagram of isomorphisms:
\[
  \xymatrix{R_{Z_1/F}(\Bmu_2)\ar[rr]^{\Sigma_\TT}\ar[dr]^{\psi_1}
    &&R_{Z_2/F}(\Bmu_2)\ar[dddl]^{\Sigma_{\partial\TT}}\ar[dl]_{\psi_2}\\
&\BZO(\TT)& \\
&&\\
&R_{Z_3/F}(\Bmu_2)\ar[uuul]^{\Sigma_{\partial^2\TT}}\ar[uu]^{\psi_3}&
}
\]

The action of the trialitarian isomorphism $\Sigma_\TT$ on
$R_{Z_1/F}(\Bmu_2)$ is easy to determine from the definition of
$\psi_\TT$:

\begin{prop}
  \label{prop:ZSpin}
For $i=1$, $2$, $3$, let $z_{i+}$ and $z_{i-}$ denote the primitive
idempotents of $Z_i$ (according to the polarization). Then
for every commutative $F$-algebra $R$ and $a_+$, $a_-\in R$ such that
$a_+^2=a_-^2=1$,
\begin{align}
  \label{eq:ZSpin1}
  \Sigma_\TT(a_+z_{1+}+a_-z_{1-}) & = a_-z_{2+}+a_+a_-z_{2-},\\
  \label{eq:ZSpin2}
  \Sigma_{\partial\TT}(a_+z_{2+}+a_-z_{2-}) & =
                                             a_-z_{3+}+a_+a_-z_{3-},\\
\label{eq:ZSpin3}
  \Sigma_{\partial^2\TT}(a_+z_{3+}+a_-z_{3-}) & = a_-z_{1+}+a_+a_-z_{1-}.
\end{align}
\end{prop}

\begin{proof}
  From~\eqref{eq:psiTTonZ} it follows that
  $
    \psi_\TT(a_+z_{1+}+a_-z_{1-}) = (a_+a_-,\,a_+,\,a_-)
    $, hence
    \[
      \partial\circ\psi_\TT(a_+z_{1+}+a_-z_{1-}) =
      (a_+,\,a_-,\,a_+a_-) =
      \psi_{\partial\TT}(a_-z_{2+}+a_+a_-z_{2-}).
    \]
    Equation~\eqref{eq:ZSpin1} follows, since $\Sigma_\TT=
    \psi_{\partial\TT}^{-1}\circ\partial\circ\psi_\TT$. Equations~\eqref{eq:ZSpin2}
    and \eqref{eq:ZSpin3} are proved similarly.
  \end{proof}

Proposition~\ref{prop:ZSpin} shows that $\Sigma_\TT$ does \emph{not}
map the subgroup $\Bmu_2$ of $R_{Z_1/F}(\Bmu_2)$ to the subgroup
$\Bmu_2$ of $R_{Z_2/F}(\Bmu_2)$; this is a characteristic feature of
trialitarian isomorphisms.

\subsection{Compositions of $8$-dimensional quadratic spaces}
\label{subsec:comp8}

Let $\Comp=\bigl((V_1,q_1),\,(V_2,q_2),\,
(V_3,q_3),\,*_3\bigr)$ and $\widetilde\Comp=\bigl((\tilde V_1, \tilde
q_1),\, (\tilde V_2,\tilde q_2),\, (\tilde V_3, \tilde q_3),\,
\tilde*_3\,)$ denote compositions of quadratic spaces of dimension~$8$
over $F$ throughout this subsection. Recall from
Remark~\ref{rem:polarcomp} that $\Comp$ and 
$\widetilde\Comp$ induce polarizations of $(V_1,q_1)$ and $(\tilde
V_1, \tilde q_1)$ respectively. Our goal is to establish criteria for
the existence of a similitude or an isomorphism between $\Comp$ and
$\widetilde\Comp$.

\begin{thm}
  \label{thm:simcriterion}
  For every similitude $g_1\colon(V_1,q_1)\to(\tilde V_1, \tilde q_1)$
  preserving the polarizations induced by $\Comp$ and
  $\widetilde\Comp$, there exist similitudes $g_2\colon(V_2,q_2) \to
  (\tilde V_2, \tilde q_2)$ and $g_3\colon(V_3,q_3) \to (\tilde V_3,
  \tilde q_3)$ such that the triple $(g_1,g_2,g_3)$ is a similitude
  $\Comp\to \widetilde\Comp$. The similitudes $g_2$ and $g_3$ are
  uniquely determined up to a scalar factor.
\end{thm}

\begin{proof}
  The similitude $g_1$ defines an isomorphism of algebras with
  quadratic pair
  \[
    \Int(g_1)\colon (\End V_1,\sigma_{b_1}, \strf_{q_1}) \to (\End
    \tilde V_1, \sigma_{\tilde b_1}, \strf_{\tilde q_1}),
  \]
  see Proposition~\ref{prop:simiso}. Since $g_1$ preserves the
  polarizations of $(V_1,q_1)$ and $(\tilde V_1,\tilde q_1)$, it
  follows that $\Int(g_1)$ preserves the polarizations of $(\End
  V_1,\sigma_{b_1}, \strf_{q_1})$ and $(\End
    \tilde V_1, \sigma_{\tilde b_1}, \strf_{\tilde q_1})$ induced by
    the trialitarian triples $\End(\Comp)$ and $\End(\widetilde\Comp)$
    respectively, hence Corollary~\ref{corol:funcTT} yields uniquely
    determined isomorphisms
    \[
      \gamma_2\colon (\End V_2,\sigma_{b_2}, \strf_{q_2}) \to (\End
      \tilde V_2, \sigma_{\tilde b_2}, \strf_{\tilde q_2})
      \quad\text{and}\quad
      \gamma_3\colon (\End V_3,\sigma_{b_3}, \strf_{q_3}) \to (\End
      \tilde V_3, \sigma_{\tilde b_3}, \strf_{\tilde q_3})
    \]
    such that $(\Int(g_1),\gamma_2,\gamma_3)$ is an isomorphism
    $\End(\Comp)\to\End(\widetilde\Comp)$. Proposition~\ref{prop:simiso}
    shows that there exist similitudes $g_2\colon(V_2,q_2)\to(\tilde
    V_2,\tilde q_2)$ and $g_3\colon(V_3,q_3) \to (\tilde V_3, \tilde
    q_3)$, uniquely determined up to a scalar factor, such that
    $\gamma_2=\Int(g_2)$ and $\gamma_3=\Int(g_3)$. It follows from
    Proposition~\ref{prop:isoTT} that $(g_1,g_2,g_3)$ is a similitude
    $\Comp\to\widetilde\Comp$. 
\end{proof}

\begin{corol}
  \label{corol:simcriterion}
  Let $n_\Comp$ and $n_{\tilde\Comp}$ denote the $3$-fold Pfister
  forms associated to $\Comp$ and $\widetilde\Comp$ by
  Proposition~\ref{prop:compdim8}. The following conditions are
  equivalent:
  \begin{enumerate}
  \item[(i)]
    $\Comp$ is similar to $\widetilde\Comp$;
  \item[(ii)]
    $n_\Comp\simeq n_{\tilde\Comp}$.
  \end{enumerate}
\end{corol}

\begin{proof}
  Recall that $q_1\simeq\langle\lambda_1\rangle n_\Comp$ and $\tilde
  q_1\simeq \langle\tilde\lambda_1\rangle n_{\tilde\Comp}$ for some
  $\lambda_1$, $\tilde\lambda_1\in F^\times$. If $\Comp$ is similar to
  $\widetilde\Comp$, then $q_1$ is similar to $\tilde q_1$, hence
  $n_\Comp\simeq n_{\tilde\Comp}$ because similar Pfister forms are
  isometric. Conversely, if $n_\Comp\simeq n_{\tilde\Comp}$, then
  there is a similitude $g_1\colon(V_1,q_1)\to (\tilde V_1, \tilde
  q_1)$. Composing $g_1$ with an improper isometry if necessary, we
  may assume $g_1$ preserves the polarizations of $(V_1,q_1)$ and
  $(\tilde V_1, \tilde q_1)$. Then Theorem~\ref{thm:simcriterion}
  yields a similitude $\Comp\to\widetilde\Comp$. 
\end{proof}

In the particular case where $\widetilde\Comp=\Comp$,
Theorem~\ref{thm:simcriterion} is a direct generalization of the
Principle of Triality discussed by
Springer--Veldkamp~\cite[Th.~3.2.1]{SpV}, as follows:

\begin{corol}
  \label{corol:PoT}
  For every proper similitude $g_1\in\GO^+(q_1)$, there exist
  similitudes $g_2\in\GO(q_2)$ and $g_3\in\GO(q_3)$ such that
  \[
    g_1(x_2*_1x_3) = g_2(x_2)*_1g_3(x_3) \qquad\text{for all $x_2\in
      V_2$ and $x_3\in V_3$.}
  \]
\end{corol}

\begin{proof}
  Theorem~\ref{thm:simcriterion} yields similitudes $g'_2\in\GO(q_2)$
  and $g'_3\in\GO(q_3)$ such that
  $(g_1,g'_2,g'_3)\in\BGO(\Comp)(F)$. Letting
  $\lambda_\Comp(g_1,g'_2,g'_3) = (\lambda_1,\lambda_2,\lambda_3)$, we
  have by Proposition~\ref{prop:simdef2}
  \[
    \lambda_1\,g_1(x_2*_1x_3) = g'_2(x_2)*_1g'_3(x_3) \qquad\text{for
      all $x_2\in V_2$ and $x_3\in V_3$.}
  \]
  Then $g_2=\lambda_1^{-1}g'_2$ and $g_3=g'_3$ satisfy the requirement.
\end{proof}

In the special case where $*_1$ is the multiplication in an octonion
algebra, Corollary~\ref{corol:PoT} is (the main part
of)~\cite[Th.~3.2.1]{SpV}.

Corollary~\ref{corol:PoT} also has a ``local'' version:

\begin{corol}
  \label{corol:localPoT}
  For every $g_1\in\go(q_1)$, there exist $g_2\in\go(q_2)$ and
  $g_3\in\go(q_3)$ such that
  \[
    g_1(x_2*_1x_3) = g_2(x_2)*_1x_3 + x_2*_1g_3(x_3) +\dot\mu(g_1)\,
    x_2*_1x_3\qquad\text{for all $x_2\in V_2$ and $x_3\in V_3$.}
  \]
\end{corol}

\begin{proof}
  Lemma~\ref{lem:dertri1} shows that projection on the first component
  $\pi_1\colon\pgo(\Comp)\to\pgo(q_1)$ is bijective, hence there exist
  $g'_2\in\go(q_2)$ and $g'_3\in\go(q_3)$ such that $(g_1+F, g'_2+F,
  g'_3+F)$ lies in $\pgo(\Comp)$, which means that there exists
  $\lambda_3\in F$ such that
  \[
    g'_3(x_1*_3x_2) = g_1(x_1)*_3x_2 + x_1*_3g'_2(x_2) -\lambda_3\,
    x_1*_3x_2 \qquad\text{for all $x_1\in V_1$ and $x_2\in V_2$.}
  \]
  By Proposition~\ref{prop:Liesimdef2}, there also exists
  $\lambda_1\in F$ such that
  \[
    g_1(x_2*_1x_3) = g'_2(x_2)*_1x_3 + x_2*_1g'_3(x_3) - \lambda_1\,
    x_2*_1x_3 \qquad\text{for all $x_2\in V_2$ and $x_3\in V_3$.}
  \]
  Then $g_2=g'_2-\lambda_1$ and $g_3=g'_3+\dot\mu(g_1)$ satisfy the
  required condition.
\end{proof}

Specializing $*_1$ to be the multiplication in an octonion algebra
(resp.\ the multiplication in a symmetric composition algebra of
dimension~$8$) yields Elduque's Principle of Local Triality
\cite[Th.~3.2]{Eld} (resp.\ \cite[Th.~5.2]{Eld}).
\medbreak

By contrast with similitudes in Theorem~\ref{thm:simcriterion},
isometries $(V_1,q_1)\to (\tilde V_1, \tilde q_1)$ do not necessarily
extend to isomorphisms $\Comp\to\widetilde\Comp$ since $(V_i,q_i)$ may
not be isometric to $(\tilde V_i,\tilde q_i)$ for $i=2$,
$3$. Nevertheless, we will obtain in
Theorem~\ref{thm:isocriterion} below an isomorphism criterion for
compositions of quadratic spaces by using the following construction
of similitudes.
\medbreak

For $\Comp$ as above, define a new composition of quadratic spaces
$\Comp'$ as follows:
\[
  \Comp'=\bigl((V_2,q_2),\,(V_1,q_1),\,(V_3,q_3),\,*'_3\bigr)
\]
where
\[
  x_2*'_3x_1=x_1*_3x_2\qquad\text{for $x_2\in V_2$ and $x_1\in V_1$.}
\]
To every anisotropic vector $u_3\in V_3$, we associate the map
\begin{equation}
  \label{eq:rhodef}
  \rho_{u_3}(x_3)=u_3q_3(u_3)^{-1}b_3(u_3,x_3) - x_3 \qquad\text{for
    $x_3\in V_3$.}
\end{equation}
Computation shows that $\rho_{u_3}$ is an isometry fixing $u_3$.

\begin{prop}
  \label{prop:constsim}
  For every anisotropic vector $u_3\in V_3$, the triples
  $(\ell_{u_3},r_{u_3},\rho_{u_3})\colon\Comp\to\Comp'$ and
  $(r_{u_3},\ell_{u_3},\rho_{u_3})\colon \Comp'\to\Comp$ are
  similitudes with composition multiplier $\bigl(1,1,q_3(u_3)\bigr)$.
\end{prop}

\begin{proof}
  Since $\mu(r_{u_3})=\mu(\ell_{u_3})=q_3(u_3)$, to prove
  $(\ell_{u_3},r_{u_3},\rho_{u_3})$ is a similitude with composition
  multiplier 
  $\bigl(1,1,q_3(u_3)\bigr)$ it suffices to show
  \[
    q_3(u_3)\rho_{u_3}(x_1*_3x_2)= \ell_{u_3}(x_1)*'_3r_{u_3}(x_2)
    \qquad\text{for all $x_1\in V_1$, $x_2\in V_2$.}
  \]
  Likewise, to prove $(r_{u_3},\ell_{u_3},\rho_{u_3})$ is a similitude
  with composition multiplier $\bigl(1,1,q_3(u_3)\bigr)$
  it suffices to show
  \[
    q_3(u_3)\rho_{u_3}(x_2*'_3x_1)=r_{u_3}(x_2)*_3\ell_{u_3}(x_1)
    \qquad\text{for all $x_2\in V_2$, $x_1\in V_1$.}
  \]
  Each of these equations amounts to
  \[
    u_3b_3(u_3,x_1*_3x_2) - (x_1*_3x_2)q_3(u_3) =
    (x_2*_1u_3)*_3(u_3*_2x_1).
  \]
  By~\eqref{eq:comp7linbis}, we may rewrite the right side as
  \[
    (x_2*_1u_3)*_3(u_3*_2x_1)= u_3b_1(x_2*_1u_3,x_1) -
    x_1*_3\bigl(u_3*_2(x_2*_1u_3)\bigr).
  \]
  Since $b_1(x_2*_1u_3,x_1)=b_3(u_3,x_1*_3x_2)$ by \eqref{eq:comp45},
  and $u_3*_2(x_2*_1u_3) = x_2q_3(u_3)$ by \eqref{eq:compp6}, the
  proposition follows.
\end{proof}

Proposition~\ref{prop:constsim} allows us to describe the group
\[
  G(\Comp) = \lambda_\Comp\bigl(\BGO(\Comp)(F)\bigr)\subset
  F^\times\times F^\times\times F^\times
\]
of composition multipliers of auto-similitudes of $\Comp$. In the next
corollary, we write $G(n_\Comp)$ for the group of multipliers of
similitudes of the Pfister form associated to $\Comp$, which is also
the set of represented values of this form because Pfister forms are
round (see~\cite[Cor.~9.9]{EKM}).

\begin{corol}
  \label{corol:GComp}
  $G(\Comp) = \{(\lambda_1,\lambda_2,\lambda_3)\in
    F^\times\times F^\times\times F^\times\mid
    \lambda_1\equiv\lambda_2\equiv\lambda_3\bmod G(n_\Comp)\}$. 
  \end{corol}

\begin{proof}
  If
  $(\lambda_1,\lambda_2,\lambda_3)=\lambda_\Comp(g_1,g_2,g_3,\lambda_3)$
  for some $(g_1,g_2,g_3,\lambda_3)\in\BGO(\Comp)(F)$, then by
  definition of $\lambda_\Comp$ (see~\eqref{eq:lambdaCompdef})
  \[
    \lambda_1=\mu(g_2)\lambda_3^{-1} \qquad\text{and}\qquad
    \lambda_2=\mu(g_1)\lambda_3^{-1}.
  \]
  Since $q_1$ and $q_2$ are multiples of $n_\Comp$, multipliers of
  similitudes of $q_1$ and of $q_2$ lie in $G(n_\Comp)$, hence
  $\lambda_1\lambda_3\in G(n_\Comp)$ and $\lambda_2\lambda_3\in
  G(n_\Comp)$. Therefore,
  $\lambda_1\equiv\lambda_2\equiv\lambda_3\bmod G(n_\Comp)$.
  \medbreak

  For the converse, we first establish:
  \smallbreak\par\noindent
  \textit{Claim: $(1,1,\nu)\in G(\Comp)$ for every $\nu\in
  G(n_\Comp)$.}
  \par\noindent
  To see this, pick any anisotropic vector $u_3\in V_3$, and let
  $v_3\in V_3$ be the image of $u_3q_3(u_3)^{-1}$ under any similitude
  of $(V_3,q_3)$ with multiplier $\nu$, so that $q_3(v_3)=\nu
  q_3(u_3)^{-1}$. By Proposition~\ref{prop:constsim}, the composition
  of maps $(r_{v_3}, \ell_{v_3},\rho_{v_3})\circ (\ell_{u_3}, r_{u_3},
  \rho_{u_3})$ is an auto-similitude of $\Comp$ with multiplier
  $(1,1,q_3(v_3))(1,1,q_3(u_3))=(1,1,\nu)$. This proves the claim.

  Since for the derived composition $\partial\Comp$ we have
  $n_{\partial\Comp}\simeq n_\Comp$, it follows that $(1,1,\nu)\in
  G(\partial\Comp)$ for 
  every $\nu\in G(n_\Comp)$, hence $(\nu,1,1)\in G(\Comp)$ for every
  $\nu\in G(n_\Comp)$.
  
  Now, suppose $(\lambda_1,\lambda_2,\lambda_3)\in
  F^\times\times F^\times\times F^\times$ is such that
  $\lambda_1\lambda_2^{-1}$, $\lambda_2^{-1}\lambda_3\in
  G(n_\Comp)$. The previous observations show
  \[
    (\lambda_1\lambda_2^{-1},1,1),\; (1,1,\lambda_2^{-1}\lambda_3)\in
    G(\Comp).
  \]
  Moreover, $(\lambda_2\Id_{V_1}, \lambda_2\Id_{V_2},
  \lambda_2\Id_{V_3}, \lambda_2)\in\BGO(\Comp)(F)$ is a similitude
  with composition 
  multiplier $(\lambda_2,\lambda_2,\lambda_2)$. Therefore, the group
  $G(\Comp)$ also contains the product
  \[
    (\lambda_1\lambda_2^{-1},1,1)\cdot(1,1,\lambda_2^{-1}\lambda_3)
    \cdot (\lambda_2,\lambda_2,\lambda_2) = (\lambda_1, \lambda_2,
    \lambda_3).
    \qedhere
  \]
\end{proof}

\begin{remark}
  Proposition~\ref{prop:constsim} and Corollary~\ref{corol:GComp} also
  hold, with the same proof,
  for compositions of quadratic spaces of dimension~$2$ or $4$.
\end{remark}

\begin{thm}
  \label{thm:isocriterion}
  The compositions $\Comp$ and $\widetilde\Comp$ are isomorphic if and
  only if 
  $(V_i,q_i)$ and $(\tilde V_i,\tilde q_i)$ are isometric for $i=1$,
  $2$ and $3$.
\end{thm}

\begin{proof}
  If $g=(g_1,g_2,g_3)\colon\Comp\to\widetilde\Comp$ is an isomorphism,
  then from the relations between the multipliers of $g_1$, $g_2$,
  $g_3$ and the composition multiplier $\lambda(g)$
  in~\eqref{eq:lambdamu} it follows that $g_1$, $g_2$ and $g_3$ are
  isometries, hence $(V_i,q_i)\simeq(\tilde V_i,\tilde q_i)$ for all
  $i$.

  For the converse, assume $(V_i,q_i)$ is isometric to $(\tilde V_i,
  \tilde q_i)$ for $i=1$, $2$, $3$, and pick an isometry
  $g_1\colon(V_1,q_1) \to (\tilde V_1, \tilde q_1)$. Composing it with
  an improper isometry if needed, we may assume $g_1$ preserves the
  polarizations induced by $\Comp$ and
  $\widetilde\Comp$. Theorem~\ref{thm:simcriterion} then yields a
  similitude $g=(g_1,g_2,g_3)\colon\Comp\to\widetilde\Comp$. Let
  $\lambda(g)=(\lambda_1,\lambda_2,\lambda_3)$. From the
  relations~\eqref{eq:lambdamu} between $\lambda(g)$ and the
  multipliers of $g_1$, $g_2$, $g_3$ it follows that
  $\mu(g_1)=\lambda_2\lambda_3$, hence $\lambda_2\lambda_3=1$ since
  $g_1$ is an isometry. The triple $(g_1,\lambda_2g_2,g_3)$ also is a
  similitude $\Comp\to\widetilde\Comp$, and
  \[
    \lambda(g_1,\lambda_2g_2,g_3) =
    \lambda(g)\cdot(\lambda_2,\lambda_2^{-1}, \lambda_2) =
    (\lambda_1\lambda_2,1,1) = (\mu(g_3),1,1).
  \]
  Since $(\tilde V_3,\tilde q_3)\simeq(V_3,q_3)$, the multiplier
  $\mu(g_3)$ is the multiplier of a similitude of $q_3$, hence also of
  $n_\Comp$. Corollary~\ref{corol:GComp} then shows that there exists
  an auto-similitude $(g'_1,g'_2,g'_3)$ of $\Comp$ such that
  $\lambda(g'_1,g'_2,g'_3) = (\mu(g_3)^{-1},1,1)$. Then $(g_1\circ
  g'_1, \lambda_2g_2\circ g'_2, g_3\circ g'_3)$ is a similitude
  $\Comp\to\widetilde\Comp$ with composition multiplier $(1,1,1)$,
  i.e., it is an isomorphism.
\end{proof}

Corollary~\ref{corol:simcriterion} and
Theorem~\ref{thm:isocriterion} can be given a cohomological
interpretation: over a separable closure of $F$,
Corollary~\ref{corol:2.30} (or Theorem~\ref{thm:isocriterion})
shows that all the compositions of quadratic spaces of dimension~$8$
are isomorphic. Therefore, if $\Comp_0$ is a composition of hyperbolic
quadratic spaces of dimension~$8$ over $F$ (such as the composition
associated to the split para-octonion algebra), standard arguments of
nonabelian Galois cohomology (see for instance~\cite[\S29]{BoI}) yield
canonical bijections
\[
  H^1\bigl(F,\BOrth(\Comp_0)\bigr) \quad\longleftrightarrow\quad
  \fbox{\parbox{60mm}{isomorphism
    classes of compositions of quadratic spaces of dimension~$n$ over
    $F$}}
\]
and
\[
  H^1\bigl(F,\BGO(\Comp_0)\bigr) \quad\longleftrightarrow\quad
  \fbox{\parbox{60mm}{similarity
    classes of compositions of quadratic spaces of dimension~$n$ over
    $F$}}
\]
because $\BOrth(\Comp_0)$ (resp.\ $\BGO(\Comp_0)$) is the group of
automorphisms (resp.\ auto-similitudes) of $\Comp_0$. Since by
Proposition~\ref{prop:simEndComp} the group
$\BPGO(\Comp_0)$ is the automorphism group of the trialitarian triple
$\End(\Comp_0)$, there is an additional canonical bijection
\[
  H^1\bigl(F,\BPGO(\Comp_0)\bigr) \quad\longleftrightarrow\quad
  \fbox{isomorphism classes of trialitarian triples over $F$}
\]

Now, Corollary~\ref{corol:simcriterion} yields a bijection between
$H^1\bigl(F,\BGO(\Comp_0)\bigr)$ and the set of isometry classes of
$3$-fold quadratic Pfister forms. Similarly,
Theorem~\ref{thm:isocriterion} yields a bijection between
$H^1\bigl(F,\BOrth(\Comp_0)\bigr)$ and the set of triples of quadratic
forms $(q_1,q_2,q_3)$ up to isometry, subject to the condition that
there exists a $3$-fold quadratic Pfister form $n$ such that $q_1$,
$q_2$, $q_3$ are similar to $n$ and the orthogonal sum $n\perp
q_1\perp q_2\perp q_3$ is a $5$-fold quadratic Pfister form. This can
also be viewed as a description of $H^1(F,\BSpin_8)$ for $\BSpin_8$
the spin group of $8$-dimensional hyperbolic quadratic forms, because
Theorem~\ref{thm:SpinOTT} yields a canonical isomorphism
$\BSpin_8\simeq \BOrth(\Comp_0)$. We may use this description to give
an interpretation of the $\bmod\:2$ cohomological invariants of $\BSpin_8$
determined by Garibaldi in~\cite[\S18.1]{Gar} under the hypothesis
that $\charac F\neq2$, as follows: for $n=3$, $4$, $5$, let $e_n$
denote the Elman--Lam cohomological invariant of $n$-fold Pfister
forms, defined by
\[
  e_n\bigl(\langle1, -a_1\rangle\cdot\ldots\cdot\langle1, -a_n\rangle\bigr)
  = (a_1)\cup\ldots\cup(a_n)\in H^n(F,\Bmu_2),
\]
where $(a_i)\in H^1(F,\Bmu_2)$ is the cohomology class corresponding
to the square class of $a_i\in F^\times$ by Kummer theory,
see~\cite[\S16]{EKM}. For every triple $(q_1,q_2,q_3)$ as above, the
cohomology classes
\[
  e_3(n),\quad e_4(n\perp q_1),\quad e_4(n\perp
  q_2), \quad e_4(n\perp q_3),\quad e_5(n\perp q_1\perp
  q_2\perp q_3)
\]
define cohomological invariants, which distinguish these triples up to
isometry. According to~\cite[\S18.1]{Gar}, these invariants
generate the $H^*(F,\mathbb{Z}/2\mathbb{Z})$-module of mod~$2$
invariants of $\BSpin_8$. Note that these invariants are not
independent: since
\[
  (n\perp q_1) \perp (n\perp q_2) \perp (n\perp q_3) = 2n \perp
  (n\perp q_1\perp q_2\perp q_3)
\]
and $n\perp q_1\perp q_2\perp q_3$ is a $5$-fold Pfister form, it
follows that
\[
  e_4(n\perp q_1)+e_4(n\perp q_2) + e_4(n\perp q_3) = e_4(2n)
  = (-1)\cup e_3(n).
\]

\subsection{The structure group of $8$-dimensional composition algebras}
\label{subsec:strgrp8}

Let $\Calg=(A,q,\diamond)$ be a composition algebra of
dimension~$8$. Recall from 
Defintion~\ref{defn:isotopy} the structure group $\BStr(A,\diamond)$,
which is the group of autotopies of
$(A,\diamond)$. Corollary~\ref{corol:isot} identifies
$\BStr(A,\diamond)$ with a subgroup of $\BGO\bigl(\Comp(\Calg)\bigr)$,
for $\Comp(\Calg)$ the composition of quadratic spaces associated to
$\Calg$ as in~\eqref{eq:CompCalg}.

In the trialitarian triple $\TT=\End\bigl(\Comp(\Calg)\bigr)$ we have
$\aqp_1=\aqp_2=\aqp_3=(\End A,\sigma_b,\strf_q)$. Mimicking the
construction in~\S\ref{subsec:simTT}, we obtain a morphism
$\psi_{\partial^2\TT}\colon\BOmega(\aqp_3)\to\BGO(\partial^2\Comp)$ as
in~\eqref{eq:defnpsiTT}. We use it to define a morphism of algebraic
groups
\[
  \psi_\Calg\colon\BOmega(q)\to\BGO\bigl(\Comp(\Calg)\bigr)
\]
by specializing to the case where $\TT=\End\bigl(\Comp(\Calg)\bigr)$
the map $\partial\circ\psi_{\partial^2\TT}\colon \BOmega(\aqp_3)\to
\BGO(\TT)$, where $\partial$ is the shift map. Thus, for any
commutative $F$-algebra $R$ and $\xi\in\BOmega(q)(R)$,
\[
  \psi_\Calg(\xi) = \bigl(C_+(\alpha'')(\xi),\, C_-(\alpha'')(\xi),\,
  \chi_0(\xi)\bigr)
\]
(viewing $\BGO\bigl(\Comp(\Calg)\bigr)$ as a subgroup of
$\BGO(q)\times\BGO(q)\times\BGO(q)$, as in the proof of
Proposition~\ref{prop:simEndComp}), where $C_\pm(\alpha'')$ are the
canonical Clifford maps attached to $\partial^2\Comp(\Calg)$,
see~\eqref{eq:alpha''}.

\begin{thm}
  \label{thm:strgrp}
  The map $\psi_\Calg$ is an isomorphism
  $\BOmega(q)\xrightarrow{\sim}\BStr(A,\diamond)$. 
\end{thm}

\begin{proof}
  The map $\psi_\Calg$ is injective because $C(\alpha'')$ is an
  isomorphism $C(A,q)\to\End(A\oplus A)$, and the computation of
  $\lambda_\TT\circ\psi_\TT$ in~\eqref{eq:lambdaTTpsi} together with
  Corollary~\ref{corol:isot} shows that $\psi_\Calg$ maps $\BOmega(q)$
  to $\BStr(A,\diamond)$.

  To complete the proof, we show that for any commutative
  $F$-algebra $R$ the group $\BStr(A,\diamond)(R)$ is the image of
  $\BOmega(q)(R)$ under $\psi_\Calg$. Let $(g_1,g_2,g_3)$ be an
  autotopy of $(A,\diamond)_R$, which means that
  \[
    g_3(x_1\diamond x_2) = g_1(x_1)\diamond g_2(x_2) \qquad\text{for
      all $x_1$, $x_2\in A_R$.}
  \]
  By Proposition~\ref{prop:simdef2} it follows that for all $x_1$,
  $x_2$, $x_3\in A_R$
  \[
    \mu(g_2)\,g_1(x_2\diamond_1x_3) = g_2(x_2)\diamond_1 g_3(x_3)
    \quad\text{and}\quad
    \mu(g_1)\,g_2(x_3\diamond_2x_1) = g_3(x_3)\diamond_2g_1(x_1).
  \]
  Equivalently,
  \[
    \mu(g_2)\,g_1\circ r_{x_3} = r_{g_3(x_3)}\circ g_2
    \quad\text{and}\quad
    \mu(g_1)\,g_2\circ\ell_{x_3}=\ell_{g_3(x_3)}\circ g_1,
  \]
  which can be reformulated as an equation in $\End(A\oplus A)$ as
  follows:
  \begin{equation}
    \label{eq:strgrp}
    \begin{pmatrix}
      0&r_{g_3(x_3)}\\ \ell_{g_3(x_3)}&0
    \end{pmatrix}
    =
    \begin{pmatrix}
      \mu(g_2)&0\\0& \mu(g_1)
    \end{pmatrix}
    \begin{pmatrix}
      g_1&0\\0&g_2
    \end{pmatrix}
    \begin{pmatrix}
      0&r_{x_3}\\ \ell_{x_3}&0
    \end{pmatrix}
    \begin{pmatrix}
      g_1^{-1}&0\\ 0&g_2^{-1}
    \end{pmatrix}.
  \end{equation}
  Since $C(\alpha'')$ is an isomorphism, there exists $\xi\in
  C_0(A,q)_R$ such that $C_0(\alpha'')(\xi)=(g_1,\,g_2)$. Then
  $C_0(\alpha'')\bigl(\underline\mu(\xi)\bigr) =
  \bigl(\mu(g_1),\,\mu(g_2)\bigr)$, and~\eqref{eq:strgrp} yields
  \[
    C(\alpha'')\bigl(g_3(x_3)\bigr) =
    C(\alpha'')(\iota(\underline\mu(\xi)) \xi x_3\xi^{-1})
    \qquad\text{for all $x_3\in A_R$.}
  \]
  Since $C(\alpha'')$ is an isomorphism, it follows from
  Lemma~\ref{lem:eqcondXClifgrp} that $\tau_0(\xi)x_3\xi =
  \sigma_b(g_3)(x_3)$ for all $x_3\in A_R$, hence
  $\xi\in\BOmega(q)(R)$ and $g_3=\chi_0(\xi)$. Thus, $(g_1,g_2,g_3) =
  \psi_\Calg(\xi)$. 
\end{proof}

Recall from Proposition~\ref{prop:chi0ontobis} the exact sequence
\begin{equation}
  \label{eq:strgrpeq}
  1\to R^1_{Z/F}(\BGm) \to\BOmega(q) \xrightarrow{\chi_0} \BGO^+(q)
  \to 1.
\end{equation}
Since the discriminant of $q$ is trivial, we have $Z\simeq F\times F$,
hence $R^1_{Z/F}(\BGm)\simeq\BGm$ and the Galois cohomology exact
sequence derived from~\eqref{eq:strgrpeq} takes the form
\[
  1\to F^\times \to \BOmega(q)(F) \to \GO^+(q)\to 1.
\]
Substituting $\BStr(A,\diamond)(F)$ for $\BOmega(q)(F)$, we recover
the exact sequence
obtained by Petersson~\cite[(4.13)]{P} for $\Calg$ an octonion
algebra.


\begin{thebibliography}{99}

\bibitem{Alb}
  A. A. Albert, Non-associative algebras. I. Fundamental concepts and
  isotopy, \emph{Ann. of Math.} (2) {\bf 43} (1942), 685--707.

\bibitem{AlsGille}
  S. Alsaody\ and\ P. Gille, Isotopes of octonion algebras,
  ${\bf{G}}_2$-torsors and triality, \emph{Adv. Math.} {\bf 343}
  (2019), 864--909.

\bibitem{Als}
  S. Alsaody, Albert algebras over rings and related torsors,
  \emph{Canad. J. Math.} {\bf 73} (2021), no.~3, 875--898. 

\bibitem{BGBT}
  K. Becher, N. Grenier-Boley\ and\ J.-P. Tignol, The discriminant
  Pfister form of an algebra with involution of capacity~$4$. 2020
  arXiv preprint 2008.08953

\bibitem{Bou}
  N. Bourbaki, {\it \'{E}l\'{e}ments de
    math\'{e}matique. Alg\`ebre. Chapitre 8. Modules et anneaux
    semi-simples}, second revised edition of the 1958 edition,
  Springer, Berlin, 2012.

\bibitem{CKT}
  V. Chernousov, M.-A. Knus\ and\ J.-P. Tignol, Conjugacy classes of
  trialitarian automorphisms and symmetric compositions,
  \emph{J. Ramanujan Math. Soc.} {\bf 27} (2012), no.~4, 479--508. 

\bibitem{DQM}
  A. Dolphin, A. Qu\'eguiner-Mathieu, The canonical quadratic pair on
  a Clifford algebra and triality, \emph{Israel J. Math.} {\bf 242}
  (2021), 171--213. 

\bibitem{Eld}
  A. Elduque, On triality and automorphisms and derivations of
  composition algebras, \emph{Linear Algebra Appl.} {\bf 314} (2000),
  no.~1-3, 49--74.

\bibitem{EKM}
  R. Elman, N. Karpenko\ and\ A. Merkurjev, {\it The algebraic and
    geometric theory of quadratic forms}, American Mathematical
  Society Colloquium Publications, 56, American Mathematical Society,
  Providence, RI, 2008.

\bibitem{Gar}
  S. Garibaldi, \emph{Cohomological invariants: exceptional groups and
    spin groups,} Mem. Amer. Math. Soc. {\bf 200} (2009), no.~937,
  xii+81 pp.  

\bibitem{Kap}
  I. Kaplansky, Infinite-dimensional quadratic forms admitting
  composition, \emph{Proc. Amer. Math. Soc.} {\bf 4} (1953), 956--960.

\bibitem{Knus}
  M.-A. Knus, {\it Quadratic and Hermitian forms over rings},
  Grundlehren der mathematischen Wissenschaften, 294, Springer-Verlag,
  Berlin, 1991. 

\bibitem{BoI}
  M.-A. Knus\ et al., {\it The book of involutions}, American
  Mathematical Society Colloquium Publications, 44, American
  Mathematical Society, Providence, RI, 1998.

\bibitem{KV}
  M.-A. Knus\ and\ O. Villa, Quadratic quaternion forms, involutions
  and triality, \emph{Doc. Math.} {\bf 2001}, Extra Vol., 201--218. 

\bibitem{P}
  H. P. Petersson, The structure group of an alternative algebra,
  \emph{Abh. Math. Sem. Univ. Hamburg} {\bf 72} (2002), 165--186. 

\bibitem{RST}
  M. Rost, J.-P. Serre\ and\ J.-P. Tignol, La forme trace d'une
  alg\`ebre simple centrale de degr\'{e} 4,
  \emph{C. R. Math. Acad. Sci. Paris} {\bf 342} (2006), no.~2, 83--87.

\bibitem{Sch}
  W. Scharlau, {\it Quadratic and Hermitian forms}, Grundlehren der
  Mathematischen Wissenschaften, 270, Springer-Verlag, Berlin, 1985.
  
\bibitem{Shapiro}
  D. B. Shapiro, {\it Compositions of quadratic forms}, De Gruyter
  Expositions in Mathematics, 33, Walter de Gruyter \& Co., Berlin,
  2000.

\bibitem{SpV}
  T. A. Springer\ and\ F. D. Veldkamp, {\it Octonions, Jordan algebras
    and exceptional groups}, Springer Monographs in Mathematics,
  Springer-Verlag, Berlin, 2000.

\bibitem{T}
  J.-P. Tignol, La forme seconde trace d'une alg\`ebre simple centrale
  de degr\'{e} 4 de caract\'{e}ristique 2,
  \emph{C. R. Math. Acad. Sci. Paris} {\bf 342} (2006), no.~2, 89--92. \end{thebibliography}
\end{document}